%% file: ttest.tex
\documentclass[11pt]{article} 

\usepackage[utf8]{inputenc} 

\usepackage{cancel}

\usepackage{geometry} 
\usepackage{graphicx} 
\usepackage[english]{babel}

\usepackage{booktabs} 
\usepackage{array} 
\usepackage{paralist} 
\usepackage{verbatim} 
\usepackage{subfig} 
\usepackage{amsfonts}
\usepackage{amssymb}
\usepackage{amsmath}
\usepackage{amsthm}
\usepackage{makecell}
\usepackage{tikz-cd}

\usepackage{epigraph}
\usepackage{bbm}

\usepackage{hyperref}
\hypersetup{
    colorlinks=true,
    citecolor=blue
}

\usepackage{pdflscape}
\usepackage{natbib}
\usepackage{wrapfig}
\usepackage{subfig}
\usepackage[capitalize,noabbrev]{cleveref}
\usepackage{authblk}
\usepackage{changepage}

\usepackage{xcolor}

\usepackage{tocbasic}
\DeclareTOCStyleEntry[
  beforeskip=.05em plus .6pt,
  pagenumberformat=\textbf
]{tocline}{section}

\input{macros}

\title{Anytime-valid t-tests and confidence sequences\\ 
for Gaussian means with unknown variance\footnote{to appear in \emph{Sequential Analysis}}}
\author[1]{Hongjian Wang}
\author[2]{Aaditya Ramdas}
\affil[1,2]{Department of Statistics and Data Science, Carnegie Mellon University}
\affil[2]{Machine Learning Department, Carnegie Mellon University} 
\affil[ ]{\texttt{ \{hjnwang,aramdas\}@cmu.edu  }}

\date{\today}

\newtheorem{theorem}{Theorem}[section]
\newtheorem{definition}[theorem]{Definition}

\newtheorem{proposition}[theorem]{Proposition}

\newtheorem{corollary}[theorem]{Corollary}
\newtheorem{lemma}[theorem]{Lemma}
\newtheorem{remark}[theorem]{Remark}

\begin{document}

\maketitle

\begin{abstract}
    \input{t-test/abstract}
\end{abstract}

{\small  \setcounter{tocdepth}{1}
    \hypersetup{linkcolor=black} \tableofcontents}

\input{t-test/body}





\subsubsection*{Acknowledgements}
\input{t-test/ack}

\bibliography{ttest}

\newpage
\appendix
\input{t-test/appendices}

\end{document}

%% file: macros.tex
\usepackage{bbm}

\renewcommand{\le}{\leqslant}
\renewcommand{\leq}{\leqslant}

\renewcommand{\ge}{\geqslant}

\newcommand{\id}{\mathbbmss 1}

\newcommand {\matl}{\left[ \begin{matrix}}
\newcommand {\matr}{\end{matrix}\right]}
\newcommand {\Exp}{ \mathbb E }
\newcommand {\EE}{ \mathrm E }

\renewcommand {\Pr}{ \mathbb P }

\newcommand{\CI}{\operatorname{CI}}

\newcommand{\e}{\mathrm e}
\newcommand{\kl}{D_{\operatorname{KL}}}

\renewcommand{\d}{\mathrm{d}}

\newcommand{\cA}{\mathcal{A}}

\newcommand{\cF}{\mathcal{F}}
\newcommand{\cG}{\mathcal{G}}
\newcommand{\cB}{\mathcal{B}}

\newcommand{\cL}{\mathcal{L}}

\newcommand{\cP}{\mathcal{P}}
\newcommand{\reals}{\mathbb{R}}

\newcommand{\Otilde}{\Tilde{\mathcal{O}}}

\newcommand{\hmu}{\widetilde{\mu}}

\newcommand{\iid}{\overset{\mathrm{iid}}{\sim}}

\newcommand{\hsig}{\widetilde{\sigma}}

\newcommand{\avgX}[1]{\widehat{\mu}_{#1}}
\newcommand{\avgXsq}[1]{\overline{X^2_{#1}}}

\newcommand{\normals}{\mathcal{N}}
\newcommand{\normal}[2]{N_{#1,#2}}

\DeclareMathAlphabet{\mathbbmsl}{U}{bbm}{m}{sl}

\DeclareMathOperator*{\Prw}{\Pr}

\DeclareMathOperator*{\polylog}{polylog}

\DeclareMathOperator*{\erf}{erf}

\newcommand{\ttfrac}[2]{{#1}/{#2}}

\newcommand{\revise}[1]{#1}

%% file: t-test/abstract.tex
In 1976, Lai constructed a nontrivial confidence sequence for the mean $\mu$ of a Gaussian distribution with unknown variance $\sigma^2$. Curiously, he employed both an improper (right Haar) mixture over $\sigma$ and an improper (flat) mixture over $\mu$. Here, we elaborate carefully on the details of his construction, which use generalized nonintegrable martingales and an extended Ville's inequality. While this does yield a sequential t-test, it does not yield an ``e-process'' (due to the nonintegrability of his martingale). In this paper, we develop two new e-processes and confidence sequences for the same setting: one is a test martingale in a reduced filtration, while the other is an e-process in the canonical data filtration. These are respectively obtained by swapping Lai's flat mixture for a Gaussian mixture, and swapping the right Haar mixture over $\sigma$ with the maximum likelihood estimate under the null, as done in universal inference. We also analyze the width of resulting confidence sequences, which have a curious polynomial dependence on the error probability $\alpha$ that we prove to be not only unavoidable, but (for universal inference) even better than the classical fixed-sample t-test.
    Numerical experiments are provided along the way to compare and contrast the various approaches, including some recent suboptimal ones.

%% file: t-test/body.tex
\section{Introduction}

The classical location tests for Gaussian data fall into either of the two categories: the Z-test, where the population variance is assumed to be known \emph{a priori}, and a sample average is rescaled by the sample deviation to create a standardized test statistic; or the t-test \citep{student1908probable}, where the population variance is unknown, and a plug-in estimator of the variance is used in lieu of the variance, accounting for the heavier (``regularly varying'') tail of the t-test statistic than the Gaussian Z-test statistic.

This paper concerns sequential analogs of t-tests. Sequential Z-tests have been widely established since \cite{wald1945sequential}. One-sided or ``power-one'' variants have even been generalized nonparametrically to
 \emph{subGaussian} data \citep{robbins1970statistical,howard2021time}, and recently to any data in the square-integrable class \citep{wang2022catoni}.
 

Sequential extensions of the t-test were studied by \cite{rushton1950sequential,hoel1954property,ghosh1960some,sacks1965note}, who were interested in calculating the \emph{approximate} thresholds to control the type 1 error at a particular predefined stopping time. Inspired by the sequential Z-test work of Robbins, \cite{lai1976confidence} constructed a sequential t-test that controls a different type of error that will be of particular interest to us: nonasymptotic, conservative (i.e.\ non-approximate) time-uniform control of the type 1 error at all \emph{arbitrary} stopping times, possibly not specified in advance. We will define this formally later, but we mention at the outset that the filtration \revise{will become an interesting matter as different sequential methods are defined on different filtrations.} Studying this phenomenon in the context of t-tests will be an important part of this paper. 


Lai and more recent authors \citep{grunwald2020safe,perez2022statistics} situate the Gaussian t-test problem into the broader framework of \emph{group-invariant} tests; indeed, the t-test null is invariant under rescaling of observations by a constant.

The tests are mostly constructed using variants of the sequential likelihood ratio statistic by \cite{wald1945sequential},
 utilizing the fact that the likelihood ratios are nonnegative martingales (see \cref{sec:lrm}).
  Notably, \cite{perez2022statistics} shows the group-invariant extension of the likelihood ratio statistic to be optimal (in the ``e-power" sense to be defined in \cref{sec:testproc}) \revise{primarily} in the point-versus-point ($\mu/\sigma = \theta_0$ versus $\mu/\sigma = \theta_1$) setting. \revise{We shall closely discuss their contributions later.}
  Building upon these group invariance tools, \cite{lindon2022anytime} construct a sequential F-test for the regression set-up, which coincides with one of the t-tests that we shall present when no covariate exists. A slightly different approach is taken by \cite{jourdan2023dealing} who study the same problem in the bandit context, sequentially bounding the empirical mean and variance separately using existing Z-test-type bounds and obtaining via a union bound a sequence of confidence ``boxes" for mean and variance simultaneously. We shall return to and compare our bounds with some of these results later.


We shall focus in this paper primarily on the construction of \emph{e-processes}, which are sequences of test statistics whose realized values reflect the evidence against the null in favor of the alternative. These e-processes can  be thresholded and/or inverted to yield nonasymptotic, conservative, ``anytime-valid'' hypothesis tests and confidence intervals (called ``confidence sequences"), \revise{but we emphasize that they have clear interpretations without thresholding them to yield tests}. Indeed, these e-processes directly measure evidence against the null: values below one are expected under the null, and large(r) values indicate (more) evidence against the null.

Nonnegative supermartingales are the prototypical examples of e-processes \citep{ramdas2022testing}. We emphasize that we are testing a composite null hypothesis, and it is not a priori obvious how to construct such ``simultaneous'' or ``composite'' nonnegative supermartingales or e-processes, or to compare various options. We shall formalize these concepts in greater detail in \cref{sec:seq-stat,sec:testproc}. The main contributions of this paper are listed as follows:
\begin{enumerate}
    \item We apply the method of ``Universal Inference", a variant of the likelihood ratio test by \cite{wasserman2020universal}, to construct a t-test e-process for the point null\footnote{\revise{In this paper, a ``point null'' refers to a null hypothesis that a parameter of interest takes a specific value. Since it possibly contains multiple distributions, this is different from a ``simple null''.}} $\mu = \mu_0$, as well as for the one-sided composite null $\mu \le \mu_0$, implying respectively a two-sided and a one-sided confidence sequence for the mean $\mu$, both of which are new.

    \item We fill in the missing details of \citeauthor{lai1976confidence}'s \citeyearpar{lai1976confidence} sequential t-test, looking into the curious question of why his approach yields a confidence sequence without an e-process, using the theory of non-\revise{integrable} nonnegative supermartingales \citep{ensm}.

    \item We introduce a simple variant of  \cite{lai1976confidence}, deriving instead closed-form e-processes for the null $\mu = \mu_0$. The resulting confidence sequence enjoys improved tightness, and unlike \citeauthor{lai1976confidence}'s original, does not require numerical inversion or thresholding.

    \item We derive the information-theoretic lower bound on the width of confidence intervals, and the upper bound on the growth of e-processes for the t-test, both of which are attained, or almost attained, by our results (Contributions 1 and 3 in this list). We also show that other conservative sequential t-tests available in the literature (\citeauthor{lai1976confidence}'s original method, the plug-in method by \cite{jourdan2023dealing}, the median method by \cite{quantile}), and very surprisingly, \emph{even the fixed-time classical t-test}, fail to be optimal \revise{in the rate with respect to the sample size $n$ and/or the type 1 error level $\alpha$.} 

     {\item We conduct a detailed comparison among tests we derive and in the literature in terms of (a) which filtration they are safe for optional stopping in (the larger the better); (b) what their e-powers are against a fixed alternative (the larger the better); and (c) for their derived confidence sequences, the width or expected width, especially as a function of sample size $n$ and level $\alpha$.} 
    
\end{enumerate}



\section{Preliminaries}\label{sec:prelims}

\subsection{Notations}

A universal probability space $(\Omega, \mathcal{A}, \Pr)$ 
is used when we consider the randomness of data, \revise{where the expectation is denoted by $\Exp$}. We use the indexed letters $X_1, X_2, \dots$ to denote an infinite stream of observations, taking values in $\mathbb R$.
We denote by $\{ \cF_n \}_{n \ge 0}$ the canonical filtration generated by the data, i.e.\ $\cF_n  = \sigma(X_1, \dots, X_n)$.
Our shorthand notations on the sample include the partial sum $S_n = \sum_{i=1}^n X_i$, the sample average $\avgX{n}= S_n/n$, the partial sum of squares $V_n = \sum_{i=1}^n X_i^2$, the average square $\avgXsq{n} = V_n/n$, the sample variance without Bessel's correction $s_{n}^2 = V_n/n - S_n^2/n^2$, and the t-statistic $T_{n-1}=\sqrt{n-1} \frac{S_n - n\mu_0}{ \sqrt {n V_n - S_n^2} }$ where $\mu_0$ is the actual mean of $X_1, X_2,\dots$.

Usually, we consider i.i.d.\ observations $X_1, X_2, \dots$; and we use italic $P$ etc.\ to denote distributions over $\mathbb R$, while calligraphic $\mathcal{P}$ etc.\ to denote classes of distributions. 
\revise{For a sample-dependent event $A$, we use the notation $P(A)$ for the probability $\Pr[A]$ when $X_1, X_2, \dots \iid P$; and for a random variable realized as a function of the sample $Z = Z(X_1,\dots, X_n)$, we write $\EE_P(Z)$ for the expectation $\Exp[Z]$ when $X_1, X_2, \dots \iid P$.}
The conditional expectation $\EE_P( \, \cdot \, | \cF)$ given $\cF \subseteq \cA$ is defined similarly. 
Univariate normal distribution with mean $\mu$ and variance $\sigma^2$ will be referred to
as $\normal{\mu}{\sigma^2}$ -- so ``$\normal{0}{1}(X_1 > 0) = 0.5$'' means ``$\Pr[X_1 > 0] = 0.5$ when $X_1, X_2 , \dots \iid$ standard normal''; its density is denoted by $p_{\mu, \sigma^2}$. The class of all normal distributions on $\mathbb R$  (excluding the degenerate $\sigma = 0$ case) is denoted by $\normals$. Classes of normal distributions with specified mean (or range of means) but any positive variance are denoted by, e.g., $\normals_{\mu = \mu_0}$, $\normals_{\mu \le \mu_0}$; those with specified standardized mean by, e.g., $\normals_{\mu/\sigma = \theta_0}$.

Occasionally, we allow the observations $X_1, X_2, \dots$ to be non-i.i.d., in which case we use blackboard bold italic letters like $\mathbbmsl P$ to denote distributions on $\mathbb R \times \mathbb R \times \dots$, i.e.\ the distributions of the entire stochastic processes. The notations $\mathbbmsl P(\, \cdot \, )$, $\EE_{\mathbbmsl P} (\, \cdot \, )$, and $\EE_{\mathbbmsl P} (\, \cdot \, | \cF)$ are used for the probabilities and expectations when $X_1, X_2,\dots \sim \mathbbmsl P$.

\subsection{Sequential Statistics}
\label{sec:seq-stat}

\revise{The theory of sequential hypothesis testing} is formulated as testing the null $H_0 : P \in \mathcal{P}_0$ against the alternative $H_1: P \in \mathcal{P}_1$, \revise{where the statistician} makes a decision (to reject, \revise{accept, or sample further}) each time a new data point $X_n \sim P$ is seen until rejection \revise{or acceptance}. \revise{In the formulation of \cite{wald1945sequential,wald1947sequential}, the decision scheme is required to control
the type 1 error rate $P(H_0 \text{ is ever rejected})$ for $P \in \mathcal{P}_0$ within a prescribed level $\alpha$, and the type 2 error rate $P(H_0 \text{ is ever accepted})$ for $P \in \mathcal{P}_1$ within $\beta$. A slightly different formulation, proposed by \cite{darling1968some}, only requires the type 1 error rate to be controlled within $\alpha$, but the decision schemes under this constraint are often proved to have ``power one'', i.e.,\ $P(H_0 \text{ is ever rejected}) = 1$ for $P \in \mathcal{P}_1$. We shall follow the latter formulation. We derive sequential tests of type 1 error rate strictly bounded by $\alpha$ (``level-$\alpha$ sequential tests'') and prove their power guarantees. These power guarantees are stated in terms of ``e-power'', to be formally introduced in \cref{sec:epwr}, which is a finer-grained quantification of power than the rejection rate $P(H_0 \text{ is ever rejected})$ under $P \in \cP_1$.}


The classical duality between tests and interval estimates also manifests in the sequential setting. Let $\mathcal{P}$ be a family of distributions and let $\theta : \mathcal{P} \to \mathbb R$ be the parameter of interest. A $(1-\alpha)$-\emph{confidence sequence} (CS) \citep{darling1967confidence} over $\mathcal{P}$ for $\theta$ is a sequence of confidence intervals $\{ \CI_n \}_{n \ge 1}$ such that, for any $P \in \mathcal{P}$, $P( \forall n \ge 1, \; \theta(P) \in \CI_n )\ge 1-\alpha$; or equivalently, $P( \theta(P) \in \CI_\tau ) \ge 1-\alpha$ for any stopping time $\tau$ \revise{(see Lemma 3 of  \cite{howard2021time})}. If the null $\mathcal{P}_0 \subseteq \mathcal{P}$ is defined as a preimage $\theta^{-1}(\theta_0)$, a sequential test can be constructed via rejection $H_0$ whenever $\theta_0 \notin \CI_n$, and vice versa. In this paper, the asymptotic tightness of a 
confidence sequence is measured in terms of (1) its \emph{growth rate}, the rate at which its width increases when $n$ and the sample are fixed while $\alpha \to 0$ (or while $\alpha \downarrow \alpha_{\star}(n)$ if $\alpha$ is not allowed to be below $\alpha_{\star}(n)$); and (2) its \emph{shrinkage rate}, the rate at which its width decreases when $\alpha$ is fixed while $n \to \infty$). \revise{For example, it is known that the minimax rates of growth $\sqrt{\log(1/\alpha)}$ and shrinkage $\sqrt{\log \log(n)/n}$ can be obtained in CSs for the mean over all 1-subGaussian random variables \citep{howard2021time,waudby2020estimating}, and also all random variables with variance \revise{upper bounded by some constant $\sigma^2$} \citep{wang2022catoni}. As we shall see soon, the optimal rates for the t-test are larger.}

\subsection{Test Processes and \revise{Their Maximal Inequalities}}
\label{sec:testproc}

\revise{There has been a recent revival of the idea, originally due to \cite{robbins1970statistical}, of constructing} sequential tests and confidence sequences 
via nonnegative supermartingales (NSMs) and more generally, e-processes
 \citep{howard2021time,ramdas2022testing}. Let us \revise{define them below.} 


Consider a stochastic process $\{ M_n \}_{n \ge 0}$ where each $M_n$ is defined as a $[0,\infty]$-valued function of $X_1, \dots, X_n$. We say it is
\begin{itemize}
    \item a \emph{nonnegative supermartingale} (NSM) for $P$  on a filtration $\{ \cG_n \}_{n\ge 0}$ if $\EE_P(M_0) < \infty$, $M_n$ is $\cG_n$-measurable and $\EE_P( M_{n+1} |  \cG_{n}) \le M_{n}$ for all $n\ge 0$;
    \item a \emph{nonnegative martingale} (NM) for $P$ if, further, the equality $\EE_P( M_{n+1} |  \cG_{n}) =  M_{n}$ holds. 
\end{itemize}

Recently, \cite{ensm} defined a wider class of processes by dropping the integrability assumption $\EE_P(M_0) < \infty$, as the conditional expectation $\EE_P(M_{n+1} | \cG_{n})$ can still be well-defined as long as $M_{n+1}$ is nonnegative. We say the process $\{ M_n \}$ is

\begin{itemize}
    \item an \emph{extended nonnegative supermartingale} (ENSM) for $P$ on $\{ \cG_n \}$ if $M_n$ is $\cG_n$-measurable and $\EE_P(M_{n+1} | \cG_{n}) \le M_{n}$ are satisfied, irrespective of the finiteness of $\EE_P(M_0)$.
\end{itemize}
\revise{We say the process $\{M_n\} $ is NSM (or NM, ENSM) for $\cP$ if it is an NSM (or NM, ENSM) for all $P \in \cP$. When an NSM (or NM) for $\cP$ has an initial expected value $\EE_P(M_0) = 1$ for all $P\in\cP$, we call it a test supermartingale (or test martingale) for $\cP$. We call a nonnegative random variable $M = M(X_1, \dots, X_n)$ with $\sup_{P \in \cP} \EE_P(M) \le 1$ an \emph{e-value} for $\cP$.} Further, we say the process $\{ M_n \}$ is
\begin{itemize}
    \revise{\item an \emph{e-process} for $\cP$ on $\{ \cG_n \}$ if for any $\{ \cG_n \}$-adapted stopping time $\tau$, $M_\tau$ is an e-value for $\cP$; i.e., $\sup_{P \in \cP}\EE_P( M_\tau ) \le 1$.}
\end{itemize}
  Note that the filtration $\{ \cG_n \}$ is usually the canonical one $\{ \cF_n \}$, but as we shall see later, other non-trivial choices will lead to interesting results. \revise{NMs and NSMs for $\cP$ are clearly e-processes for $\cP$, on the same filtration. More generally, the following equivalent characterization of e-processes is due to \citet[Lemma 6]{ramdas2020admissible}: 
\begin{itemize}
    \item A nonnegative process $\{M_n\}$ adapted to $\{\cG_n\}$ is an e-process for \revise{$\cP$} on $\{ \cG_n \}$ if and only if for each $P \in \cP$, there is a test supermartingale $\{ 
N_n^P \}$ for $P$ on $\{ \cG_n \}$  that upper bounds $\{M_n\}$: $P(M_n \le N_n^P) = 1$.
\end{itemize}

}

A parametrized family of NSMs or ENSMs can lead to a new NSM or ENSM by \emph{mixture}. If $\{ 
M_n(\theta) \}$ is an NSM for $\cP$ for any $\theta \in \Theta$, so is the mixture $\{ 
\int M_n(\theta) \mu(\d \theta) \}$ for any finite measure $\mu$ over $\Theta$, under mild measurability assumptions; if each is an ENSM for $\cP$, the mixture is still an ENSM for $\cP$ as long as $\mu$ is $\sigma$-finite \citep[Section 5.2]{ensm}. Similar results hold straightforwardly for the mixtures of e-processes as well.

An important property of these processes is Ville's inequality \citep{ville1939etude} and the recently discovered \emph{extended} Ville's inequality \citep[Theorem 4.1]{ensm}. These maximal inequality bounds the crossing probability of a NSM (and hence e-process) or ENSM over an unbounded time horizon.
\begin{lemma}[Ville's inequality]\label{lem:ville}
If $\{ M_n \}_{n \ge 0}$ is an NSM for $\cP$, for all $P \in \cP$, $\varepsilon \in (0, 1)$,
\begin{equation}
    P( \exists n \ge 0 , \ M_n \ge \varepsilon  ) \le \varepsilon^{-1}\EE_P(M_0).
\end{equation}
Consequently, if $\{ M_n \}_{n \ge 0}$ is an e-processes for $\cP$, for all $P \in \cP$, $\varepsilon \in (0, 1)$,
\begin{equation}
    P( \exists n \ge 0 , \ M_n \ge \varepsilon  ) \le \varepsilon^{-1}.
\end{equation}
\end{lemma}

\begin{lemma}[Extended Ville's inequality]\label{lem:evi}
    Let $\{ M_n \}_{n \ge 0}$ be an ENSM for $\cP$. Then for all $P \in \cP$, $\varepsilon \in (0, 1)$,
\begin{equation}\label{eqn:evi}
    P( \exists n \ge 0 , \ M_n \ge \varepsilon  ) \le \varepsilon^{-1}\EE_P( \id_{ \{ M_0 < \varepsilon \} } M_0) + P( M_0 \ge \varepsilon ).
\end{equation}
\end{lemma}

Lemmas~\ref{lem:ville} and~\ref{lem:evi} suggest that, if $\cP_0$ is a set of null distributions, NSMs, e-processes, and ENSMs for $\cP_0$ can be seen as \emph{measurements of evidence} against the null $H_0: P \in \mathcal{P}_0$. To see that, if we have at hand an NSM or e-process $\{  M_n \}$ for $\cP_0$ such that $M_0 = 1$, we may reject the null $\cP_0$ whenever the process $M_n$ exceeds $1/\alpha$. This, due to Lemma~\ref{lem:ville}, controls the sequential type 1 error $P( H_0 \text{ is ever rejected}  )$ within $\alpha$ for any $P \in \cP_0$. Similarly, if we have an ENSM $\{M_n\}$ for $\cP_0$, we may set the right hand side of \eqref{eqn:evi} to $\alpha$ and solve the corresponding $\varepsilon$, rejecting $H_0: P \in \cP_0$ whenver $M_n$ exceeds $\varepsilon$.  We refer to test (super)martingales, e-processes, and ENSMs for $\cP_{0}$ as \emph{test processes} for the null $\cP_{0}$.

A confidence sequence for a parameter $\theta$ follows in the same manner as mentioned in \cref{sec:seq-stat} if we have an NSM or an ENSM $\{ M^{\theta_0}_n \}$ for every $\cP_{\theta_0} = \{ P \in \cP : \theta(P) = \theta_0 \}$: by solving the inequality $M_n^{\theta_0} \le 1/\alpha$ or $M_n^{\theta_0} \le \varepsilon$ in $\theta_0$, that is. An illustrative comparison between NSMs and ENSMs in the Gaussian case with known variance for both constructing tests and confidence sequences is provided recently by \citet[Section 5]{ensm}.

Of course, all above is applicable to the non-i.i.d.\ case as well. NMs  etc.\ for $\mathbbmsl P$ are defined in similar manners. NMs, NSMs, ENSMs, and e-processes for $P$ (or $\cP$, $\mathbbmsl P$) are all referred to as {test processes} for $P$ (or $\cP$, $\mathbbmsl P$).

\subsection{e-Power of Test Processes}
\label{sec:epwr}
Desirably, a powerful {test process} shrinks or stays small under the null $\cP_0$, and grows under the alternative $\cP_1$ so rejection can happen.  To measure the power of a test process under a fixed alternative distribution $Q \in \cP_1$, we define the \emph{e-power} of \revise{test processes as follows \citep{vovk2022efficiency}.

\begin{definition}[e-power]\label{def:epwr}
    For a test process $\{ M_n \}$, we define its e-power as the $Q$-almost sure limit $\liminf_{n\to\infty} \frac{\log M_n}{n}$. For a single e-value $M = M(X_1, \dots, X_n)$, we define its e-power as $\EE_Q (\log M)$.
\end{definition}
The motivation, we note, is that it quantifies the rate at which the process grows exponentially under the alternative, i.e., how fast evidence accumulates. The value of an e-process is often intuitively understood as one's total capital when betting against a casino pricing according to the null (hence the term ``game-theoretic statistics'', see e.g.\ \cite{shafer2005probability}).
The idea of quantifying and optimizing the growth of a capital process by its expected log-value can be traced back to \cite{kelly1956new,breiman1961optimal}, among other classical references in the literature. \cite{breiman1961optimal}, for example, justifies this log-optimality criterion by noting that ``under reasonable betting systems $M_n$ increases exponentially and maximizing $\EE_Q(\log M_n)$ maximizes the rate of growth.''

More concretely, many e-processes in the literature take a multiplicative form $M_n = \prod_{i=1}^n M(X_i)$ where $M(X_i)$'s are i.i.d.\ e-values. In this case, the e-power of the entire process $\{M_n\}$ and that of the individual e-values $M(X_i)$ coincide due to the strong law of large numbers, and the growth is characterized asymptotically as
$M_n = \exp(  n \EE_Q( 
\log M_1 ) + o_Q(n) )$. Alternatively, one can also informally say that the e-power $\EE_Q(\log M_1)$ is roughly inversely proportional to the expectation of the rejection time $\tau = \min \{ n : M_n \ge 1/\alpha  \}$ as
\begin{equation}
  \log (1/\alpha) \approx \EE_Q ( \log M_{\tau}  ) = \EE_Q(\log M_1) \cdot \EE_Q (\tau),
\end{equation}
where the equality is due to Wald's identity (see e.g.\ \citet[Theorem 2.1]{roters1994validity}). More formally, lower bounds are known for the expected rejection time of simple testing problems, attained exactly by the inverse of e-power of the likelihood ratio martingale, a topic we shall return to in \cref{sec:lrm}.
}



\subsection{Sequential t-Test and t-Confidence Sequences}



Different objectives are pursued for the sequential Gaussian t-test problem. Early authors have studied the problem of testing the null $\normals_{\mu = 0}$ (or more generally $\normals_{\mu/\sigma = \theta_0}$) against the alternative $\normals_{\mu/\sigma = \theta_1}$ \citep{rushton1950sequential,ghosh1960some,sacks1965note}, and its scale-invariant nature translates into the group-invariant setting as the problem of testing orbital nulls and alternatives  \citep{perez2022statistics}.

In this paper, we shall roughly follow the same routine as \cite{lai1976confidence} and aim at power against \revise{a broader range of} alternatives, chiefly focusing instead on the construction of (1) test processes for the null $\normals_{\mu=0}$ (and study their behavior under both null and non-null underlying distributions); (2) Confidence sequences for the population mean $\mu$ over the set of all Gaussian distributions $\normals$, which we shall often shorthand as ``t-confidence sequences", or ``t-CSs" for brevity.
One may note that (1) is sufficient to produce tests for arbitrary point location null $\normals_{\mu = \mu_0}$ by shifting, and consequently (2) by Lemmas~\ref{lem:ville} and~\ref{lem:evi} and inversion. 
Occasionally we are interested in test processes for the one-sided null $\normals_{\mu \le 0}$ as these give rise to one-sided tests and CSs. In fact, we shall later construct test processes for $\normals_{\mu/\sigma = \theta_0}$ which is the strongest statistical quantity in the setting, as it leads to both CSs and tests for the mean $\mu$ and the normalized mean $\mu/\sigma$.

\if
See \cref{fig:cd-various-t-test-tools} below.

\begin{figure}[!h]
    \centering
   \begin{tikzcd}[row sep=huge, column sep=8em]
\text{TP for $\normals_{\mu/\sigma = \theta_0}$} \arrow[d, "\text{take }\theta_0 = 0"] \arrow[r, "\text{VI or EVI}"] & \text{CS over $\normals$ for $\mu/\sigma$} & \\
\text{TP for $\normals_{\mu = 0}$} \arrow[d, "\text{replace each $X_i$ with $X_i - \mu_0$}"]  &  & \text{TP for $\normals_{\mu \le 0}$} \arrow[d, "\text{replace each $X_i$ with $X_i - \mu_0$}"]  \\
\text{TP for $\normals_{\mu = \mu_0}$ } \arrow[r, "\text{VI or EVI}"] & \text{CS over $\normals$ for $\mu$}  & \text{TP for $\normals_{\mu \le \mu_0}$ }  \arrow[l, "\text{VI or EVI}"] 
\end{tikzcd}
    \caption{The relationship between various t-tests.}
    \label{fig:cd-various-t-test-tools}
\end{figure}
\fi

\subsection{Likelihood Ratio Martingales}
\label{sec:lrm}
It is well known that the likelihood ratio process is a nonnegative martingale.
\begin{lemma} \label{lem:lrm} Let $Q \ll P$ be probability distributions on $\mathbb R$.
Then, the process
\begin{equation}
    L_n = \prod_{i=1}^n \frac{\d Q}{\d P}(X_i)
\end{equation}
is a nonnegative martingale for $P$ on the canonical filtration $\{\mathcal{F}_n\}_{n \ge 0}$.
\end{lemma}
\revise{Recalling Definition~\ref{def:epwr}, the e-power of $\{L_n\}$ equals the Kullback-Leibler divergence $\kl(Q\| P)$. The growth of $L_n = \exp( n \kl(Q\| P) + o_Q(n) )$ is optimal in this point versus point case, matching the  ${\Theta}\left( \log(1/\alpha) \kl(Q\| P)^{-1} \right)$ lower bound in the expected rejection time with type 1 error rate $\alpha$. See e.g.\ \citet[Section 2]{garivier2019explore} and \citet[Lemma 1]{kaufmann2016complexity} for this lower bound in the equivalent bandit formulation.}

We shall use two stronger forms of Lemma~\ref{lem:lrm} stated below and proved in \cref{sec:pf-lr}. First, the distribution on the numerator, $Q$, can be varying and depend on previous observations.

\begin{lemma} \label{lem:lrm-general} Let $P$ be a probability distribution on $\mathbb R$. For each $n \ge 1$ let $Q_n: \Omega \times \mathcal{B}(
\mathbb R) \to \mathbb R$ be a Markov kernel such that: 1) $Q_n(\, \cdot \,, B)$ is $\mathcal{F}_{n-1}$-measurable for all $B \in \mathcal{B}(\mathbb R)$, and 2) $Q_n(\omega, \, \cdot \,) \ll P$ for \revise{$P$-almost} all $\omega \in \Omega$.
Then, the process
\begin{equation}
    L_n = \prod_{i=1}^n \frac{\d Q_i}{\d P}(X_i) 
\end{equation}
is a nonnegative martingale for $P$ on the canonical filtration $\{\mathcal{F}_n\}_{n \ge 0}$.
\end{lemma}

Second, the data does \emph{not} have to be i.i.d.\ \citep[Section 2]{lai1976confidence}. Recall that if a stochastic process $X_1,X_2,\dots$ has distribution $\mathbbmsl P$, the joint distribution of $(X_1,\dots, X_n)$ is the push-forward measure of $\mathbbmsl P$ under the coordinate map $(x_1,\dots) \mapsto (x_1, \dots, x_n)$, which we denote by $\mathbbmsl P_{(n)}$.

\begin{lemma}\label{lem:lrm-joint}
    Let $\mathbbmsl P$ and $\mathbbmsl Q$ be probability distributions on $\mathbb R \times \mathbb R \times \dots$ such that $\mathbbmsl Q_{(n)} \ll \mathbbmsl P_{(n)}$ for all $n$.
    Then, the process
    \begin{equation}
     L_0 = 1, \quad   L_n = \frac{\d \mathbbmsl Q_{(n)}}{\d \mathbbmsl P_{(n)}}(X_1, \dots, X_n) 
    \end{equation}
    is a nonnegative martingale for $\mathbbmsl P$ on the canonical filtration $\{\mathcal{F}_n\}_{n \ge 0}$.
\end{lemma}

\section{t-Test e-Processes via Universal Inference}

 We shall demonstrate how the method of universal inference \citep{wasserman2020universal} leads to e-processes and confidence sequences for the t-test. Throughout this section, for $n \ge 0$, we denote by $\hmu_n$ and $\hsig_n$ any point estimators for $\mu$ and $\sigma$ adapted to the canonical filtration $\mathcal{F}_n$ (i.e.\ based on $X_1,\dots, X_n$). For example, they can simply be the sample mean $\avgX{n}$ and sample standard deviation $s_n$;
 we may also use the Bayesian approach by sequentially updating them as the posterior mean or maximum a posteriori from, say, a normal-inverse-gamma prior\footnote{\revise{The normal-inverse-gamma distribution serves as the conjugate prior for the Gaussian family when both of $\mu$ and $\sigma$ are unknown. See \cref{sec:sim-ep} for more description.}} on $(\mu, \sigma^2)$. The method of sequential universal inference \citep[Section 8] {wasserman2020universal} implies the following plug-in likelihood ratio martingale for a fixed Gaussian distribution:

 \begin{corollary}[Universal inference Z-test martingale]\label{cor:plugin-lr}
     For any $\mu$, $\sigma$ and their point estimators $\{ \hmu_n \}$, $\{ \hsig_n \}$ adapted to the canonical filtration $\{\mathcal{F}_n\}_{n \ge 0}$ (such as the sample mean and standard deviation), the process
\begin{equation}\label{eqn:plugin-lr}
     \ell_n^{\mu, \sigma^2} = \frac{\prod_{i=1}^n  p_{\hmu_{i-1}, \hsig^2_{i-1}}(X_i)}{\prod_{i=1}^n  p_{\mu, \sigma^2}(X_i)} = \frac{\sigma^n}{\prod_{i=1}^n\hsig_{i-1}} \exp \left\{ \sum_{i=1}^n \left( \frac{(X_i - \mu)^2}{2\sigma^2} - \frac{(X_i - \hmu_{i-1})^2}{2\hsig^2_{i-1}} \right)  \right\}
\end{equation}
    is a martingale for $\normal{\mu}{\sigma^2}$ on $\{\mathcal{F}_n\}_{n \ge 0}$.
 \end{corollary}
 It is not hard to see that Corollary~\ref{cor:plugin-lr} follows from Lemma~\ref{lem:lrm-general}, so we omit its proof. It is so named as the process $\{ \ell_n^{\mu,\sigma} \}$ itself can be used, when $\sigma^2$ is known, as a sequential Z-test.
 From martingales for $\normal{\mu}{\sigma^2}$ \emph{for each $\mu$ and $\sigma$} to test processes for the t-test nulls, $\normals_{\mu = 0}$ and $\normals_{\mu \le 0}$, we can simply take the infima of \eqref{eqn:plugin-lr} over $\{ \mu = 0, \sigma > 0 \}$ and $\{ \mu \le 0, \sigma > 0 \}$, which leads to e-processes for these nulls. They both have a closed-form expression and are 
 stated below as Theorem~\ref{thm:ui-ttest} and Theorem~\ref{thm:ui-ttest-onesided}, both proved in \cref{sec:pf-ui}.

\begin{theorem}[Universal inference t-test e-process]\label{thm:ui-ttest} For any point estimators $\{ \hmu_n \}$ and $\{ \hsig_n \}$ adapted to the canonical filtration $\{\mathcal{F}_n\}_{n \ge 0}$, the process
\begin{equation}\label{eqn:e-proc-ui-point}
    R_n
     =\left(  \avgXsq{n} \right)^{n/2}  \e^{n/2} \prod_{i=1}^n \frac{1}{\hsig_{i-1}} \exp \left\{ -\frac{1}{2}\left( \frac{X_i - \hmu_{i-1}}{\hsig_{i-1}} \right)^2  \right\}
\end{equation}
is an e-process for $\normals_{\mu = 0}$ on $\{ \cF_n \}$. Consequently, define
\begin{equation}\label{eqn:tn}
    W_n =  \frac{1}{\alpha^{2/n} \e}  \exp \left\{ \frac{\sum_{i=1}^n \log \hsig_{i-1}^2 + \left( \frac{X_i - \hmu_{i-1}}{\hsig_{i-1}} \right)^2 }{n} \right\}.
\end{equation}
Then $  \CI_n = \left[  \avgX{n}  \pm  \sqrt{ \avgX{n}^2 - \avgXsq{n} + W_n } \right]$ forms a $(1-\alpha)$-CS for $\mu$ over $\normals$.
\end{theorem}
It is worth remarking that for the general location null $\normals_{\mu = \mu_0}$, one must resist the temptation to replace all $X_i$ in \eqref{eqn:e-proc-ui-point} above by $X_i - \mu_0$, which, while indeed produces an e-process for $\normals_{\mu = \mu_0}$, loses its power. The correct modification is to \emph{only} shift $X_i$ in the $ \avgXsq{n}$ term, i.e.\ using the e-process for $\normals_{\mu = \mu_0}$
\begin{equation}
     R_n^{\mu_0}
     =\left(\frac{\sum_{i=1}^n (X_i - \mu_0)^2 }{ n} \right)^{n/2}  \e^{n/2} \prod_{i=1}^n \frac{1}{\hsig_{i-1}} \exp \left\{ -\frac{1}{2}\left( \frac{X_i - \hmu_{i-1}}{\hsig_{i-1}} \right)^2  \right\}.
\end{equation}
This can be understood as shifting $X_i$ and also $\hmu_{i-1}$ by $\mu_0$. This is how we arrive at the confidence sequence $  \CI_n = \left[  \avgX{n}  \pm  \sqrt{ \avgX{n}^2 - \avgXsq{n} + W_n } \right]$ above.

As a test process, the e-process $\{ R_n \}$ distinguishes the null $\normals_{\mu = 0}$ and the alternative $\normals_{\mu \neq 0}$ due to the following asymptotic result on a  {limit which, we recall, is defined in \cref{sec:testproc} as the ``e-power" of the test process.}
\begin{proposition}[e-power of the universal inference t-test e-process]\label{prop:div-ui-eproc}
    \revise{Under any distribution $P$ with mean $\mu$ and variance $\sigma^2$, for example $\normal{\mu}{\sigma^2}$,}
    suppose there is a $\gamma > 0$ such that $\{ \hmu_n \}$ converges to $\mu$ in $L^3$ with rate $\EE_P (|\hmu_n  - \mu|^3) \lesssim n^{-\gamma} $, $\{ \hsig_n^{-2} \}$ converges to $\sigma^{-2}$ both in $L^2$ with rate $\EE_P (\hsig_n^{-2}  - \sigma^{-2})^2 \lesssim n^{-\gamma} $ and almost surely, and has uniformly bounded 3\textsuperscript{rd} moment $\limsup_n \EE_P (\hsig_n^{-6}) < \infty $. Then,
    \begin{equation}
        \lim_{n \to \infty} \frac{\log R_n}{n} = \frac{1}{2} \log(1 + \mu^2/\sigma^2) \quad \text{almost surely}. 
    \end{equation}
    Consequently, $\{ R_n \}$ diverges almost surely to $R_\infty = \infty$ exponentially fast under $\normals_{\mu \neq 0}$.
\end{proposition}

\revise{Note that we place no Gaussian assumption on the true alternative distribution $P$ for this e-power statement, only that the \emph{plug-in point estimators} need to satisfy some convergence assumptions. These assumptions, we note,}
are mild, as they are satisfied by the empirical mean and variance \revise{under $\normal{\mu}{\sigma^2}$} due to the moment properties of the inverse-$\chi^2$ distribution, as well as their smoothed or Bayesian extensions (posterior means or maximum a posteriori under reasonable priors). The exponential growth of $\{ R_n \}$ under the alternative $\normals_{\mu \neq 0}$ is in contrast to its restrained behavior under the null $\normals_{\mu = 0}$, characterized by Ville's inequality for e-processes (Lemma~\ref{lem:ville}).

The $\frac{1}{2} \log(1 + \mu^2/\sigma^2)$ e-power in Proposition~\ref{prop:div-ui-eproc}, we remark, is universal among test processes for t-tests. 
If one uses the plug-in likelihood ratio $\{ \ell_n^{0,\sigma^2} \}$ in Corollary~\ref{cor:plugin-lr} to conduct the Z-test for the null $\normal{0}{\sigma^2}$, then, under the actual distribution $\normal{\mu}{\sigma^2}$, it is not hard to see that a faster rate $\frac{\log  \ell_n^{0,\sigma^2}}{n} \to \mu^2/2\sigma^2$ holds under similar assumptions of the point estimators (which we show in \cref{sec:uiz}).
We shall see in the rest of the paper many more occurrences of the same $\frac{1}{2} \log(1 + \mu^2/\sigma^2)$ limit, and we shall discuss that this is indeed the optimal rate in \cref{sec:ripr}.

Apart from the asymptotics of the test process, we can also
analyze the confidence sequence it implies. Let us briefly study the asymptotics of the radius $ \sqrt{ \avgX{n}^2 - \avgXsq{n} + W_n } = \sqrt{ W_n - \left(\avgXsq{n} - \avgX{n}^2\right) }$ of the CS in Theorem~\ref{thm:ui-ttest}, as $\alpha \to 0$, $\sigma \to \infty$, and $n \to \infty$. Note that both $W_n$ and $\avgXsq{n} - \avgX{n}^2$ converge to $\sigma^2$, as long as the estimators $\tilde{\sigma}_i$ and $\tilde{\mu}_i$ are consistent. From \eqref{eqn:tn}, the deviation of $T_n$ from $\sigma^2$ is approximately $\sigma^2(\alpha^{-2/n}  \exp(1/\sqrt{n}) - 1)$; While by the $\chi^2$-tail bound and the Chernoff bound,
\begin{equation}
    \Pr \left[ \left| \frac{n}{n-1}\left(\avgXsq{n} - {\avgX{n}}^2 \right) - \sigma^2 \right|\ge \sigma^2 t \right] \lesssim \e^{-nt^2},
\end{equation}
so the deviation of $\avgXsq{n} - {\avgX{n}}^2$ from $\sigma^2$ is approximately $\sigma^2/n$.
Hence the radius of the CS scales, as a function of $\alpha$, $\sigma$ and $n$, at the rate of $\sigma \alpha^{-1/n} n^{-1/2}$. We remark that the ``growth rate'' of the CS, i.e., its dependence on $\alpha \to 0$, is in the time-dependent form of $\alpha^{-1/n}$ which is worse than the typical known-variance rate $\sqrt{\log(1/\alpha)}$. This, as we shall explain in \cref{sec:opt-CI}, is also unavoidable. 




Let us now state the e-process for the one-sided null.

\begin{theorem}[Universal inference one-sided t-test e-process]\label{thm:ui-ttest-onesided} For any point estimators $\{ \hmu_n \}$ and $\{ \hsig_n \}$ adapted to the canonical filtration $\{\mathcal{F}_n\}_{n \ge 0}$, the process
\begin{equation}
    R^-_n = 
    \left( \avgXsq{n} - (\avgX{n}\wedge 0)^2 \right)^{n/2}   \e^{n/2} \prod_{i=1}^n \frac{1}{\hsig_{i-1}} \exp \left\{ -\frac{1}{2}\left( \frac{X_i - \hmu_{i-1}}{\hsig_{i-1}} \right)^2  \right\}
    \label{eqn:rn-minus}
\end{equation}
is an e-process for $\normals_{\mu \le 0}$ on $\{ \cF_n \}$.
\end{theorem}

There is clear similarity between the expression of the point null e-process \eqref{eqn:e-proc-ui-point}, and that of the one-sided \eqref{eqn:rn-minus}. The only difference is in that the former's $\avgXsq{n}$ term is here off by a $- (\avgX{n}\wedge 0)^2$ term in \eqref{eqn:rn-minus}. When the data are predominantly negative, this term inhibits the exponential growth of the e-process. Therefore, while we shall witness exponential growth of \eqref{eqn:e-proc-ui-point} when the actual mean is negative (significant in the scale of standard deviation), no such growth is likely to happen in \eqref{eqn:rn-minus}. We can formalize this by an asymptotic result similar to Proposition~\ref{prop:div-ui-eproc}.

\begin{proposition}[e-power of the universal inference one-sided t-test e-process]\label{prop:div-ui-eproc-1s}
    \revise{Under any distribution $P$ with mean $\mu$ and variance $\sigma^2$, for example $\normal{\mu}{\sigma^2}$,} suppose the point estimators $\{ \hmu_n \}$ and $\{ \hsig_{n}^{-2} \}$ satisfy the same assumptions as Proposition~\ref{prop:div-ui-eproc}. Then,
    \begin{equation}
        \lim_{n \to \infty} \frac{\log R_n^-}{n} = \frac{1}{2} \log(1 + (\mu \vee 0)^2/\sigma^2) \quad \text{almost surely}. 
    \end{equation}
    Consequently, $\{ R_n^- \}$ diverges almost surely to $R_\infty^- = \infty$ exponentially fast under $\normals_{\mu > 0}$.
\end{proposition}

A slightly weaker similarity between the expressions of these e-processes and the plug-in martingale \eqref{eqn:plugin-lr} can also be observed under careful comparison, which we do in the ``Test Process" row of \cref{tab:big-comp-ui}, where we compare all aspects of universal inference methods for Z-test (including a one-sided test involving taking infimum over $t \le 0$ in \eqref{eqn:plugin-lr}) and t-test.



\section{Sequential t-Tests via Scale Invariance}
\subsection{\citeauthor{lai1976confidence}'s \citeyearpar{lai1976confidence} Confidence Sequence}

Let us first quote a theorem due to \cite{lai1976confidence} who presented it almost without proof.
\begin{theorem}[Lai's t-Confidence Sequence; Theorem 2 of \cite{lai1976confidence}]\label{thm:lai-cs}
Choose a starting time $m \ge 2$ and a constant $a > 0$. Recall that $s_n^2 =  \frac{1}{n} \sum_{i=1}^n (X_i - \avgX{n})^2$. Further, define
\begin{align}
    b &:= \frac{1}{m} \left(1 + \frac{a^2}{m-1} \right)^m,
    \\
    \xi_n &:=  \sqrt{s_n^2 [(bn)^{1/n} - 1]}.
\end{align}
Then, the intervals $\CI_n =  [ \avgX{n} \pm \xi_n ] $ satisfy, for any $\mu$ and $\sigma > 0$,
\begin{equation}
    \normal{\mu}{\sigma^2}\left( \exists n \ge m, \mu \notin \CI_n  \right) \le 2(1-F_{m-1}(a) + af_{m-1}(a)),
\end{equation}
where $F_{m-1},f_{m-1}$ denote the CDF, PDF of t-distribution with $m-1$ degrees of freedom.
\end{theorem}
If we want $\{ \CI_n \}$ to be a $(1-\alpha)$-CS over $\normals$ for $\mu$, we need to solve $a$ from the equation $2(1-F_{m-1}(a) + af_{m-1}(a)) = \alpha$ beforehand.
To see the relationship between $a$ and $\alpha$,
note that when $a$ is large,
\begin{gather}
    f_{m-1}(a) \asymp  a^{-m},
    \\
    1-F_{m-1}(a) \asymp  a^{-(m-1)}.
\end{gather}
Hence
\begin{equation}
    \alpha \asymp a^{-(m-1)} \implies a \asymp \alpha^{- 1/(m - 1)}.
\end{equation}
The radius hence grows as (for fixed $m, n$ and $\alpha \to 0$)
\begin{equation}
    \xi_n \approx \sigma \sqrt{\left( \frac{1}{m} \left(1 + \frac{a^2}{m-1} \right)^m n \right)^{1/n} - 1  } \asymp \alpha^{-\frac{m}{n(m-1)}}.
\end{equation}
In terms of its shrinkage rate (fixed $m, \alpha$ but $n \to \infty$), we have 
\begin{align}
    \xi_n \approx \sigma \sqrt{  \frac{\log bn}{n} + \Otilde(n^{-2})}.
\end{align}

The CS is only valid from some time $m\ge2$, which we shall soon explain (and, in some sense, remedy). The proof of this theorem, as it turns out, hinges on the extended Ville's inequality (Lemma~\ref{lem:evi}) for nonintegrable nonnegative supermartingale. Equally interestingly, it mixes a parametrized family of martingales \emph{under a coarser filtration}. We shall begin our reworking of Lai's CS from this concept, and eventually state and prove Theorem~\ref{thm:lai-ensm}, a more revealing version of Theorem~\ref{thm:lai-cs}; as well as Theorem~\ref{thm:lai-e}, a variant of Theorem~\ref{thm:lai-cs} that does not involve an improper mixture and nonintegrability.

\subsection{Scale Invariant Filtration}\label{sec:si-filt}
We say that a function $f:\mathbb R^n \to \mathbb R$ is scale invariant if it is measurable and for any $x_1, \dots ,x_n \in \mathbb R$ and $\lambda > 0$, $f(x_1, \dots, x_n) = f(\lambda x_1, \dots, \lambda x_n)$.
Let us define the following sub-filtration of the canonical filtration $\{ \cF_n \}$.

\begin{definition}[Scale invariant filtration]\label{def:si-filt} For $n \ge 1$, let
    \begin{equation}\label{eqn:si-filt}
      \cF_n^* = \sigma( f(X_1, \dots, X_n) : f \text{ \emph{is scale invariant}} ).
    \end{equation}
    Then, the filtration $\{ \cF^*_n \}_{n \ge 1}$ is called the \emph{scale invariant filtration} of data $X_1, X_2, \dots$.
\end{definition}

Definition~\ref{def:si-filt} states that $\cF^*_n$ is the coarsest $\sigma$-algebra to which all scale invariant functions of $X_1, \dots, X_n$ are measurable. For example, recall that we denote by $T_{n-1}$ the t-statistic of the data $\sqrt{n-1} \frac{S_n - n \mu_0}{ \sqrt {n V_n - S_n^2} }$, which is $\cF^*_n$-measurable when $\mu_0 = 0$ and is a quantity frequently used to construct scale-invariant statistics later. To see that $\{ \cF^*_n \}$ is indeed a filtration, let $f$ be any scale invariant function $\reals^n \to \reals$, and define $g:\reals^{n+1}\to\reals$ as $g(x_1, \dots, x_n, x_{n+1}) = f(x_1, \dots, x_n)$ which is also scale invariant. So $f(X_1,\dots, X_n) = g(X_1,\dots, X_n, X_{n+1})$ is $\cF^*_{n+1}$-measurable. The arbitrariness of $f$ implies that $\cF^*_{n} \subseteq \cF^*_{n+1}$.

The filtration $\{ \cF_n^* \}$ contains all the information up to time $n$ about the \emph{relative} sizes of the observations. For example, $\{ \max\{ X_1, \dots, X_4 \} \ge 2X_3 \} \in \cF_4$ while $\{ \max\{ X_1, \dots, X_4 \} \ge 2 \} \notin \cF_4$. The reader may recall the exchangeable sub-$\sigma$-algebra in the theory of exchangeability and backwards martingales such as in \citet[Chapter 12]{klenke2013probability} for an analogy. \revise{The reader might also find the following analogy helpful in understanding the intuition: Imagine in prehistoric times when units of length were yet to be invented, our ancestors collected a sequence of leaves and fathomed their sizes. Due to the lack of units however, no absolute real-valued sizes (in inches, say) but only \emph{ratios} of sizes relative to each other (or to the first leaf) were recorded.} Actually, when $X_1$ is non-zero, $\cF_n^*$ has the following clean expression.
\begin{proposition}\label{prop:si-filt-rule}
    If $X_1 \neq 0$,
    \begin{equation}\label{eqn:si-filt-simple}
        \cF_n^* = \sigma\left( \frac{X_1}{|X_1|},  \frac{X_2}{|X_1|}, \dots, \frac{X_n}{|X_1|} \right).
    \end{equation}
\end{proposition}
 \begin{proof}
    Let $\cF_n^{**} = \sigma\left( \ttfrac{X_1}{|X_1|},  \ttfrac{X_2}{|X_1|}, \dots, \ttfrac{X_n}{|X_1|} \right)$.
    Suppose $f:\reals^n \to \reals$ is scale invariant. Then, $f(X_1, \dots, X_n) = f(X_1/|X_1|, \dots, X_n/|X_1|)$, which is clearly $\cF_n^{**}$-measurable, implying $\cF_n^{*} \subseteq \cF_n^{**}$.
    The inclusion $\cF_n^{**} \subseteq \cF_n^{*}$ is trivial: $x_1/|x_1|$, $x_2/|x_1|$, ..., $x_n/|x_1|$ are themselves scale invariant functions of $x_1, \dots, x_n$.
 \end{proof}
In our scenario of application, the t-test under Gaussian distribution, data are not non-zero but almost surely non-zero. We do not distinguish \eqref{eqn:si-filt} and \eqref{eqn:si-filt-simple} as the definition of the scale-invariant filtration because processes adapted to it are often derived by manipulations of $X_1/|X_1|, \dots, X_n/|X_1|$, which we call the \emph{scale invariant reduction} of the observations, {ignoring the $\normal{\mu}{\sigma^2}$-negligible event $\{X_1 = 0\}$}. Indeed, \cite{lai1976confidence} defines such filtration via \eqref{eqn:si-filt-simple} for the t-test case. However,
we remark that one would need the more general \eqref{eqn:si-filt} in Definition~\ref{def:si-filt} for the case when $X_1 \neq 0$ does not hold almost surely (e.g.\ when exploring a nonparametric extension of our set-up and methods).

As a counter-example, one may verify that the e-processes for t-test that we have previously introduced via universal inference, $\{ R_n \}$ in \eqref{eqn:e-proc-ui-point} and $\{ R_n^- \}$ in \eqref{eqn:rn-minus}, are \emph{not} adapted to $\{ 
\cF_n^* \}$. They would have been if the plug-in estimates $\hmu_{i-1}$ and $\hsig_{i-1}$ were allowed to be homogeneous of degree 1. However, as $\hmu_{0}$ and $\hsig_{0}$ must be constants, this cannot be the case. In fact, when $n=1$, these two e-processes are necessarily non-constant, and thus are not $\cF_1^*$-measurable. \revise{In what follows, we shall present numerous test processes on the scale invariant filtration $\{ \cF^*_n \}$, beginning with some generality. A close-up discussion on this intriguing issue follows afterward, in \cref{sec:does-filtration-matter}.}



\subsection{Scale Invariant Likelihood Ratios \revise{for Location-Scale Families}
}\label{sec:scale-lr}

Let us suppose, temporarily in this subsection for the sake of generality, that $P_{\mu, \sigma^2}$ is a distribution on $\mathbb R$ parametrized by $\mu \in \mathbb R$ and $\sigma > 0$, and $\eta$ a reference measure on $\mathbb R$, such that $ P_{\mu, \sigma^2} \ll \eta$ with density
\begin{equation}
    \frac{\d P_{\mu, \sigma^2}}{\d \eta}(x) = \sigma^{-1} g(\sigma^{-1}(x-\mu)).
\end{equation}
Further, suppose $\eta$ does not charge the singleton $\{ 0 \}$, so $P_{\mu, \sigma^2} (X_i = 0)$ is always $0$. The following was observed by \citet[Section 5]{lai1976confidence}.

\begin{lemma}[Density of the scale invariant reduction]\label{lem:jeffreys}
 Let $\varphi_n:\mathbb R^n \to \mathbb R^{n}$ be the function
 \begin{equation}
     \varphi_n(x_1, \dots, x_n) = (x_1/|x_1|, x_2/|x_1|, \dots, x_n/|x_1|).
 \end{equation}
 (It does not matter how $\varphi_n$ is defined when $x_1 = 0$.)
 
 Let $Q_{\mu ,\sigma^2}^n$ be the push-forward measure of $P_{\mu, \sigma^2}^{\otimes n}$ under the map $\varphi_n$. Let $\eta^{\langle n-1 \rangle}$ be the measure on $\{ \pm 1 \} \times \mathbb R^{n-1}$ that charges both $\{ 1 \} \times \mathbb R^{n-1}$ and $\{ - 1 \} \times \mathbb R^{n-1}$ with $\eta^{\otimes (n-1)}$. Then,
 \begin{equation}\label{eqn:max-inv-density}
     \frac{\d Q_{\mu ,\sigma^2}^n}{\d \eta^{\langle n-1 \rangle}} (x_1, \dots, x_n) =  \int_{\tau > 0}  \left\{\prod_{i=1}^n (1/\tau) g(x_i /\tau - \mu/\sigma)) \right\} \frac{   \d \tau}{\tau} = \int_{\tau > 0}  \left\{\prod_{i=1}^n \frac{\d P_{\tau\mu/\sigma, \tau^2}}{\d \eta}(x_i) \right\} \frac{   \d \tau}{\tau} .
 \end{equation}
\end{lemma}

The lemma says, in probabilistic terms, that the ``maximal invariant" reduction of the sample, $(X_1/|X_1|, X_2/|X_1|, \dots, X_n/|X_1|)$, has a density that relates to the density of the original sample in the form of \eqref{eqn:max-inv-density} --- which depends on $\mu$ and $\sigma^2$ only through $\mu/\sigma$.

Denoting the class of distributions having the same standardized mean by $\cP = \{ P_{\mu',{\sigma'}^2} : \mu'/\sigma' = \mu/\sigma \}$, we can further write the right hand side of \eqref{eqn:max-inv-density} in the form of
\begin{equation}
    \int_{\cP}  \left\{\prod_{i=1}^n \frac{\d P}{\d \eta}(x_i) \right\} \mathrm{J}(\d P) .
\end{equation}
Where $\mathrm{J}$ is the \emph{Jeffreys prior} over $\cP$ with density $\frac{1}{\sigma(P)}$. This improper prior is known in the Bayesian literature as an uninformative prior on the scale parameter. We have shown above that taking the mixture (in some sense) with the Jeffreys prior $\int_{\tau > 0} (\dots) \frac{\d \tau}{\tau}$ of likelihood ratios is equivalent to taking the likelihood ratio of the scale invariant reduction $X_2/|X_1|, \dots, X_n/|X_1|$. 
\revise{Indeed, our treatment for the location-scale family $\{ P_{\mu, \sigma^2} \}$ so far in this subsection to compute the density of $Q^n_{\mu,\sigma^2}$ is a special case of group invariant models formulated by \citet[Section 2.2]{perez2022statistics}, where the Jeffreys prior $J(\d P)$ with density $\sigma^{-1}(P)$ arises as the Haar measure of the group $(\mathbb R^+, \times)$, and the function $V_n$ corresponds to the ``maximally invariant statistic" under the group action. The exact relation of our work to that of \cite{perez2022statistics} shall be discussed in detail in \cref{sec:comp-perez}}. We shall make further remarks on how our approaches are related to previous Bayesian work with Jeffreys prior in \cref{sec:jzs}.

Now using the fact that the general non-i.i.d.\ likelihood ratios are martingales (Lemma~\ref{lem:lrm-joint}) on $Q_{\mu ,\sigma^2}^n$, we have:

\begin{lemma}[Scale invariant likelihood ratio]\label{lem:si-lr}
   For any $\theta$ and $\theta_0$, the process
   \begin{equation}\label{eqn:si-lr}
       h_n(\theta; \theta_0) = \frac{ \int_{\tau > 0}  \left\{\prod_{i=1}^n \frac{\d P_{\tau\theta, \tau^2}}{\d \eta}(X_i) \right\} \frac{   \d \tau}{\tau}  }{ \int_{\tau > 0}  \left\{\prod_{i=1}^n \frac{\d P_{\tau\theta_0, \tau^2}}{\d \eta}(X_i) \right\} \frac{   \d \tau}{\tau}  }
   \end{equation}
   is an NM for $\{ P_{\mu_0, \sigma_0^2} : \mu_0/\sigma_0 = \theta_0 \}$ on the scale invariant filtration $\{ \mathcal{F}^*_n \}_{n \ge 1}$ with expected value 1.
\end{lemma}

\subsection{An  Extended Test Martingale for t-Test}
\label{sec:lai-mix}

It is now convenient for us to replace the general distribution $P_{\mu, \sigma^2}$ in the previous subsection with the Gaussian $\normal{\mu}{\sigma^2}$, and $\eta$ with the Lebesgue measure on $\mathbb R$. 
A direct calculation with the Gaussian density function in Lemma~\ref{lem:si-lr} gives the following,

\begin{corollary}[Scale invariant t-likelihood ratio]\label{cor:t-lr} Let $\theta$ be any real number.
The process $\{ h_{\theta,n}  \}_{n \ge 1}$, defined by
\begin{equation}\label{eqn:t-lr}
    h_{\theta,n} =  \frac{\e^{- n \theta^2 /2 }}{\Gamma(n/2)} \int_{y > 0} y^{n/2-1}  \exp\left\{ -y + \theta S_n \sqrt{\frac{2y}{V_n}} \right\}  \d y,
\end{equation}
is an NM for $\normals_{\mu = 0}$ on the scale invariant filtration $\{ \mathcal{F}^*_n \}_{n \ge 1}$.
\end{corollary}
The $\theta$ in the martingale \eqref{eqn:t-lr} parametrizes the standardized mean $\mu/\sigma$ of the \emph{alternative}. That is, when the actual distribution is in the class $\normals_{\mu/\sigma = \theta}$, the process $\{h_{\theta, n}\}$ grows the fastest. A flat integral over $\theta$ yields an ENSM that stands behind Lai's t-CS (Theorem~\ref{thm:lai-cs}), which we restate in our language as follows. 

\begin{theorem}[Scale invariant t-test extended martingale]\label{thm:lai-ensm}
    The process $\{ H_n \}_{n \ge 1}$, defined as $H_1 = \infty$ and
    \begin{equation}
        H_n = \sqrt{\frac{2 \pi }{n}} \left( \frac{n V_n }{n V_n - S_n^2}  \right)^{n/2}, \quad (\text{for }n \ge 2)
    \end{equation}
    is an ENSM for $\normals_{\mu = 0}$ on the scale invariant filtration $\{ \cF^*_n \}_{n \ge 1}$. Lai's t-CS stated in Theorem~\ref{thm:lai-cs} follows from applying the extended Ville's inequality (Lemma~\ref{lem:evi}) to this ENSM.
\end{theorem}

To see how $\{H_n\}$ behaves under $\normals_{\mu = 0}$, recall that $T_{n-1}=\sqrt{n-1} \frac{S_n}{ \sqrt {n V_n - S_n^2} }$ has Student's t-distribution of $n-1$ degrees of freedom.  $H_n$ can thus be re-expressed as
\begin{equation}\label{eqn:lai-ensm-T-expr}
    H_n = \sqrt{\frac{2 \pi }{n}} \left( 1 + \frac{T_{n-1}^2}{n-1}  \right)^{n/2}.
\end{equation}
$\{ H_n \}$ works favorably as a test process for the null $\normals_{\mu = 0}$ and the alternative $\normals_{\mu \neq 0}$ due to the following symptotic result on \emph{both} its e-power under the alternative and its convergence under the null.

\begin{proposition}[Asymptotic behavior of the scale invariant t-test extended martingale]\label{prop:conv-t-ensm}
    Under any \revise{distribution with mean $\mu$ and variance $\sigma^2$, for example} $\normal{\mu}{\sigma^2}$,
    \begin{equation}
        \lim_{n \to \infty}  \frac{\log H_n}{n}  = \frac{1}{2}\log( 1 + \mu^2/\sigma^2 ) \quad \text{almost surely}.
    \end{equation}
    Consequently, $\{ H_n \}$ diverges almost surely to $H_\infty = \infty$ exponentially fast under $\normals_{\mu \neq 0}$. Furthermore, $\{ H_n \}$ converges almost surely to $H_\infty = 0$ under $\normals_{\mu = 0}$.
\end{proposition}

The reader may compare the ENSM $\{ H_n \}$ with the ENSM for the Z-test case by \citet[Proposition 5.7]{ensm}, obtained via a flat mixture over the standard Gaussian likelihood ratio. Both are free of any parameter. Under the null, both, as extended \emph{martingales}, start from $\infty$ and shrink to 0 almost surely. Under the alternative, both start from $\infty$ and diverge back to $\infty$.
Both can be seen as frequentist embodiments of the Bayesian idea of uninformative, improper prior.


\subsection{Classical Test Martingales for t-Test}
\label{sec:N-mix}

In some sense a classical, integrable test martingale issued at 1 is preferred. Besides a simple, universally valid rejection rule ``reject when the test process exceeds $1/\alpha$", the use of classical Ville's inequality often leads to closed-form CSs as opposed the one in Theorem~\ref{thm:lai-ensm} that involves root finding.
Further, these classical NMs often come with ``tunable hyperparameters" arising from the mixture distributions. We replace the flat mixture on $\theta$ that leads to Theorem~\ref{thm:lai-ensm} by a Gaussian one, obtaining the following classical test martingales.


\begin{theorem}[Scale invariant t-test martingale]\label{thm:lai-e} For any $c > 0$, the process $\{ G_n^{(c)} \}_{n \ge 1}$ defined by
\begin{equation}\label{eqn:gaussian-mix-eproc}
    G_n^{(c)} = \sqrt{\frac{c^2}{n+c^2}} \left( \frac{(n+c^2) V_n }{(n+c^2) V_n  - S_n^2}  \right)^{n/2}
\end{equation}
is an NM for $\normals_{\mu = 0}$ on the scale invariant filtration $\{ \mathcal{F}^*_n \}_{n \ge 1}$ with expected value $\Exp[G_n^{(c)}] = G_1^{(c)} = 1$. Consequently, let
\begin{equation}\label{eqn:gaussian-mix-radius}
    \operatorname{radius}_n = \sqrt{\frac{ (n+c^2)\left(1- \left( \frac{\alpha^2 c^2}{n+c^2} \right)^{1/n}\right) }{ \left\{ \left( \frac{\alpha^2 c^2}{n+c^2} \right)^{1/n} (n+c^2) - c^2 \right\} \vee 0   } \left(\avgXsq{n} - \avgX{n}^2 \right)},
\end{equation}
the intervals
\begin{equation}\label{eqn:gaussian-mix-cs}
   \left[  \avgX{n} \pm \operatorname{radius}_n   \right]
\end{equation}
form a $(1-\alpha)$-CS for $\mu$ over $\normals$. (When the denominator in \eqref{eqn:gaussian-mix-radius} takes 0, $\operatorname{radius}_n = \infty$, the CI is the entire $\mathbb R$ at time $n$.)
\end{theorem}

\revise{The tuning parameter $c^2$ of prior precision, we note, does not affect the safety (the sequential test having type 1 error rate within $\alpha$) or coverage (the CS covering $\mu$ with at least $1-\alpha$ probability) of our method, but does affect the power of the sequential test and the tightness of the confidence sequence. First, considering the growth of the e-process under the alternative, since the mixing distribution $\normal{0}{c^{-2}}$ is put over all possibilities of the true effect size $\theta = \mu/\sigma$, it would be ideal if $\normal{0}{c^{-2}}$ ``covers'' the true underlying effect size with adequate weight. Therefore, if one has acquired prior knowledge of the rough scale of the true effect size $\theta$, matching $c^{-1}$ to the same scale can improve power, as we shall demonstrate numerically in \cref{sec:sim-ep}. 
However, as we shall see soon in Proposition~\ref{prop:conv-t-nm}, in the large sample regime different choices of $c^2$ lead to the same e-power under a fixed alternative distribution. On the other hand, the tuning parameter $c^2$ plays a distribution-free role on the tightness of the $(1-\alpha)$-CS, as the expression for the radius \eqref{eqn:gaussian-mix-radius} contains decoupled $(c, \alpha, n)$-dependent and data-dependent terms. In this case, difference choices of $c^2$ optimize the radius of the CS at different times $n$. As discussed in a similar situation of the Gaussian mixture Z-test CS by \citet[Section 3.5]{howard2021time}, there is a trade-off in ``making a bound tighter
for some range of times requires making it looser at other times.'' We shall numerically demonstrate this in our case in \cref{sec:sim-epr}. More specifically, to approximately minimize $\operatorname{radius}_n$ at some fixed $n$,}
one can take $c^2$ such that
\begin{equation}\label{eqn:optim-c}
   \frac{c^2}{n+c^2}  = \alpha^{2/(n-1)} 2^{-n/(n-1)},
\end{equation}
in which case
\begin{equation}
    \operatorname{radius}_n = \sqrt{\left(\alpha^{-2/(n-1)} 2^{n/(n-1)}  - 2\right) \left(\avgXsq{n} - \avgX{n}^2 \right) },
\end{equation}
which is always finite and has a growth rate of $\alpha^{-1/(n-1)}$ as $\alpha \to 0$ while keeping the sample fixed, a slightly worse rate compared to the $\alpha^{-1/n}$ of the universal inference CS, and at least as good as the $\alpha^{-m/n(m-1)}$ rate of Lai's original CS since $m\le n$. On the other hand, if $c$ and $n$ are both fixed, the radius grows like $(\alpha- \alpha_{\star}(n) )^{-1/n}$ as $\alpha \downarrow \alpha_{\star}(n) =  \exp ({\Theta}(-n \log n))$.

While the test martingale \eqref{eqn:gaussian-mix-eproc} works without any limitation on $n$ or $c$, the CS \eqref{eqn:gaussian-mix-cs} is non-trivial only when the  $\left( \frac{\alpha^2 c^2}{n+c^2} \right)^{1/n} (n+c^2) - c^2 $ term in the denominator is positive. 
For a fixed $\alpha$ and a fixed $c$, the CS is non-trivial on $n \ge n_0$ for some $n_0$, instead of all $n \ge 1$. For example, when $c = 0.01$ and $\alpha = 0.05$, the range of \eqref{eqn:gaussian-mix-cs} is $n \ge 3$. Thus,  the starting time $m$ in Theorem~\ref{thm:lai-cs} seems to be avoided by switching from the extended Ville's inequality to
the classical Ville's inequality, but it remains in another form.
In this case, if $n\to\infty$, $\operatorname{radius}_n$ shrinks as a function of $n$ at the rate of
\begin{equation}
    \sqrt{ 1- \left( \frac{\alpha^2 c^2}{n+c^2} \right)^{1/n} } \approx \sqrt{\frac{\log(  n / \alpha^2 c^2 )}{n}}.
\end{equation}

One may further express $G_n^{(c)}$ in terms of the t-statistic under $\normals_{\mu = 0}$ as well,
\begin{equation}\label{eqn:lai-nm-tstat}
    G_n^{(c)} = \sqrt{\frac{c^2}{n+c^2}} \left( 1 + \frac{ n}{ \frac{(n+c^2)(n-1) }{T_{n-1}^2 } + {c^2}  } \right)^{n/2}
\end{equation}
which again leads to the following asymptotic result that shows their merits for being test process candidates.

\begin{proposition}[Asymptotic behavior of the scale invariant t-test martingales]\label{prop:conv-t-nm}
    Under any \revise{distribution with mean $\mu$ and variance $\sigma^2$, for example} $\normal{\mu}{\sigma^2}$,
    \begin{equation}
        \lim_{n \to \infty}  \frac{\log G_n^{(c)} }{n}  = \frac{1}{2}\log( 1 + \mu^2/\sigma^2 ) \quad \text{almost surely}.
    \end{equation}
    Consequently, $\{ G_n^{(c)} \}$ diverges almost surely to $G^{(c)}_\infty = \infty$ exponentially fast under $\normals_{\mu \neq 0}$. Furthermore, $\{ G_n^{(c)} \}$ converges almost surely to $G^{(c)}_\infty = 0$ under $\normals_{\mu = 0}$.
\end{proposition}

The NM $\{ G_n^{(c)} \}$ thus have the same asymptotic properties as the ENSM $\{ H_n \}$, but the free parameter $c$, which is absent for the ENSM, does introduce a difference that emerges only non-asymptotically. 
To wit, if one fixes $n$ and the data-dependent quantity $T^2_{n-1}$, the value of $G_n^{(c)}$ would approach 0 (hence the power vanishes) if $c$ is too large or too small.

The reader may compare Theorem~\ref{thm:lai-e} with Theorem~\ref{thm:lai-ensm}, and compare this comparison with the comparison between the Gaussian mixed NSM \citep[Proposition 5.6]{ensm} and flat mixed ENSM \citep[Proposition 5.7]{ensm} in the Z-test case. Multiple similarities manifest. A full comparison shall be presented next in \cref{tab:big-comp}.
It is unclear why \cite{lai1976confidence} skipped this more universally accepted method of a proper Gaussian mixture and used an improper flat mixture instead, which seems very ahead of its time in hindsight.

We also note that \citet[Proposition 3.7]{lindon2022anytime} concurrently derive the same t-test martingale and confidence sequence as our Theorem~\ref{thm:lai-e}, a special case of their F-test martingale.
Our paper, however, is concerned much more with comparing the Gaussian mixture with Lai's flat mixture in theory and practice, with understanding the tradeoffs between test martingales in the reduced filtration and e-processes in the original filtration for this fundamental problem, as well as deriving their dependence on $\alpha$.

In place of the full Gaussian prior, a half-Gaussian prior mixture \revise{over $\mu/\sigma > 0$} will give us a ``semi-one-sided" test process in the sense that it is valid only for the point null $\normals_{\mu = 0}$ (unlike Theorem~\ref{thm:ui-ttest-onesided}) but still leads to \revise{a test powerful against $\normals_{\mu > 0}$ and} a one-sided confidence sequence.

\begin{theorem}[Scale invariant semi-one-sided t-test e-process]\label{thm:si-onesided}
    For any $c>0$, the process $G_n^{(c-)}$ defined by
    \begin{equation}\scriptsize
        G_n^{(c-)} = 2 \sqrt{\frac{c^2}{n+c^2}} \left( \left(1- \frac{S_n^2}{(n+c^2) V_n }  \right)^{-n/2} - \left(1- \frac{(S_n \wedge0) ^2}{(n+c^2) V_n }  \right)^{-n/2} \right)= 2 G_n^{(c)} - 2 G_n^{(c)} |_{(S_n \leftarrow S_n \wedge 0)}.
    \end{equation}
    is an e-process for $\normals_{\mu = 0}$ on the scale invariant filtration $\{\cF_n^*\}_{n \ge 1}$.
\end{theorem}
Here, the notation $G_n^{(c)} |_{(S_n \leftarrow S_n \wedge 0)}$ refers to replacing the sample sum $S_n$ in the two-sided martingale \eqref{eqn:gaussian-mix-eproc} with its negative part $S_n \wedge 0$.
The process is an e-process instead of a martingale because a step of approximation is used to unravel a non-closed-form martingale, as seen from its proof in \cref{sec:pf-lr}. As a direct consequence of Proposition~\ref{prop:conv-t-nm}, the process $G_n^{(c-)} = 2 G_n^{(c)} - 2 G_n^{(c)} |_{(S_n \leftarrow S_n \wedge 0)}$ diverges almost surely to $G^{(c-)}_\infty = \infty$ exponentially fast under $\normals_{\mu > 0}$ and converges almost surely to $G^{(c-)}_\infty = 0$ under $\normals_{\mu = 0}$, and satisfies exactly the same asymptotic as Proposition~\ref{prop:div-ui-eproc-1s},
\begin{equation}
        \lim_{n \to \infty} \frac{\log G^{(c-)}_n}{n} = \frac{1}{2} \log(1 + (\mu \vee 0)^2/\sigma^2) \quad \text{almost surely}. 
    \end{equation}
\revise{
Further, \citet[Corollary 8]{perez2022statistics} show that at any fixed sample size $n$ the scale-invariant t-likelihood ratio \eqref{eqn:t-lr} is actually an e-\emph{value} for the larger one-sided null $\normals_{\mu \le 0}$ if $\theta \ge 0$ and vice versa. Therefore, the mixture $G_n^{(c-)}$ expressed above is also an e-value for $\normals_{\mu \le 0}$, meaning that one can non-sequentially test the null $\normals_{\mu \le 0}$ using this statistic. However, it is unclear if the t-likelihood ratio process \eqref{eqn:t-lr}, and thus $\{G_n^{(c-)}\}$, is an e-\emph{process} for $\normals_{\mu \le 0}$. We leave this question open for future investigation. Again, it is worth noting that $\{G_n^{(c-)}\}$ being an e-process for the point null $\normals_{\mu = 0}$ is sufficient for a one-sided confidence sequence for $\mu$.

\subsection{Remarks}

\subsubsection{Remark I: Symmetry and e-Hacking}

An ideal property that the test processes and confidence sequences derived in this section via scale invariant reduction enjoy, but those via universal inference do not, is that they are all symmetric with respect to the observations $X_1,\dots, X_n$. This is because they depend on the sample only through the symmetric statistics $S_n$ and $V_n$.

This makes the methods robust against the following way to ``hack'' the e-values or the confidence sequences: a practitioner may permute a batch of data, compute the e-processes or confidence sequences repeatedly via multiple orders of observation,
and announce the maximal e-process or minimal width confidence sequence. While e-processes allow optional stopping and continuation, they do not allow ``optional permutation'' like this. However symmetric methods, since their final results are path-independent, are robust against this type of hacking.

\subsubsection{Remark II: Asymptotic Coverage for Non-Gaussian Data}

Classical tests derived under Gaussian assumption (e.g.\ the classical t-, F-, and $\chi^2$-tests) are frequently applied to data that are not known to be Gaussian on grounds of the central limit theorem, leading to type 1 error or interval coverage guanratees that are of asymptotic nature. In this subsection, let us substantiate the similar claim that our confidence sequences derived in this section are all \emph{asymptotic confidence sequences} under any i.i.d.\ data with finite variance. Let us use $\cL^2$ to denote the set of all these distributions on $\mathbb R$. 

The concept of asymptotic confidence sequences is proposed recently by \citet[Section 2.1]{waudby2021time}. To wit, let $\mathcal{P}$ be a family of distributions and  $\theta : \mathcal{P} \to \mathbb R$ be the parameter of interest. A $(1-\alpha)$-{asymptotic confidence sequence} over $\mathcal{P}$ for $\theta$ is a sequence of intervals $\{ [\hat{\theta}_n \pm r_n ] \}_{n \ge 1}$ such that, the convergence
$\lim_{n \to\infty} r_n/r_n^* = 1 $ holds $P$-almost surely for all $P \in \cP$, where $\{ [\hat{\theta}_n \pm r_n^* ] \}_{n \ge 1}$ form a $(1-\alpha)$-CS over $\mathcal{P}$ for $\theta$ (non-asymptotically, in the sense that is defined in \cref{sec:seq-stat}). Appealing to a similar technique, we prove in \cref{sec:pf-si} that, for example,
\begin{theorem}\label{thm:asympcs}
    The intervals \eqref{eqn:gaussian-mix-cs} form a $(1-\alpha)$-{asymptotic confidence sequence} for $\mu$ over $\mathcal{L}^2$.
\end{theorem}

On the other hand, we note that it is \emph{impossible} to construct a non-degenerate universal \emph{non-}asymptotic confidence interval even for the set of all \emph{bounded} distributions (with unknown bounds). Thus it also excludes the existence of CIs or CSs for larger classes, e.g.\ the set of all subGaussian distributions, or $\cL^2$. This result can be traced back to \citet[Theorem 2]{bahadur1956nonexistence}. The crux of the matter, we remark, lies in the \emph{convexity} of these large nonparametric sets of distributions.

\subsubsection{Remark III: Relation to the Group Invariance Study by \cite{perez2022statistics}}
\label{sec:comp-perez}

Up to the derivation of the t-likelihood ratio between $\mu/\sigma = 0$ and $\mu/\sigma = \theta$ in Corollary~\ref{cor:t-lr}, we have largely shared the path with \cite{perez2022statistics} who are primarily interested in the optimality of such canonical group-invariant (scale-invariant in the t-test) likelihood ratios for point null (for the ``maximally invariant parameter", e.g.\ $\mu/\sigma = 0$) against the point alternative (e.g.\ $\mu/\sigma = \theta$). \cite{perez2022statistics} do discuss the extension to composite nulls and alternatives in their Section 3.3, via a seemingly similar method of mixture (that takes place on \emph{both} sides of the ratio, over the null and alternative sets alike). However, their approaches and conclusions differ substantially from ours, and let us restate their two major results in t-test terms. 

Central to the results of \citet[Section 3.3]{perez2022statistics} is the concept of Bayesian \emph{marginal} measure 
\begin{equation}
    \pi_\theta P_\theta(\cdot) :=  \int_\Theta P_{\theta} (\cdot) \pi(\d \theta)
\end{equation}
for a family of measures $\{P_\theta\}_{\theta \in \Theta}$ and a prior distribution $\pi$ on $\Theta$.

Corollary 8 by \cite{perez2022statistics} concerns testing the null $\normals_{\mu/\sigma \in \Delta_0}$ against the alternative $\normals_{\mu/\sigma \in \Delta_1}$. Suppose that among \emph{all} pairs of distributions on $\Delta_0$ and $\Delta_1$, $\pi^0$ and $\pi^1$ minimize the Kullback-Leibler divergence $\kl( \pi^1_{\theta} Q^n_{\sigma\theta, \sigma^2} \| \pi^0_\theta Q^n_{\sigma\theta, \sigma^2} )$, where we recall that $Q^n_{\mu, \sigma^2}$ is the distribution of the reduced sample, defined in Lemma~\ref{lem:jeffreys}, and depends only on $\theta = \mu/\sigma$ (hence the minimization is independent of the choice of $\sigma$). 
Then, the likelihood ratio $\frac{\d \pi^1_\theta Q^n_{\sigma \theta, \sigma^2}}{\d \pi^0_\theta Q^n_{\sigma \theta, \sigma^2}}(X_1/|X_1|, \dots, X_n/|X_1|)$ is the optimal e-value --- for the null $\normals_{\mu/\sigma \in \Delta_0}$ and in terms of the worst case e-power against the alternative $\normals_{\mu/\sigma \in \Delta_1}$. 
The minimization, however, can only be nontrivially established when $\Delta_0$ and $\Delta_1$ are closed and separated, and when this happens the minimizing measures $\pi^0$ and $\pi^1$ are singular (i.e.\ Dirac masses on the boundary) as is exemplified by ``Example 1 (continued)" after Corollary 8 in \cite{perez2022statistics}. As we mention after Theorem~\ref{thm:si-onesided}, while a mixture leads to an e-value for the null $\normals_{\mu \le 0}$ against the alternative $\normals_{\mu > 0}$ in identical form as the e-process in Theorem~\ref{thm:si-onesided}, it is unclear if it is also an e-process under this larger null.

Corollary 9 by \cite{perez2022statistics}, on the other hand, requires no such minimum divergence property on the mixing measures. Still letting $\pi_0$ and $\pi_1$ be distributions on $\Delta_1, \Delta_2 \subseteq \mathbb R$, it states that the likelihood ratio $\frac{\d \pi^1_\theta Q^n_{\sigma \theta, \sigma^2}}{\d \pi^0_\theta Q^n_{\sigma \theta, \sigma^2}}(X_1/|X_1|, \dots, X_n/|X_1|)$ is the optimal e-value for the \emph{marginal} null $\{\pi^0_\theta \normal{\sigma \theta}{\sigma^2} : \sigma > 0  \}$ against the \emph{marginal} alternative $\{ \pi^1_\theta \normal{\sigma \theta}{\sigma^2} : \sigma > 0  \}$, understood as the mixture model where a distribution is first drawn according to the prior from the null or alternative set, and then the sample is drawn from the distribution. It is in general not an e-value for the composite null $\normals_{\mu/\sigma \in \Delta_0}$; further, in the point null case $\Delta_0 = \{ 0 \}$ and $\Delta_1 = \mathbb R \setminus \{0 \}$ where the null, the mixture, and the e-value obtained via the mixture coincide with those of our Theorem~\ref{thm:lai-e} (taking $\pi^1 = \normal{0}{c^{-2}}$), their statement does not provide power optimality for every single distribution in $\normals_{\mu/\sigma \in \Delta_1}$ as our statement does.



We finally note that their ``ratio of mixtures" approach, in contrast to our ``mixture of ratios", bears a closer tie to the classical Bayesian method of \emph{Bayes factors}, a topic discussed immediately next.

}

\subsubsection{Remark IV: Bayesian t-Test with the JZS Prior, and Cauchy Mixture}
\label{sec:jzs}

In a highly influential paper, \cite{rouder2009bayesian} provide a Bayesian framework for the t-test that makes extensive use of \emph{Bayes factors}. To explain this approach, we consider as we did in \cref{sec:scale-lr} a general location-scale family $P_{\mu, \sigma^2}$ dominated by a reference measure $\eta$. The Bayes factor for the null $\cP_{\mu = 0} = \{ P_{0,\sigma^2} : \sigma>0 \}$ and the alternative $\cP_{\mu \neq 0}$ is defined as \citep[p.229]{rouder2009bayesian}
\begin{equation}
    B_n = \frac{M^{0}_n}{M^{1}_n} = \frac{\int \left\{ \prod_{i=1}^n \frac{\d P_{0, \tau^2}}{\d \eta} (X_i) \right\} \pi_0(\d \tau^2) }{\int \left\{ \prod_{i=1}^n \frac{\d P_{m, \tau^2}}{\d \eta} (X_i) \right\} \pi_1(\d m, \d \tau^2)  },
\end{equation}
where $\pi_0$ and $\pi_1$ are priors on $\mathbb R^+$ and $\mathbb R \times \mathbb R^+$, that can be chosen freely by the statistician. \citet[p.231]{rouder2009bayesian}, regarding themselves as ``objective Bayesians", recommend the following choice:
\begin{equation}
\pi_0(\d \tau^2) = \tau^{-2} \d \tau^2 = 2\tau^{-1}\d \tau,\quad \text{and} \quad\pi_1(\d m, \d \tau^2) = C_{0, 1}(\d (m/\tau)) \cdot \pi_0 (\d \tau^2),
\end{equation}
where $C_{0, 1}$ is the standard Cauchy distribution. This is dubbed the ``JZS prior", an acronym for the Jeffreys prior $\pi_0$, and the Cauchy prior on $\mu/\sigma$ due to \cite{zellner1980posterior}. We can immediately write $B_n$ as
\begin{equation}
 B_n^{\text{JZS}} =   \frac{ \int_{\tau > 0}  \left\{\prod_{i=1}^n \frac{\d P_{0, \tau^2}}{\d \eta}(X_i) \right\} \frac{   \d \tau}{\tau}  }{ \int_{\theta} \int_{\tau > 0}  \left\{\prod_{i=1}^n \frac{\d P_{\tau\theta, \tau^2}}{\d \eta}(X_i)  \right\} \frac{   \d \tau}{\tau} C_{0, 1}(\d \theta)  }.
\end{equation}

Comparing to mixing the scale-invariant likelihood ratio \eqref{eqn:si-lr} (letting $\theta_0 = 0$) with a prior $\varpi$ on the alternative $\theta$,
\begin{equation}
 M_n^\varpi =  \int  h_n(\theta; 0) \varpi(\d \theta) = \frac{ \int_{\theta} \int_{\tau > 0}  \left\{\prod_{i=1}^n \frac{\d P_{\tau\theta, \tau^2}}{\d \eta}(X_i)  \right\} \frac{   \d \tau}{\tau} \varpi(\d \theta)  }{ \int_{\tau > 0}  \left\{\prod_{i=1}^n \frac{\d P_{0, \tau^2}}{\d \eta}(X_i) \right\} \frac{   \d \tau}{\tau}  }.
\end{equation}
Clearly, $ B_n^{\text{JZS}}$ equals $1/ M_n^\varpi$ when the mixture measure $\varpi$ is taken to be $C_{0,1}$. Some remarks regarding the comparison. First, \cite{rouder2009bayesian} did not mention the sequential benefits of their approach, as $ \{ 1/B_n^{\text{JZS}} \}$ is an NM for $\cP_{\mu = 0}$ by virtue of Lemma~\ref{lem:si-lr}, or equivalently $\{ B_n^{\text{JZS}} \}$ is an anytime-valid p-value which we briefly defined in \cref{sec:seq-stat}. 
Second, while choosing $\pi_0$ to be the Jeffreys prior  seems necessary to attain a test process according to our \cref{sec:scale-lr}, different (perhaps objectivistic) methodological choices lead to different priors on $\theta = \mu/\sigma$. \citeauthor{lai1976confidence}'s choice in \cref{sec:lai-mix} is a flat $\varpi$ while in \cref{sec:N-mix} we choose $\varpi = \normal{0}{c^{-2}}$, both leading to closed-form expressions. \citeauthor{rouder2009bayesian}'s Cauchy prior, in turn, arises from a hyper-prior on the precision $c^{-2}$ of this Gaussian prior $\theta \sim \normal{0}{c^{-2}}$, which is $c^2 \sim \chi^2_1$ proposed by \cite{zellner1980posterior}, but this has led to the $ B_n^{\text{JZS}}$ that lacks a closed form.

\revise{
\subsubsection{Remark V: Does the Filtration Matter?}\label{sec:does-filtration-matter}

We finally discuss the issue of the reduced, scale invariant filtration $\{ 
\cF_n^* \}$, on which various test processes in this section are defined. First, we note the perhaps surprising fact that stopping, say, the test martingale $\{G_n^{(c)}\}$ in Theorem~\ref{thm:lai-e}, at a stopping time $\tau$ on the larger, canonical filtration $\{\cF_n \}$, and rejecting the null $\normals_{\mu = 0}$ if $G_\tau^{(c)} \ge 1/\alpha$ does \emph{not} inflate the type 1 error rate. This is formally due to the following statement.

\begin{proposition}
Let $\{ M_n \}$ be a nonnegative supermartingale with $M_0 = 1$ on the filtration $\{ \cG_n \}$ in the probability space $( \Omega, \cA, \Pr )$. In particular, $\cG_{\infty}$ can be strictly smaller than $\cA$.
For any $\cA$-measurable random variable $T$ taking values in $\{ 0,1,2,\dots ,\infty \}$ and $\alpha \in (0,1)$, $\Pr[  M_T \ge 1/\alpha  ] \le \alpha$.
\end{proposition}
\begin{proof} As an NSM, $M_n$ converges to some random variable $M_\infty$ almost surely, so we have 
\begin{align}
    \{  M_T \ge 1/\alpha \} = \bigcup_{n=0,1,\dots,\infty}\left(\{ M_n \ge 1/\alpha   \} \cap \{ T = n \} \right) \subseteq \left\{ \sup_{n} M_n \ge 1/\alpha  \right\}
\end{align} 
   since each $\{ M_n \ge 1/\alpha   \}$ is a subset of $\left\{ \sup_{n} M_n \ge 1/\alpha  \right\}$. This concludes the proof because the event $\left\{ \sup_{n} M_n \ge 1/\alpha  \right\}$ has probability at most $\alpha$ due to Ville's inequality.
\end{proof}

In particular, the above is true when $T$ is a stopping time on a filtration finer than $\{ \cG_n \}$.
It is easy to see that this argument extends easily to the corresponding stopped tail bound for an e-process for a family of data-generating distributions $\cP$, and for the coverage probability of a stopped confidence sequence at $T$.
While the \emph{tail probability} $\Pr[ M_T \ge 1/\alpha ]$ is always bounded by $\alpha$ under arbitrary random time $T$, it is worth noting
that there is no guarantee on the \emph{expected value} $\Exp[M_T]$. That is, $M_T$ is not necessarily an e-value.

Back to the test martingale $\{ G_n^{(c)} \}$ on $\{\cF_n^*\}$, if one stops the test at some stopping criterion $\tau$ that comes from the ``sizes'' of the observations, or even some external source of randomness coupled to the experiment, while it is safe to reject if $G^{(c)}_\tau \ge 1/\alpha$ or report the confidence sequence $[\avgX{\tau} \pm \operatorname{radius}_\tau]$ at this time, the stopped value $G^{(c)}_\tau$ might not be an e-value and thus it is unsafe to use $G^{(c)}_\tau$ in various downstream tasks that require e-value inputs. These include:
\begin{itemize}
    \item Multiple hypothesis testing with the e-BH procedure \citep{xu2021unified} under arbitrary dependence.
    \item Combining independent e-values via multiplication \citep[Proposition 2]{grunwald2020safe} to obtain a supermartingale.
    \item Taking a weighted average over dependent e-values to obtain a new e-value. 
\end{itemize}
On the other hand, it is possible to sightly \emph{decrease} the random variable $M_T$, making use of the tail bound $\Pr[M_T\ge 1/\alpha] \le \alpha$, to construct an e-value $f(M_T)$, i.e.\ $\Exp[f(M_T)] \le 1$, so that all of the e-value-based downstream tasks above can take $f(G_\tau^{(c)})$ as input. For example, $f(x) =  (x\vee 1 - 1 - \log(x\vee 1))/\log^2 (x\vee 1)$ works. This is called an ``adjuster'' in the e-value literature. See e.g.\ \cite{shafer2011test,KOOLEN2014144,calibration} for the topic in various contexts, and the recent study by \cite{choe2024combining} for the cross-filtration problem similar to our discussion here.
}

\begin{landscape}
\section{Comparison of Results}

\subsection{Theoretical Comparison}

We have presented three t-confidence sequences (i.e., confidence sequences for $\mu$ over $\normals$) so far, based on universal inference and scale invariance, and we summarize them in \cref{tab:t-cs}. Two additional results that can be obtained from (1) the two-step plug-in method of \citet[Appendix E]{jourdan2023dealing}, (2) the confidence sequence for median by \citet[Theorem 1]{quantile} are also included, both described in \cref{sec:more-tcs}. In terms of shrinkage rates, these two results enjoy the optimal iterated logarithmic rate faster than the other three. However, we point out that this is achieved by the technique of ``stitching" which \emph{universally} turns a $n^{-1/2}$ CI into $n^{-1/2}\log\log(n)$ CS (by paying a significant price of constants), thus applicable to the other three methods as well, see e.g.\ \citet[Section 3.1]{howard2021time} and \citet[Section 4]{duchi2024information}. Growth-wise, the original flat-mixed CS by \cite{lai1976confidence} and the result of \cite{jourdan2023dealing} are subject to a parameter-dependent worse rate, and the result of \cite{quantile} becomes degenerate when $\alpha$ is small enough, thus they are less favorable when $\alpha$ is very small (an issue that Theorem~\ref{thm:lai-e} can avoid by tuning $c^2$.).


\begin{table}[!h] 
    \centering \small
    \begin{tabular}{c||c|c|c|c|c} \hline
        Result & Theorem~\ref{thm:ui-ttest} & Theorem~\ref{thm:lai-cs} \citep{lai1976confidence} & Theorem~\ref{thm:lai-e} & \makecell{Theorem~\ref{thm:plug}  \\ (JDK23)} & \makecell{Proposition~\ref{prop:med} \\ (HR22)}
    \\ \hline \hline
        Method & Universal inference & \multicolumn{2}{c|}{Scale invariant likelihood mixture} & \makecell{Likelihood mixture (mean) \\ $\chi^2$ Chernoff bound  (variance)} & Chernoff bound for median \\ \hline
        Mixture & N/A & Flat & \multicolumn{2}{c|}{Gaussian} & N/A \\ \hline
        Free parameters & \makecell{Point estimators \\ $\tilde{\mu}_i$ and $\tilde{\sigma}_i^2$} & Starting time $m \ge 2$ & Prior precision $c^2$ & \makecell{Stitching parameters \\ $\eta>0, s>1$} & N/A \\ \hline
        \makecell{Growth rate  \\ ($\alpha \to  0$, or $\alpha \to  \alpha_{\star}$\\ $=\e^{-\widetilde{\Theta}(n)}$  while $n$ fixed) } & $\alpha^{-1/n}$ & $\alpha^{-m/n(m-1)}$ & \makecell{$\alpha^{-1/(n-1)}$ (CI) \\ $(\alpha-\alpha_{\star})^{-1/n}$ (CS)  } & $\alpha^{-(1+\eta)/(n-1)}$ & $\sqrt{\log(1/(\alpha -\alpha_{\star}))}$ \\ \hline
         \makecell{Shrinkage rate \\  ($n \to  \infty$ while $\alpha$ fixed)} & \multicolumn{3}{c|}{$ n^{-1/2} \polylog(n) $} &  \multicolumn{2}{c}{$n^{-1/2} \log\log (n)$}   \\ \hline
    \end{tabular}
    \caption{Comparison of t-confidence sequences in this paper. JDK23 refers to the paper by \cite{jourdan2023dealing}, and HR22 refers to the paper by \cite{quantile}.}
    \label{tab:t-cs}
\end{table}

We now zoom out and cross-compare sequential t-tests with sequential Z-tests in terms of their test processes. As we mentioned earlier, both universal inference and likelihood mixture can be used on Z-test and t-test to construct e-processes or extended e-processes, and confidence sequences.
Additional results on universal inference Z-tests may be found in \cref{sec:uiz}. 
We summarize in \cref{tab:big-comp-ui,tab:big-comp}  the results from universal inference and likelihood mixture respectively. Note that the plug-in method of \cite{jourdan2023dealing} and the median bound of \cite{quantile} yield confidence sequences without any test process, so are not included in these tables.

\begin{table}[]
\renewcommand{\spadesuit}{K}
    \vspace{-3em}
    \centering \small
    \begin{tabular}{c||c|c|c|c}  \hline
        Result & Corollary~\ref{cor:plugin-lr} & Proposition~\ref{prop:1s-z-ui} & Theorem~\ref{thm:ui-ttest} & Theorem~\ref{thm:ui-ttest-onesided}  \\ \hline \hline
        Problem &  \multicolumn{2}{c|}{Sequential Z-test} & \multicolumn{2}{c}{Sequential t-test} \\ \hline
        Null & $\normal{0}{\sigma^2}$ & $\normals_{\mu \le 0, \sigma^2}$ & $\normals_{\mu = 0}$ & $\normals_{\mu \le 0}$ \\ \hline
        Alternative & $\normals_{\mu \neq 0, \sigma^2 }$ & $\normals_{\mu > 0, \sigma^2 }$ & $\normals_{\mu \neq 0}$ & $\normals_{\mu > 0}$ \\ \hline
        \makecell{Test process; \\ class; \\ filtration} & \makecell{$\sigma^n \e^{ \frac{V_n}{2\sigma^2}  } \spadesuit_n$ \\ NM  \\ $\{\cF_n\}$ } & \makecell{ $\sigma^n \e^{ \frac{ V_n - n(\avgX{n} \wedge 0)^2 }{2\sigma} }\spadesuit_n$ \\ e-process \\ $\{\cF_n\}$ } & \makecell{ $\left(  \avgXsq{n} \right)^{n/2}  \e^{n/2} \spadesuit_n $ \\ e-process \\ $\{\cF_n\}$ } & \makecell{ $ \left( \avgXsq{n} - (\avgX{n}\wedge 0)^2 \right)^{n/2}   \e^{n/2} \spadesuit_n$ \\ e-process \\ $\{\cF_n\}$ } \\ \hline 
        Behavior under null & \multicolumn{4}{c}{\makecell{\revise{with probability $1-\alpha$: $1\to [0, 
 1/\alpha ]$} } } \\ \hline
        Behavior under alternative &  \multicolumn{4}{c}{$1\to\infty$ a.s. } \\ \hline
        e-power & $\mu^2/2\sigma^2$ & $(\mu 
\vee 0) ^2/2\sigma^2$ &  $\frac{1}{2}\log(1+\mu^2/\sigma^2)$   & $\frac{1}{2}\log(1+(\mu 
\vee 0)^2/\sigma^2)$
        \\ \hline
    \end{tabular}
    \caption{Comparison of universal inference-based Z-tests and t-tests, where $\spadesuit_n := \prod_{i=1}^n \frac{  \exp \left\{ -\frac{1}{2}\left( \frac{X_i - \hmu_{i-1}}{\hsig_{i-1}} \right)^2   \right\} }{\hsig_{i-1}}$. \revise{The arrow notation $x \to y$ in ``Behavior under null" and ``Behavior under alternative" means that the test process $M_n$ starts off from $M_0 = x$ and reaches $\lim_{n \to \infty} M_n = y$ (where $y$ is a number) or stays $M_n \in y$ (where $y$ is a set).} The ``e-power" of a process $M_n$ refers to the almost sure limit of $\frac{\log M_n}{n}$ under the true distribution $\normal{\mu}{\sigma^2}$.}
    \label{tab:big-comp-ui}
\end{table}

\begin{table}[!h]
    \centering  \small
    \begin{tabular}{c||c|c|c|c|c|c}  \hline  
       Result  &  WR23 Prop.\ 5.6 & WR23 Prop.\ 5.7 & WR23 Prop.\ 6.3 &  Theorem~\ref{thm:lai-e} &  Theorem~\ref{thm:lai-ensm} & Theorem~\ref{thm:si-onesided} \\
       \hline \hline
       Problem & \multicolumn{3}{c|}{Sequential Z-test} &  \multicolumn{3}{c}{Sequential t-test} \\ \hline
       Null &
       \multicolumn{2}{c|}{$\normal{0}{\sigma^2}$} &  $\normals_{\mu \le 0, \sigma^2 }$ & \multicolumn{3}{c}{$\normals_{\mu = 0}$} \\ \hline
       Alternative &
       \multicolumn{2}{c|}{$\normals_{\mu \neq 0, \sigma^2 }$} & $\normals_{\mu > 0, \sigma^2 }$ &  \multicolumn{2}{c|}{$\normals_{\mu \neq 0}$} & $\normals_{\mu > 0}$ \\ \hline
       \makecell{Base martingales; \\ parametrized by} &  \multicolumn{3}{c|}{ \makecell{Gaussian likelihood ratio; \\ $\mu$, the alternative mean }}& \multicolumn{3}{c}{ \makecell{Scale invariant likelihood ratio; 
 \\ $\theta = \mu/\sigma$, the alternative standardized mean }} \\ \hline
        Mixture & $\mu \sim \normal{0}{c^{-2}}$ & $\mu \sim F$ & $\mu \sim F^{>0}$ & $\theta \sim \normal{0}{c^{-2}}$ & $\theta \sim F$ & $\theta \sim \normal{0}{c^{-2}}^{>0}$  \\ \hline
        \makecell{Test process; \\ class; \\ filtration} & \makecell { $\sqrt{\frac{c^2 }{ n + c^2 }}\e^{ \frac{n^2  \avgX{n} ^2 }{2(n + c^2)\sigma^2} }$ \\ NM \\ $\{\cF_n\}$ } & \makecell { $\frac{1}{\sqrt{n}} \e ^{ \frac{n  \avgX{n}^2}{2\sigma^2} }$ \\ ENSM \\ $\{\cF_n\}$  } &   \makecell { $\frac{1}{\sqrt{4n}} V(\sqrt{n} \avgX{n}/\sigma ) $ \\ ENSM \\ $\{\cF_n\}$ } & \makecell {\tiny $\sqrt{\frac{c^2}{n+c^2}} \left( \frac{(n+c^2) V_n }{(n+c^2) V_n  - S_n^2}  \right)^{\frac{n}{2}}$  \\ NM \\ $\{\cF_n^*\}$ } & \makecell { $\sqrt{\frac{2 \pi }{n}} \left( \frac{n V_n }{n V_n - S_n^2}  \right)^{\frac{n}{2}}$ \\ ENSM \\ $\{\cF_n^*\}$ } & \makecell{$2 G_n^{(c)} - 2 G_n^{(c)} |_{(S_n \leftarrow S_n \wedge 0) }$ \\ e-process \\ $\{\cF_n^*\}$}  \\ \hline
        Behavior under null & $1\to0$ a.s. & \multicolumn{2}{c|}{$\infty \to 0$  a.s.} & $1\to0$  a.s. & $\infty \to 0$  a.s. & $1\to0$ a.s. \\ \hline
         \makecell{Behavior under \\ alternative} & $1\to \infty$ a.s. & \multicolumn{2}{c|}{$\infty \to \infty$  a.s.} & $1\to \infty$  a.s. & $\infty \to \infty$  a.s. & $1\to \infty$ a.s. \\ \hline
         e-power & \multicolumn{2}{c|}{$\mu^2/2\sigma^2$} & $(\mu 
\vee 0) ^2/2\sigma^2$ &  \multicolumn{2}{c|}{$\frac{1}{2}\log(1+\mu^2/\sigma^2)$}   & $\frac{1}{2}\log(1+(\mu 
\vee 0)^2/\sigma^2)$
          \\ \hline
    \end{tabular}
    \caption{Comparison of mixture-based Z-tests and t-tests. WR23 refers to the paper by \cite{ensm}, which also covers the \emph{sub}Gaussian case, but we omit it for a more direct comparison. The function $V$ is $V(x) = \e^{x^2/2} (1 + \erf(x/\sqrt{2}))$, and $G_n^{(c)}$ refers to the test process listed above under Theorem~\ref{thm:lai-e}.}
    \label{tab:big-comp}
\end{table}

\end{landscape}

\revise{

\subsection{Simulations on e-Processes}\label{sec:sim-ep}
We empirically demonstrate the pros and cons of the two major two-sided e-processes for $\normals_{\mu = 0}$ we present in this paper, Theorem~\ref{thm:ui-ttest} and~\ref{thm:lai-e}, under various underlying distributions and choices of hyperparameters that reflect prior knowledge. Our simulations contain two null data-generating distributions $\normal{0}{1}$ and $\normal{0}{4}$, as well as two alternative $\normal{0.1}{1}$ and $\normal{1}{4}$. For the method of Gaussian mixture of scale invariant likelihood ratios (Theorem~\ref{thm:lai-e}), the mixture precision parameters range from $c^2 \in \{ 1, 0.1, 50 \}$, corresponding to the prior distributions $\mu/\sigma\sim \normal{0}{1}$, $\normal{0}{10}$, and $\normal{0}{0.02}$. For universal inference method (Theorem~\ref{thm:ui-ttest}), the plug-in estimators for $\mu$ and $\sigma^2$ are chosen as the maximum a posterior (MAP) estimators via using a normal-inverse-gamma prior, the conjugate prior for Gaussian models with unknown $\mu$ and $\sigma^2$. The prior
\begin{equation}
    (\mu, \sigma^2) \sim N\Gamma^{-1}(\mu_0, \nu_0,\alpha_0, \beta_0),
\end{equation}
corresponds to estimating the mean through additional $\nu_0$ prior observations with mean $\mu_0$, and variance through additional $2\alpha_0$ observations with mean $\mu_0$ and sum of squared deviations $2\beta_0$. We choose $N\Gamma^{-1}(0, 20,10, 10)$ and $N\Gamma^{-1}(0, 20,10, 40)$ in our simulations, and also employ the idea of using the first 40 observations to fit the prior (which we call ``burning in") and start the e-process afterward.
The computed log-e-processes are plotted in \cref{fig:logeprs}, averaged over 100 independent repeats. While none exhibits any growth under nulls, they differ qualitatively under alternatives. 

Under the $\normal{0.1}{1}$ alternative, the ``confident" Gaussian prior $\normal{0}{0.02}$ leads in power when the true effect size $\mu/\sigma = 0.1$ is indeed small and falls around the unit standard deviation range of the prior, beating the ``underconfident" priors $\normal{0}{1}$ and $\normal{0}{10}$, and universal inference. Among universal inference instances, the one with a $N\Gamma^{-1}(0, 20,10, 10)$ prior which correctly guesses the true parameters, and the one that burns in to learn them are much more powerful than the other one with a ``wrong'' prior.

Under the $\normal{1}{4}$ alternative, the ``modest'' Gaussian prior  $\normal{0}{1}$ is advantageous, and the ``overconfident" $\normal{0}{0.02}$ is penalized for not adequately capturing the true effect size $\mu/\sigma = 0.5$. Among universal inference, this time the $N\Gamma^{-1}(0, 20,10, 40)$ prior is closer to the underlying true distribution so it leads in power, followed by the burned-in one, and followed in turn by the one with the ``wrong'' prior.

We thus conclude that in general universal inference is more sensitive to the quality of prior information, and a weak or biased prior likely degrades the power of the test. The technique of burning-in, as we show, partially helps the method obtain all-case adaptivity to any underlying distribution via the construction of a data-dependent prior, paying the price that the evidence does not accumulate during the burn-in phase. This is a wisdom that resonates with the ``split-sample" idea of the fixed-time universal inference \citep{wasserman2020universal}.

\begin{figure}[!t]
    \centering
    \includegraphics[width=\textwidth]{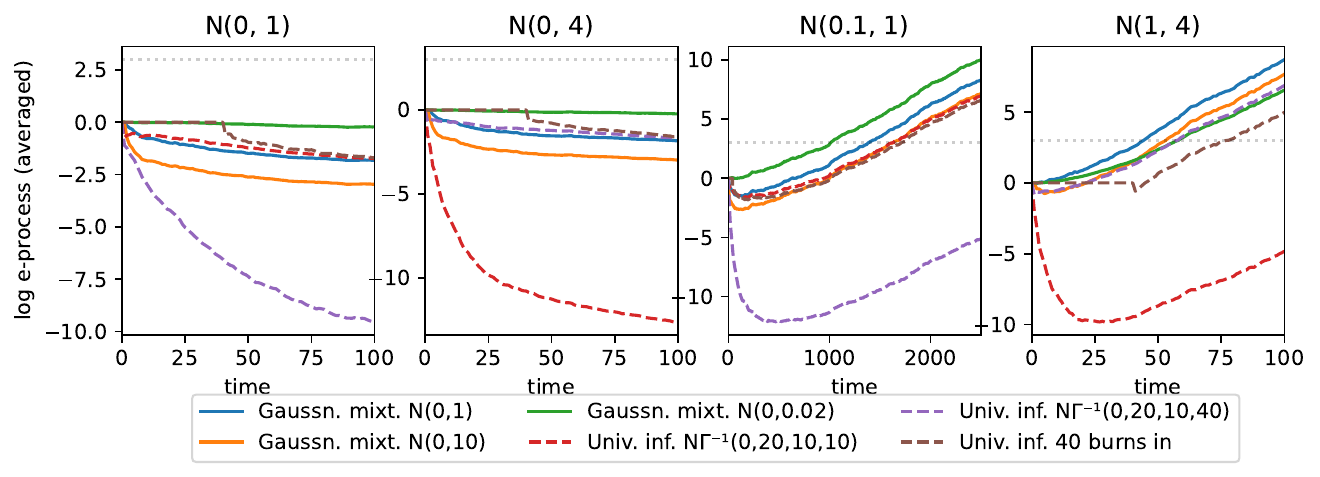}
    \caption{Logarithm of e-processes for the null $\normals_{\mu = 0}$, averaged over 100 independent repeats. Dotted grey horizontal lines represent the rejection threshold $\log 20$ for $\alpha = 0.05$. \revise{The Gaussian mixture method is less sensitive to the prior choice, whereas universal inference exhibits a higher sensitivity to the plug-in estimators. The burn-in technique helps universal inference improve worst-case performance.}}
    \label{fig:logeprs}
\end{figure}

\subsection{E-Processes on a Real Dataset}
We next apply our methods to a real public dataset by \citet[Data set 285]{hand1993handbook}, from a study on the effect of two treatments on anorexia. The dataset contains the weights in lbs before and after the treatment of 72 young female anorexia patients, divided into Control (26 subjects), Treatment I (``family treatment", 17 subjects), and Treatment II (``Cognitive Behavioural treatment", 29 subjects). This dataset is curated in the R package ``MASS'' under the identifier ``anorexia''.

For within-subject studies like this, the paired t-test (i.e.\ classical one-sample t-test on the differences of weights pre- and post-treatment) for $\normals_{\mu = 0}$ is commonly applied to each group testing the null that there was no effect on weight. We obtain classical t-test p-values of 0.776, 0.0007, and 0.035 respectively from the three groups. 

We then simulate observing these weight changes (in the original order of the dataset) sequentially and test the $\normals_{\mu = 0}$ null via our two e-process methods. For the Gaussian mixture method, we use the prior $\normal{0}{1}$ and the prior that minimizes the radius of the CS according to \eqref{eqn:optim-c}; and for universal inference, we use 2/3 of each dataset for burn-in due to the relatively small sample sizes. 

The evolution of these e-processes, and the final p-values are shown in \cref{fig:anorexia}. These p-values are computed using $p = 1 \wedge (\max_n M_n)^{-1}$ where $\max_n M_n$ is the maximum of the e-process, and it is clear that $P(p \le \alpha) \le \alpha$ under any null $P$. (This is an instance of the ``anytime-valid p-values'' due to \cite{johari2015always}.)
We observe that no false discoveries are made by either of the methods in the Control group. For Treatment I, the hypothetical sequential researcher can always make the rejection, sometimes earlier. In particular, we see that the width-minimizing prior for the Gaussian mixture method does not produce a larger final e-value. For Treatment II, however, neither of the methods makes any rejection even though the classical t-test provides a (borderline) significant p-value of 0.035. 
This experiment highlights that if one is unsure what the effect strength is and hence unsure if the sample size suffices, a sequential test is a much safer bet in that if the intended fixed sample size is $n$, one can hope to get away with less than $n$ observations if the effect is strong (Treatment I), while potentially collecting more than $n$ observations if the signal is weaker than expected (Treatment II).

\begin{figure}[!b]
    \centering
    \includegraphics[width=\linewidth]{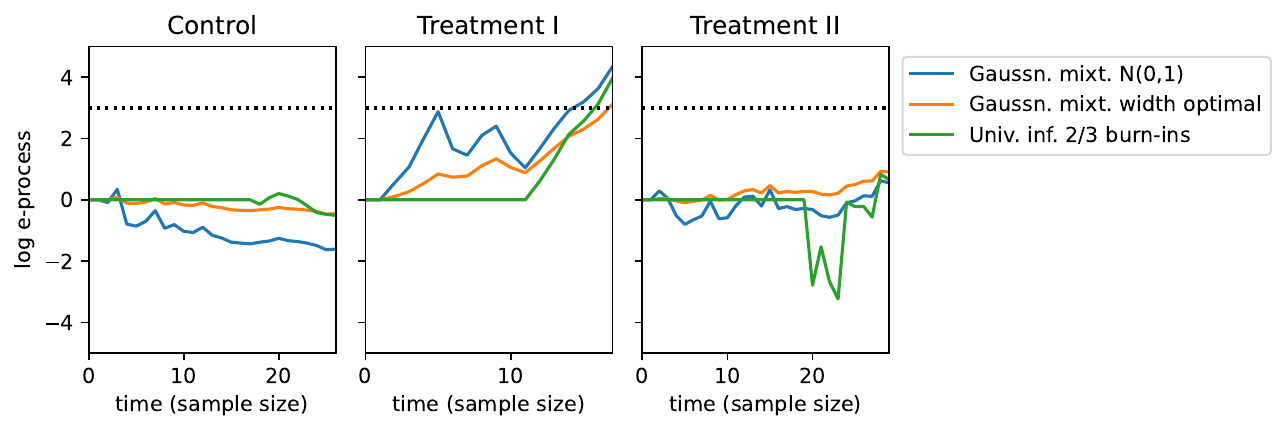}
    \vskip 1em
    \begin{tabular}{|c|c|c|c|}\hline
       Method  &  \makecell{Control \\ ($n=26$)}  & \makecell{Treatment I \\ ($n=17$)} & \makecell{Treatment II \\ ($n=29$)} \\
         \hline
      Classical t-test   &   0.776 & 0.0007* & 0.035* \\ \hline
      Gaussian mixture $\normal{0}{1}$ & 0.711 & \makecell{0.013* \\ ($\tau_{\text{rej}}=15$) } & 0.539 \\ \hline
      Gaussian mixture width optimal & 0.923 & \makecell{0.044* \\ ($\tau_{\text{rej}} = 17$)} & 0.397 \\ \hline
      Universal inference 2/3 burn-ins & 0.816 & \makecell{0.019* \\ ($\tau_{\text{rej} } = 16$)} & 0.436 \\
      \hline
    \end{tabular}
    \caption{E-processes testing the null of no effect over three groups. Dotted grey horizontal lines represent the rejection threshold $\log 20$ for $\alpha = 0.05$. In the table, we compare p-values obtained via different methods. Asterisks denote p-values smaller than 0.05, with the associated $\tau_{\text{rej}}$ denoting the rejection time when the e-process first reaches 20 (or p-value first reaches 0.05).}
    \label{fig:anorexia}
\end{figure}

}

\subsection{Simulations on Confidence Sequences}\label{sec:sim-epr}

\begin{figure}[!t]
    \centering
    \includegraphics[width=0.5\textwidth]{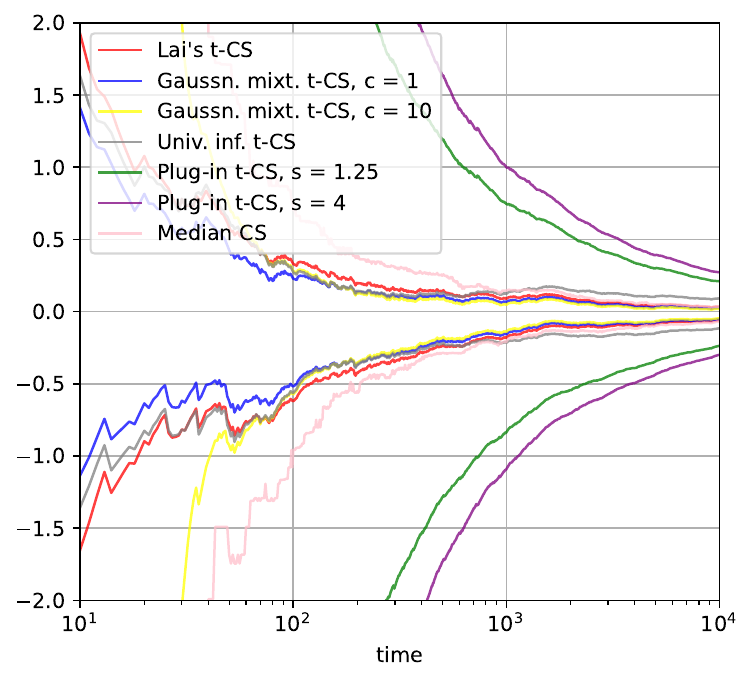}
    \caption{Five classes of confidence sequences for t-test under $\normal{0}{1}$ observations.}
    \label{fig:csplots}
\end{figure}

\begin{figure}[!t]
    \centering
    \includegraphics[width=0.5\textwidth]{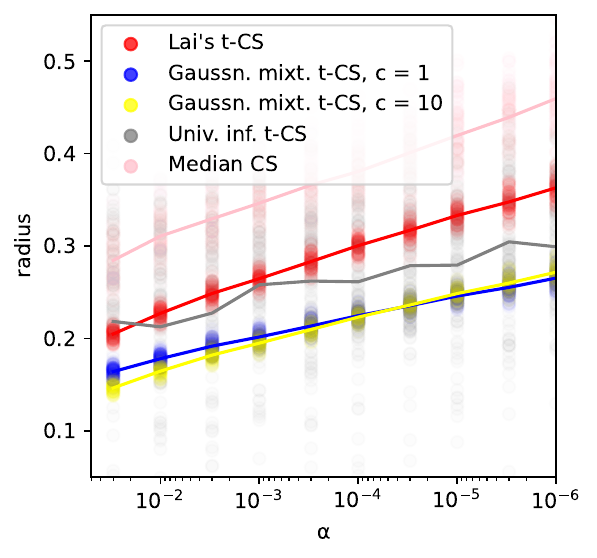}
     \includegraphics[width=0.49\textwidth]{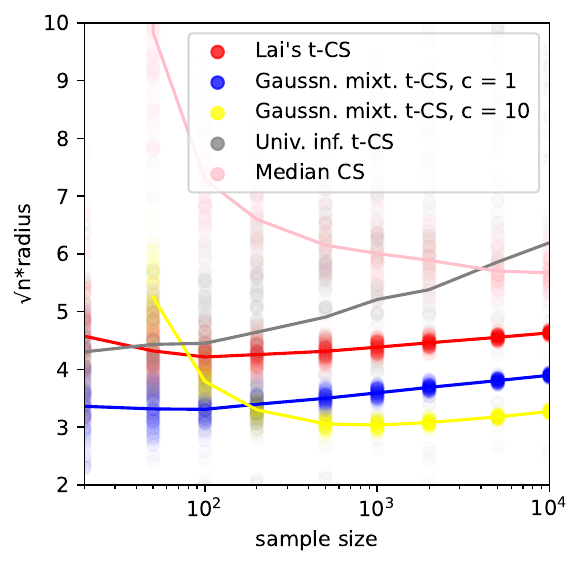}
    \caption{
    This pair of plots studies the behavior of width as one of $\alpha$ and $n$ vary, holding the other fixed.
    In the left plot, we plot widths of 3 CSs against $\alpha$ at $n=500$; whereas in the right plot, we plot widths multiplied by $\sqrt{n}$ against sample size $n$ at $\alpha=0.05$. Both are under $\normal{0}{1}$ observations, repeated \revise{100} times. Lines indicate averages over \revise{100} repeats.}
    \label{fig:rates}
\end{figure}

We demonstrate the three confidence sequences in the paper, as well as the additional ``plug-in" one, Theorem~\ref{thm:plug}, and the median CS, Proposition~\ref{prop:med}, by some experiments \revise{with the data-generating distribution set to $\normal{0}{1}$}. For the universal inference CS, we calculate the point estimators by putting a prior where we ``imagine" having seen pre-observations $-1$ and $1$ \revise{(i.e.\ a $N\Gamma^{-1}(0, 2, 1, 1)$ prior)}, and defining $\tilde \mu_i$ and $\tilde \sigma_i^2$ as the empirical means and variances posterior. For the Gaussian mixture of scale invariant likelihood ratios, we take $c^2$ to be 1 and 100.

First, we present them visually in the plot of \cref{fig:csplots} in a single run of i.i.d.\ standard normal observations. For the plug-in CS, we fix $\eta=0.5$ and let $s \in\{ 1.25, 4\}$ since $s$ is the parameter that controls the shrinkage rate.
We see that the Gaussian mixture CS we derived in \cref{sec:N-mix} with $c=1$ performs better than Lai's improper, flat mixture one; increasing $c$ results in interval explosion at earlier times but is slightly tighter at later times. The universal inference CS seems unfavorable at later times, but we shall soon see in repeated experiments that this is not necessarily the case. The plug-in CS by \cite{jourdan2023dealing}, however, is much looser than the rest even at $n=10000$ despite its asymptotically optimal iterated logarithmic rate, thus is excluded from the next simulation. A similar issue of large constant is observed in the iterated-logarithmic-rate median CS by 
\cite{quantile} as well.

We next compare their rates of growth and shrinkage. Still drawing observations from i.i.d.\ $\normal{0}{1}$, we first fix $n=500$ and let $\alpha$ vary, then fix $\alpha = 0.05$ and let $n$ vary, plotting the widths of the CSs in \revise{100} independent runs, since all of these CSs have random widths. The results are shown in \cref{fig:rates}. We see with greater clarity that Lai's CS is looser than the Gaussian mixture CS especially when $\alpha$ is small, even though the extended Ville's inequality seems a tighter and more advanced technique; larger precision $c^2$ in Gaussian mixture gradually gains an advantage as $n$ increases. Even after a ``smoothing" prior, the universal inference CS is still the most volatile. Its width varies significantly among runs. \revise{We also see that for the Gaussian mixture CS, larger values of the prior precision $c^2$ lead to tightness at later times, agreeing with the expression of the fixed-time near-optimal choice of $c^2$ in \eqref{eqn:optim-c}.}

\section{Optimality}

In this section, we show that the $\mathcal{O}(\alpha^{-1/n})$ or $\mathcal{O}(\alpha^{-1/(n-1)})$ growth rate (when $n$ is fixed and $\alpha \to 0$) in our t-confidence sequences shown in \cref{tab:t-cs}, and that the $\frac{1}{2}\log(1+\mu^2/\sigma^2)$ e-power in our t-test processes shown in \cref{tab:big-comp-ui,tab:big-comp}, are both optimal. These two kinds of optimality require different techniques, but the Kullback-Leibler divergence
\begin{equation}
    \kl( N_{\mu,\sigma^2} \| N_{0,\mu^2+\sigma^2} ) = \frac{1}{2}\log(1 + \mu^2/\sigma^2)
\end{equation}
is used in both cases.

\subsection{Information-Theoretic Lower Bound of t-Confidence Intervals}\label{sec:opt-CI}

Recall from \cref{tab:t-cs} that our best CSs have width scaling with $\alpha$ like the nonstandard rate $\alpha^{-1/n}$, which is very different from the rate of $\sqrt{\log(1/\alpha)}$ that one typically sees for confidence intervals and sequences when the variance is known. Indeed, we will see now that for t-CIs (and hence t-CSs), this $\alpha^{-1/n}$ rate is actually minimax optimal.



We define $\mathcal{L}_{\alpha, n}$ to be the set of all \emph{lower} $(1-\alpha)$-confidence bounds
for $\mu$ over $\normals$ with fixed sample size $n$, i.e., functions $L:\mathbb R^n \to \mathbb R$ such that
\begin{equation}
 \text{for all }P \in \normals,\quad   P(  \mu(P) \ge L(X_1,\dots, X_n)  ) \ge 1-\alpha.
\end{equation}

Fix a particular $\varepsilon > 0$ and $P \in \normals$. For any $L \in \mathcal{L}_{\alpha, n}$, denote by $W_{\varepsilon}(L, P)$ the $\varepsilon$-upper quantile of the random variable $\mu(P) - L(X_1,\dots, X_n)$ under $P$, i.e., 
\begin{equation}
    P(\mu(P) - L(X_1,\dots, X_n) \ge W_{\varepsilon}(L, P) ) = \varepsilon.
\end{equation}
Let us study the minimax rate
\begin{equation}
  M^{-}_{\alpha, n} =  \inf_{ L\in  \mathcal{L}_{\alpha, n}} \sup_{P \in \normals} \sigma^{-1}(P) \cdot W_{2\alpha}(L, P). 
\end{equation}
The definition above characterizes the following: the event that a lower CI covers $\mu$ with an \emph{excess} of $M^{-}_{\alpha, n}$ has probability at least $2\alpha$, which is an event disjoint from
the event that it fails to cover $\mu$ which has probability at most $\alpha$. As we shall soon see, this one-sided lower bound immediately implies a lower bound for two-sided CIs.
The two layers of $2\alpha$-low probability and $(1-\alpha)$-high probability are introduced to include confidence intervals with data-dependent radii, which, though uncommon in Z-tests, recur in t-tests throughout this paper. 

We establish the following lower bound for $ M^{-}_{\alpha, n}$:
\begin{theorem}[Minimax Lower Bound of Lower t-CIs]\label{thm:it-lb} If $\alpha < 1/3$,
    $ M^{-}_{\alpha, n}  \ge \sqrt{ ( 6\alpha - 9\alpha^2 )^{-2/n}  - 1 } $.
\end{theorem}
This lower bound is proved by employing a standard two-point information-theoretic argument, via a pair of $N_{\mu,\sigma^2}$ and $ N_{0,\mu^2+\sigma^2}$ such that the total variation distance between $N_{\mu,\sigma^2}^{\otimes n}$ and $ N_{0,\mu^2+\sigma^2}^{\otimes n}$ is at most $1-3\alpha$. The details are in \cref{sec:pf-opt}. 

Immediately following the one-sided lower bound, we can establish the following two-sided lower bound. 
Let $\mathcal{I}_{\alpha, n}$ be the set of all two-sided $(1-\alpha)$-confidence intervals
for $\mu$ over $\normals$ with fixed sample size $n$.
For any $\CI \in \mathcal{I}_{\alpha, n}$, denote by $W_{\varepsilon}(\CI, P)$ the $\varepsilon$-upper quantile of its width $|\CI|$ under $P$, and let
\begin{equation}
    M_{\alpha, n} = \inf_{ \CI \in  \mathcal{I}_{\alpha, n}} \sup_{P \in \normals} \sigma^{-1}(P) \cdot W_{\alpha}(\CI, P).
\end{equation}
Then, we have the following minimax lower bound, also proved in \cref{sec:pf-opt}:
\begin{corollary}[Minimax Lower Bound of t-CIs]\label{cor:it-lb2}If $\alpha < 1/3$, $M_{\alpha, n}    \ge \sqrt{ ( 6\alpha - 9\alpha^2 )^{-2/n}  - 1 }$.
\end{corollary}

If $n$ is fixed, the lower bound above is asymptotically $\mathcal{O}(\alpha^{-1/n})$, the growth rate attained by universal inference method (Theorem~\ref{thm:ui-ttest}), and almost attained by Gaussian mixture of scale invariant likelihood ratios method (Theorem~\ref{thm:lai-e}) in our paper.

\subsection{The Surprising $\alpha$-Suboptimality of the Classical t-Test}

While the lower bound $\mathcal{O}(\alpha^{-1/n})$ proved above is matched by one of the t-CSs we derive, we now come to the surprising fact that the confidence interval obtained by inverting the classical one-sample Student's t-test with a fixed sample size $n$ fails to attain the same optimality as a function of $\alpha$.
Under $\normal{\mu}{\sigma^2}$, the classical t-test relies on the following ``pivotal" quantity that follows the t-distribution of $n-1$ degrees of freedom:
\begin{equation}
T_{n-1}=\sqrt{n-1} \frac{S_n - n\mu}{ \sqrt {n V_n - S_n^2} },
\end{equation}
where $S_n$ and $V_n$ denote the sum of samples and the sum of sample squares respectively. The following, therefore, is a $(1-\alpha)$-confidence interval for $\mu$:
\begin{equation}
    \CI_n^{\mathsf{t}} = \left[ \avgX{n} \pm \sqrt{\frac{nV_n - S_n^2}{n^2 (n-1)}} \cdot t_{1-\alpha/2}(n-1)  \right],
\end{equation}
where $t_{1-\alpha/2}(n-1)$ denotes the $1-\alpha/2$ quantile of the student t-distribution with $n-1$ degrees of freedom. Since the survival function of the student t-distribution with $n-1$ degrees of freedom scales as $\mathcal{O}(x^{-(n-1)})$, $t_{1-\alpha/2}(n-1)$, therefore the width of  $ \CI_n^{\mathsf{t}}$ as well, scales as $\mathcal{O}(\alpha^{-1/(n-1)})$ when the sample is fixed and $\alpha \to 0$.

The $\alpha^{-1/(n-1)}$ growth rate of the classical t-CI is
the same as the Gaussian mixture CS, Theorem~\ref{thm:lai-e}, but 
worse than the $\alpha^{-1/n}$ optimal rate attained by the other of our t-CSs, i.e., Theorem~\ref{thm:ui-ttest} by universal inference. That is, if we fix a sample size $n$ and decrease $\alpha$, the CI reported by Theorem~\ref{thm:ui-ttest} will eventually be tighter than $ \CI_n^{\mathsf{t}}$ when $\alpha$ is sufficiently small, as is shown in \cref{fig:classical}.

\begin{figure}
    \centering
    \includegraphics[width=0.5\linewidth]{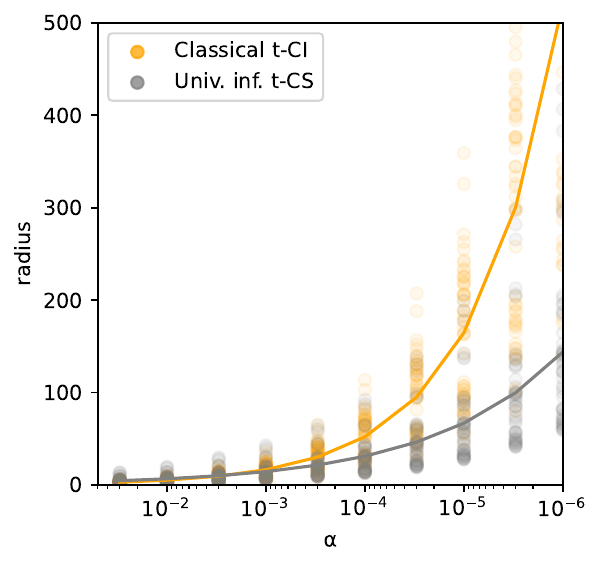}
    \caption{Growth rates of the t-CS by universal inference (Theorem~\ref{thm:ui-ttest}) and the classical t-test CI with $n=3$ observations.}
    \label{fig:classical}
\end{figure}

We explain this phenomenon from the following 
perspective.
    Consider the case with only $n=1$ observation $X_1$. The classical t-test is undefined; however, Theorem~\ref{thm:ui-ttest} implies the following valid e-value (taking $\hmu_{0} =0$ and $\hsig_{0} = 1$) for $\normal{\mu}{\sigma^2}$,
    \begin{equation}\label{eqn:1obevalue}
         | X_1 - \mu 
| \exp \left( \frac{1 -X_1^2}{2} \right).  
    \end{equation}
    To see that this is an e-value without invoking the full Theorem~\ref{thm:ui-ttest}, use the fact that $|x-\mu| \leq \sigma \exp((x-\mu)^2/2\sigma^2 - 1/2)$ for all $x,\sigma$ to conclude that
    \begin{align}
       & \EE_{\normal{\mu}{\sigma^2}} \left( | X_1 - \mu 
| \cdot \exp \left( \frac{1 -X_1^2}{2} \right)  \right) \le \\ &   \EE_{\normal{\mu}{\sigma^2}} \left( \sigma \exp\left(   \frac{(X_1-\mu)^2}{2\sigma^2} - \frac 1 2 \right) \cdot \exp \left( \frac{1 -X_1^2}{2} \right)  \right) 
 =  \EE_{\normal{\mu}{\sigma^2}} \left( \frac{\d \normal{0}{1}}{\d \normal{\mu}{\sigma^2}}(X_1) \right) = 1.
    \end{align}

   Applying Markov's inequality to the e-value above gives rise to a non-trivial CI for $\mu$:
    \begin{equation}\label{eqn:1obCI}
        \left[ X_1 \pm \alpha^{-1} \exp\left( \frac{X_1^2 - 1}{2} \right) \right]
    \end{equation}
    with  $\mathcal{O}(\alpha^{-1})$ growth rate. 
    It takes a second observation for the classical t-test to be non-trivial, thresholding the t-statistic $T_1 = \frac{(X_1 - \mu)+(X_2 - \mu)}{|X_1 - X_2|}$, which has a Cauchy distribution, resulting in a $\mathcal{O}(\alpha^{-1})$-CI that corresponds to the $\mathcal{O}(x^{-1})$ tail of Cauchy distribution.  While with $n=2$, Theorem~\ref{thm:ui-ttest} produces a better CI in terms of growth rate:
     taking just $\hmu_{0} =0,  \hmu_{1} = X_1$ and $\hsig_{0} = \hsig_{1} = 1$ for simplicity, the e-value reads,
        \begin{equation}
           \EE_{\normal{\mu}{\sigma^2}} \left[ \frac{(X_1-\mu)^2 + (X_2-\mu)^2}{2} \exp\left( 1 - \frac{X_1^2 + (X_2-X_1)^2}{2}   \right)  \right] \le 1.
        \end{equation}
        Applying Markov's inequality to the e-value above gives rise to a CI with $\mathcal{O}(\alpha^{-1/2})$ growth rate due to the quadratic dependence on $\mu$.

 {        \begin{remark}[On one-observation e-values] \normalfont
       E-values for $\normals_{\mu=\mu_0}$ via one single observation $X_1$ can be constructed by methods other than \eqref{eqn:1obevalue}. For example,
           in anticipation of the \emph{reverse information projection} to be introduced in the next subsection, and fully formulated in \cref{sec:pf-opt} by Proposition~\ref{prop:ripr}, the following likelihood ratio,
    \begin{equation}\label{eqn:num-mu-eval}
        B := \frac{\d \normal{0}{1}}{\d \normal{\mu_0}{\mu^2_0 + 1}}(X_1) =  \sqrt{\mu^2_0 + 1} \exp\left( \frac{(X_1 - \mu_0)^2}{2(\mu_0^2+1)}  -\frac{X_1^2}{2} \right),
    \end{equation}
    is an e-value for the entire set $\normals_{\mu=\mu_0}$, not only for $ \normal{\mu_0}{\mu^2_0 + 1}$, meaning that $\EE_P[B] \leq 1$ for any $P \in \normals_{\mu_0}$. \eqref{eqn:num-mu-eval}, however, does not lead to a closed-form CI via inversion. 
        \end{remark}}

        \begin{remark}[Related work on one-observation CIs] \normalfont
            \eqref{eqn:1obCI} is an instance of ``one-observation confidence intervals" for $\mu$ over $\normals$, a classical topic studied by numerous authors including Charles M.\ Stein, \cite{abbott1962two}, and \cite{portnoy2019invariance} who studies the lower bound in this problem\footnote{The paper by \cite{portnoy2019invariance} contains an important typo: the two $\Phi(\lambda(1+\tfrac{1}{c_1}))$ terms should be $\Phi(\lambda(1-\tfrac{1}{c_1}))$ instead, in the main result Theorem 1.1.}. These constructions typically look like $[X_1 \pm c_{\alpha} |X_1|]$, where, for example, $c_{0.05}\approx 10$, and are often tighter than \eqref{eqn:1obCI}. 
            Our CI \eqref{eqn:1obCI}, however, enjoys a simple closed-form expression that requires no numerical inversion, and is proved directly by the e-value \eqref{eqn:1obevalue} to provide intuition. All these CIs including our    \eqref{eqn:1obCI}, we note, contain 0  (or another fixed constant by shifting) regardless of the value of $X_1$.       
            These prior results, together with our \eqref{eqn:1obCI}, showcase the lesser-known suboptimality of the classical t-test CI under small samples; in particular, its unnecessary vacuity when $n=1$.
              {On a separate note, our one-observation CI \eqref{eqn:1obCI} can be \emph{randomly} and losslessly tightened due to a recent \emph{randomized} Markov's inequality by \citet[Theorem 1.2]{ramdas2023randomized}. This randomization halves the expected width of the CI. To elaborate, the randomized CI has the $1/\alpha$ term in \eqref{eqn:1obCI} replaced by a $U/\alpha$, where $U$ is an independent uniform random variable on $[0,1]$, which halves the expected width without changing its validity.}
        \end{remark} 
    

\subsection{Log-Optimality by Reverse Information Projection} \label{sec:ripr}

We now show that the recurring e-power $\frac{1}{2}\log(1+\mu^2/\sigma^2)$ in various test processes presented thus far is not a coincidence, as it is optimal in the problem of t-test.  {We study the e-power upper bound by focusing, again, on the case with a single observation $X_1$, via the \emph{expected} log-value under a fixed alternative of an e-value (called ``e-power" by \cite{vovk2022efficiency}, and ``growth rate"\footnote{To be disambiguated with the ``growth rate" of a confidence interval as $\alpha \to 0$ discussed in our paper.} by 
\cite{perez2022statistics}). If an e-value $M = M(X_1)$ for $\cP$ necessarily satisfies the e-power upper bound $\EE_Q (\log M) \le \delta(Q)$ for $Q\notin \cP$, then the e-power of any supermartingale for $\cP$ can at most be $\delta(Q)$ under $Q$ as well.

It is known that when testing the point null $P$, if the real distribution $Q \neq P$, the e-power of
an e-value $M$ is at most
\begin{equation}
     \EE_Q(\log M) \le \kl(Q\| P),
\end{equation}
with equality if and only if $M$ is the likelihood ratio $\frac{\d Q}{\d P}(X_1)$, i.e.\ the $L_1$ in Lemma~\ref{lem:lrm}. See e.g.\ \citet[Lemma 1]{vovk2022efficiency}. For example, in the Z-test, $P = \normal{0}{\sigma^2}$ and $Q = \normal{\mu}{\sigma^2}$, then the upper bound reads
\begin{equation}
     \kl( \normal{\mu}{\sigma^2} \| \normal{0}{\sigma^2}) = \frac{\mu^2}{2\sigma^2},
\end{equation}
matching those Z-test grow rates in \cref{tab:big-comp-ui,tab:big-comp}.

As a direct consequence, if we are to test a composite null $\cP$ and the real distribution $Q\notin \cP$, an e-value $M$ has e-power at most
\begin{equation}\label{eqn:ripr-kl}
     \EE_Q(\log M) \le \inf_{P \in \cP} \kl(Q\| P).
\end{equation}
With the t-test where $\cP = \normals_{\mu = 0}$ and $Q = \normal{\mu}{\sigma^2}$, a simple calculation shows the following.

\begin{proposition} \label{prop:ripr-simple} The probability measure $\normal{0}{\sigma^2+
\mu^2}$ attains the following minimum
\begin{equation}\label{eqn:riprmin}
   \kl( \normal{\mu}{\sigma^2} \| \normal{0}{\sigma^2+
\mu^2}) =  \min_{P \in  \normals_{\mu = 0}} \kl( \normal{\mu}{\sigma^2} \| P) = \frac{1}{2}\log(1+\mu^2/\sigma^2).
\end{equation}
Consequently, the e-power of a t-test e-value is at most $ \frac{1}{2}\log(1+\mu^2/\sigma^2)$ under $\normal{\mu}{\sigma^2}$.
\end{proposition}
This upper bound is matched by the e-powers of all two-sided t-test processes we derive throughout the paper, as displayed in \cref{tab:big-comp-ui,tab:big-comp}. We shall prove an extended version of Proposition~\ref{prop:ripr-simple} in \cref{sec:pf-opt}, as Proposition~\ref{prop:ripr}, where we discuss the minimization \eqref{eqn:riprmin} as an instance of ``reverse information projection", a topic studied recently by \cite{larsson2024numeraire} regarding the design of log-optimal e-values.}

\section{Concluding Remarks}

In this paper, we derive several sequential t-tests via universal inference and scale invariant likelihood ratio mixtures. In particular, we focus on the construction of test processes (nonnegative supermartingales, e-processes, and extended nonnegative supermartingales), and the confidence sequences they imply.

The two methods we explore yield similar results in numerous regards. They asymptotically share the same optimalities in terms of test processes and confidence sequences alike; they both extend to the one-sided test; and the computation of the test processes only requires linear time in both cases (unlike the quadratic cost that universal inference incurs worst-case, as mentioned by \citet[Section 7]{wasserman2020universal}).

One of the key differences between the two methods is the filtrations to which the processes are adapted, as the test processes scale invariant likelihood ratio mixtures produce are only adapted to a reduced filtration, allowing for a smaller set of stopping times that optionally stop the experiment. On the other hand, the universal inference test processes, while robust under a larger set of stopping times, sometimes suffer from lower stability depending on the quality of the estimators plugged in. We remark that the issue of filtration disappears once we convert tests into confidence sequences.

While we only focus on the class of one-dimensional normal distributions $\normals$ in this paper, we expect similar techniques may lead to sequential Hotelling's T-tests (with test processes and confidence \emph{sphere} sequences \citep{chugg2023time}) for the class of multivariate normal distributions. A slightly different way to generalize in dimensionality is the sequential F-tests for linear models, on which the most recent and comprehensive work by \cite{lindon2022anytime} only covers one of the numerous methods in our paper (namely, Gaussian mixture of scale invariant likelihood ratios). 
We leave these topics for future authors to study in depth.

%% file: t-test/ack.tex
The authors acknowledge support from NSF grants IIS-2229881 and DMS-2310718. AR thanks Ashwin Pananjady for useful conversations.

%% file: t-test/appendices.tex
\section{Additional Sequential t-Tests in the Literature}\label{sec:more-tcs}

\subsection{A Plug-in t-Confidence Sequence by \cite{jourdan2023dealing}}\label{sec:plugin}

In this section, we review the method by \citet[Appendix E]{jourdan2023dealing}, who first derive a confidence sequence for the variance $\sigma^2$ based on the sample variance $s_n^2$ over $\normals$ (using an example by \cite{howard2021time}), and then use a standard Z-confidence sequence via the method of likelihood mixture (e.g.\ Equation 10 of \cite{howard2021time}) that requires knowing the true variance $\sigma^2$. They then use a union bound to combine the two to obtain a rectangular confidence sequence for the pair $(\mu, \sigma^2)$ over $\normals$. A tighter confidence sequence can be obtained, we remark, if only the estimation of $\mu$ is of interest. 

Throughout the section, we follow the same notation as \cite{jourdan2023dealing} of $\overline{W}_i(x) = -W_i(-\exp(-x))$ where $W_0$, $W_{-1}$ are the Lambert $W$ functions, the inverses of $x\mapsto x \e^x$. It is not hard to see that
\begin{equation}
    \lim_{x\to\infty} \frac{\overline{W}_{-1}(x)}{x} = 1, \quad  \lim_{x\to\infty} \frac{ 1/\overline{W}_{0}(x) }{\exp(x)} = 1.
\end{equation}
We also let $\zeta$ be the Riemann Zeta function, which naturally arises in \emph{stitching} \citep{howard2021time}, a technique for obtaining law-of-the-iterated-logarithm rate confidence sequences. Recall that we denote by $s_n^2$ the sample variance without Bessel correction,
\begin{equation}
    s_n^2 = \frac{1}{n}\sum_{i=1}^n ( X_i - \avgX{n} )^2.
\end{equation}
Corollary 26 of \cite{jourdan2023dealing} states the following bound:

\begin{proposition}[Stitched Upper Confidence Sequence for Gaussian Variance]\label{prop:upper-cs-var} Let $s>1, \eta > 0$. The intervals
\begin{equation}\label{eqn:upper-cs-var}
    \left[ 0, \frac{s_n^2}{\overline{W_{0}}\left( 1 + \frac{2(1+\eta)}{n-1}\left( \log \frac{\zeta(s)}{\alpha} + s \log\left( 1 + \frac{\log(n-1)}{\log(1+\eta)} \right)  \right) \right) - \frac{1}{n-1}  } \right]
\end{equation}
with $n > n_{\min}$ where $n_{\min} = \polylog(1/\alpha)$ is the largest $n$ such that the denominator above is nonpositive,
form an $(1-\alpha)$-CS for $\sigma^2$ over $\normals$.
\end{proposition}

It is very worth noting that, since $1/\overline{W}_0 (x) \approx \exp(x)$ for large $x$, the confidence sequence \eqref{eqn:upper-cs-var} has polynomial, as opposed to the seemingly logarithmic dependence on $1/\alpha$.

For the mean $\mu$, Lemma 28 of \cite{jourdan2023dealing} quotes the following well-known bound:

\begin{proposition}[Stitched Normal Mixture Confidence Sequence for Gaussian Mean]\label{prop:stitched_normal}
Let $s> 1$. The intervals
\begin{equation}
    \left[ \avgX{n} \pm {\sigma}\sqrt{\frac{\overline{W}_{-1}\left( 1+ 2\log(1/\alpha) + 2 \log \zeta(s) + 2s(1-\log(2s)) + 2s \log(2s + \log n) \right)}{n }} \right]
\end{equation}
form an $(1-\alpha)$-CS for $\mu$ over $\normals_{\sigma^2}$.
\end{proposition}

A union bound between Propositions~\ref{prop:upper-cs-var} and~\ref{prop:stitched_normal} immediately implies the following t-confidence sequence, by ``plugging in" the variance upper confidence bound into the fixed-variance confidence sequence.

\begin{theorem}[Plug-In t-Confidence Sequence]\label{thm:plug}
    Let $s>1, \eta > 0$. The intervals
\begin{equation}\label{eqn:upper-cs-var}
    \left[ \avgX{n} \pm \sqrt{ \frac{s_n^2}{n} \cdot \frac{\overline{W}_{-1}\left( 1+ 2\log(2/\alpha) + 2 \log \zeta(s) + 2s(1-\log(2s)) + 2s \log(2s + \log n) \right)}{\overline{W_{0}}\left( 1 + \frac{2(1+\eta)}{n-1}\left( \log \frac{2\zeta(s)}{\alpha} + s \log\left( 1 + \frac{\log(n-1)}{\log(1+\eta)} \right)  \right) \right) - \frac{1}{n-1}  }} \right]
\end{equation}
with $n > n_{\min}$ where $n_{\min} = \polylog(1/\alpha)$ is the largest $n$ such that the denominator above is nonpositive,
form an $(1-\alpha)$-CS for $\mu$ over $\normals$.
\end{theorem}

When $n$ is fixed, the growth rate of the CS as $\alpha \to 0$ is $\Tilde{\mathcal{O}}\left(\sqrt{\exp(\frac{2(1+\eta)}{n-1}  \log \frac{1}{\alpha}) }\right)=\Tilde{\mathcal{O}}( \alpha^{-\frac{1+\eta}{n-1}})$, depending on the stitching parameter $\eta > 0$ chosen. This is less desirable compared to two of the $\mathcal{O}(\alpha^{-1/n})$-rate t-confidence sequences we derived in this paper; see \cref{tab:t-cs}.
When $\alpha$ is fixed, the shrinkage rate of the CS as $n\to\infty$ is the optimal iterated logarithm rate $\sqrt{n^{-1} \log\log n }$; however, such asymptotic benefit, overshadowed by the nonasymptotic price of large constants, does not manifest even with very large values of $n$, as is shown in the simulation \cref{fig:csplots}. Furthermore, the plug-in approach does not give a test process: though both of the two steps are based on some supermartingales, no union bound argument conjoins them, unlike the case of confidence sequences.
In practice, therefore, the plug-in approach by \cite{jourdan2023dealing} is less appealing than our universal inference and scale-invariant likelihood ratio mixture.

\subsection{A Sequential Test for Median by \cite{quantile}}\label{sec:med}

Since Gaussian distributions are symmetric, any confidence sequence for the median is a confidence sequence for the mean as well. We quote the following recent result from \citet[Theorem 1]{quantile}:

\begin{proposition}[Confidence Sequence for Median]\label{prop:med}
    Let
    \begin{equation}
      f_n = 0.75 \sqrt{\ell_n} + 0.8 \ell_n \quad   \ell_n = \frac{1.4 \log \log 2.1 n + \log(10/\alpha)}{n}.
    \end{equation}
    For any continuous distribution $P$ on $\mathbb R$,
    \begin{equation}
     \CI_n^{\mathsf{md}}  = [ \widehat{Q}_n(0.5 - f_n ), \widehat{Q}_n^-(0.5 + f_n )  ]
    \end{equation}
    forms a $(1-\alpha)$ confidence sequence for the median of $P$; therefore also a $(1-\alpha)$ confidence sequence for $\mu$ over $\normals$.
\end{proposition}
Here, the empirical quantiles $\widehat{Q}_n$ and $\widehat{Q}_n^-$ are defined as as
\begin{equation}
    \widehat{Q}_n(p) =  \sup \{ x : \widehat F_n(x) \le p \}, \quad  \widehat{Q}_n^-(p) =  \sup \{ x : \widehat F_n(x) < p \}.
\end{equation}
When $n$ is fixed and $\alpha \to 0$, we note, the CS eventually becomes the entire $\mathbb R$ as long as 
$f_n > 0.5$, which is true when
$\log(10/\alpha) > n/10$. If, on the other hand, letting $\alpha_{\star}(n)$ be the value of $\alpha$ such that $f_n$ is exactly $0.5$, then, as $\alpha \downarrow \alpha_{\star}(n) = \exp(-\widetilde{\Theta}(n))$,
\begin{equation}
    |\CI_n^{\mathsf{md}}| \asymp 2\sigma \Phi^{-1}(0.5 + f_n) = \widetilde{\Theta}(\sqrt{\log(1/(0.5 - f_n))}) = \widetilde{\Theta}\left(\sqrt{\log(1/(\alpha - \alpha_{\star}(n)))} \right)
\end{equation}
under $\normal{\mu}{\sigma^2}$ and large $n$ (where $\widetilde{\Theta}$ hides $\log \log(1/(\alpha-\alpha_{\star}(n)))$-order terms).

When $\alpha$ is fixed, the shrinkage rate of the CS as $n\to\infty$ also enjoys the optimal iterated logarithm rate $\sqrt{n^{-1} \log\log n }$ due to the stitching technique it employs; to be precise,
\begin{equation}
    \lim_{n\to\infty} \frac{ n \sigma^{-2}| \CI_n^{\mathsf{md}}|^2 - 4.5\pi \log(10/\alpha)}{\log \log n}  = 6.3 \pi , \quad \text{almost surely under $\normal{\mu}{\sigma^2}$},
\end{equation}
indicating that when $n$ and $1/\alpha$ grows \emph{simultaneously},
in the rate of $n =\widetilde{\Theta}(\log(1/\alpha))$ (where $\widetilde{\Theta}$ hides $\log \log(1/\alpha)$-order terms),
the width of the CS stays bounded, a property satisfied by all of those $\mathcal{O}(\alpha^{-1/n})$-width CSs mentioned so far as well.

Again, as is shown in \cref{fig:csplots,fig:rates}, this asymptotically optimal rate does not translate into practical advantages for sample size $< 10000$.


Proposition~\ref{prop:med} is based on the following test supermartingale, a Chernoff bound on i.i.d.\ Bernoulli indicator random variables due to \citet[Fact 1 (b)]{howard2020time}.
\begin{proposition}[One-Sided Test Supermartingale for Median]
    For any constant $\lambda\in\mathbb R$, the process
    \begin{equation}
        B_n^{\lambda} = \exp \left( \lambda \sum_{i=1}^n \id_{\{ X_i > 0 \}} - \frac{n \lambda}{2} - n\log\left(\frac{\exp(\lambda/2) + \exp(-\lambda/2)}{2} \right)  \right)
    \end{equation}
    is a supermartingale on $\{ \cF_n \}_{n \ge 0}$ for the set of all continuous distributions with median 0; in particular, for $\normals_{\mu = 0}$.
\end{proposition}

 It is straightforward to see that the growth of $\{  B_n^{\lambda}  \}$ under some $\normal{\mu}{\sigma^2}$ has the following e-power:
\begin{equation}\label{eqn:med-gi}
    \lim_{n\to \infty} \frac{\log(B_n^\lambda)}{n} =\lambda (\Phi(\mu/\sigma) - 1/2) -  \log\left(\frac{\exp(\lambda/2) + \exp(-\lambda/2)}{2} \right)  \quad  \text{almost surely}. 
\end{equation}
Here, $\Phi$ is the cumulative distribution function of $\normal{0}{1}$.
As is shown in \cref{fig:median-gi}, this is a very small e-power compared to the optimal $\frac{1}{2}\log(1+\mu^2/\sigma^2)$, and is powerful only against the alternative $\normals_{\mu > 0}$ when $\lambda > 0$. Therefore, \cite{quantile} consider different choices of $\lambda$ for different epochs of time (``stitching") to obtain Proposition~\ref{prop:med}, which is still much looser than many of the other CSs we derive.  Another way to gain universal power against all $\normals_{\mu\neq 0}$ is, again, via the method of mixture, and \cite{quantile} consider a beta-binomial mixture distribution on $\lambda$. Below, we let $\operatorname{B}(a,b)$ denote the beta function.

\begin{proposition}[Beta-Binomial Mixture Test Supermartingale for Median]\label{prop:betabinom} Let  $T_n = \sum_{i=1}^n \id_{\{X_i>0\}} - n/2$.
    For any $a,b>0$, the process 
    \begin{equation}
        B_n^{(a, b)} = \frac{\operatorname{B}(a + n/2 + T_n,b + n/2 - T_n)}{ (1/2)^{n}  \operatorname{B}(a,b)}
    \end{equation}
    is a supermartingale on $\{ \cF_n \}_{n \ge 0}$ for the set of all continuous distributions with median 0; in particular, for $\normals_{\mu = 0}$. Under $\normal{\mu}{\sigma^2}$, it has the following e-power:
\begin{equation}\label{eqn:med-gi-mix}
    \lim_{n\to \infty} \frac{\log(B_n^{(a,b)})}{n} = \log \left( \Phi(\mu/\sigma)^{\Phi(\mu/\sigma)} \Phi(-\mu/\sigma)^{\Phi(-\mu/\sigma)} \right)  - \log(1/2).
\end{equation}
\end{proposition}

\begin{figure}[!t]
    \centering
    \includegraphics[width=0.9\linewidth]{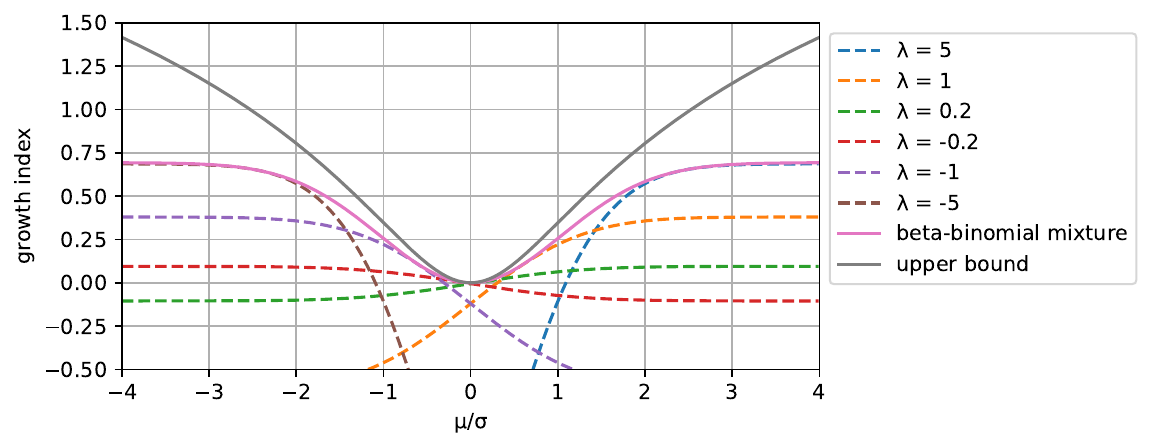}
    \caption{The e-powers of the one-sided median test supermartingale \eqref{eqn:med-gi} and mixture test supermartingale \eqref{eqn:med-gi-mix} under different choices of $\lambda$ and alternative effect sizes $\mu/\sigma$. The ``upper bound" refers to the optimal e-power $\frac{1}{2}\log(1+\mu^2/\sigma^2)$.}
    \label{fig:median-gi}
\end{figure}

This statement, proved in \cref{sec:pf-additional}, shows that the mixture does gain power over all $\mu \neq 0$ as is depicted in \cref{fig:median-gi}. However, compared to the optimal e-power $\frac{1}{2}\log(1+\mu^2/\sigma^2)$, it falls short in power especially when the effect size $\mu/\sigma$ is large.

\begin{remark} \normalfont
    The convergence \eqref{eqn:med-gi-mix} also holds in expectation (because the sequence is uniformly integrable). Following \citet[Section 5]{vovk2022efficiency}, we may define a concept of \emph{asymptotic relative efficiency} (ARE) of the test process $\{B_n^{(a,b)}\}$ here, as the ratio between the second derivatives at 0 of the e-power curves of the upper bound and $\{B_n^{(a,b)}\}$ in \cref{fig:median-gi}:
    \begin{align}
       & \operatorname{ARE}(B_n^{(a,b)}) = \lim_{ \theta \to 0 } \frac{\frac{1}{2}\log(1+\theta^2)}{\log \left( \Phi(\theta)^{\Phi(\theta)} \Phi(-\theta)^{\Phi(-\theta)} \right)  - \log(1/2)} 
        \\ =& \frac{\frac{\d^2}{\d \theta^2} \frac{1}{2}\log(1+\theta^2) |_{\theta=0} }{ \frac{\d^2}{\d \theta^2} \log \left( \Phi(\theta)^{\Phi(\theta)} \Phi(-\theta)^{\Phi(-\theta)} \right) |_{\theta=0} }
       = \sqrt{\pi/8} \approx 0.627.
    \end{align}
    This quantity bears the following meaning: fix any $\beta > 0$ and define by $n(\mu, \sigma^2)$ the sample size when $\EE_{\normal{\mu}{\sigma^2}} (\log B_n^{(a,b)})$ first reaches $\beta$, $n^*(\mu, \sigma^2)$ the sample size  when the upper bound of expected log value of e-processes, $ 
\frac{n}{2}\log(1+\mu^2/\sigma^2)$, first reaches $\beta$ under $\normal{\mu}{\sigma^2}$. Then, $   \operatorname{ARE}(B_n^{(a,b)}) = \lim_{\mu/\sigma \to 0} \frac{n(\mu, \sigma^2)}{n^*(\mu, \sigma^2)}$. 
In comparison, we note that in the fixed-time setting the classical Pitman ARE of the sign test (which Proposition~\ref{prop:betabinom} sequentializes) relative to the optimal t-test is $2/\pi \approx 0.637$ in the Gaussian case \citep[p.74]{sprent2007applied}. Note that the fixed-time sign test is exact, whereas the sequential Proposition~\ref{prop:betabinom} is conservative.

\end{remark}
%


\section{Universal Inference Z-Tests}\label{sec:uiz}

\subsection{Two-Sided Z-Test}

Recall from Corollary~\ref{cor:plugin-lr} that, for any $\mu$, $\sigma$ and their point estimators $\{ \hmu_n \}$, $\{ \hsig_n \}$ adapted to the canonical filtration $\{\mathcal{F}_n\}_{n \ge 0}$, the process
\begin{equation}
     \ell_n^{\mu, \sigma^2} = \frac{\prod_{i=1}^n  p_{\hmu_{i-1}, \hsig^2_{i-1}}(X_i)}{\prod_{i=1}^n  p_{\mu, \sigma^2}(X_i)} = \frac{\sigma^n}{\prod_{i=1}^n\hsig_{i-1}} \exp \left\{ \sum_{i=1}^n \left( \frac{(X_i - \mu)^2}{2\sigma^2} - \frac{(X_i - \hmu_{i-1})^2}{2\hsig^2_{i-1}} \right)  \right\}
\end{equation}
    is a martingale for $\normal{\mu}{\sigma^2}$ on $\{\mathcal{F}_n\}_{n \ge 0}$.

For the two-sided Z-test, we consider the point null $\normal{0}{\sigma^2}$ and the alternative $\normals_{\mu \neq 0, \sigma^2 }$. In this case, the variance $\sigma^2$ being a known quantity, the plug-in estimates $\{ \hsig_n^2 \}$ can be taken simply as $\sigma^2$. We have the following exponential growth asymptotic which is slightly faster than the $\frac{1}{2}\log(1 + \mu^2/\sigma^2)$ rate we see so far for t-tests. Both rates are optimal in their respective regimes due to our calculation in \cref{sec:ripr}.

\begin{proposition}[Asymptotic behavior of the universal inference Z-test e-process]\label{prop:div-ui-eproc-z}
    Under any $P = \normal{\mu}{\sigma^2}$, suppose there is a $\gamma > 0$ such that $\{ \hmu_n \}$ converges to $\mu$ in $L^2$ with rate $\EE_P (|\hmu_n  - \mu|^2) \lesssim n^{-\gamma} $, and $\hsig_n^{2} = \sigma^2$ for all $n$. Then,
    \begin{equation}
        \lim_{n \to \infty} \frac{ \log \ell_n^{0, \sigma^2}}{n} = \frac{\mu^2}{2\sigma^2}  \quad \text{almost surely}. 
    \end{equation}
    Consequently, $\{ \ell_n^{0, \sigma^2} \}$ diverges almost surely to $\ell^{0, \sigma^2}_\infty = \infty$ exponentially fast under $\normals_{\mu \neq 0}$.
\end{proposition}

\begin{proof}
    \begin{align}
        \frac{1}{n}\log \ell_n^{0, \sigma^2} =  \frac{1}{n} \sum_{i=1}^n \left( \frac{X_i ^2}{2\sigma^2} - \frac{(X_i - \hmu_{i-1})^2}{2\sigma^2} \right)  = \frac{1}{2\sigma^2} \overline{X_n^2}  - \frac{(X_i - \hmu_{i-1})^2}{2\sigma^2}.
    \end{align}
    Repeating the steps in the Proof of Proposition~\ref{prop:div-ui-eproc} in \cref{sec:pf}, we see that the first of these terms converges a.s.\ to $(\mu^2 + \sigma^2)/2\sigma^2$, and the last to $1/2$.
\end{proof}

\subsection{One-Sided Z-Test}

An infimum over $\mu \le 0$ yields the following one-sided Z-test.

\begin{proposition}[Universal inference one-sided Z-test e-process]\label{prop:1s-z-ui} For any $\mu$, $\sigma$ and their point estimators $\{ \hmu_n \}$, $\{ \hsig_n \}$ adapted to the canonical filtration $\{\mathcal{F}_n\}_{n \ge 0}$, the process
\begin{equation}
     \ell_n^{\sigma^2-} = \inf_{\mu \le 0}  \ell_n^{\mu, \sigma^2}  =   \frac{\sigma^n}{\prod_{i=1}^n\hsig_{i-1}} \exp \left\{ \sum_{i=1}^n \left( \frac{V_n - n(\avgX{n} \wedge 0)^2}{2\sigma^2} - \frac{(X_i - \hmu_{i-1})^2}{2\hsig^2_{i-1}} \right)  \right\} 
\end{equation}
    is an e-process for $\normal{\mu\le 0}{\sigma^2}$ on $\{\mathcal{F}_n\}_{n \ge 0}$.
\end{proposition}
The corresponding asymptotic can be similarly derived so we omit its proof.

\begin{proposition}[Asymptotic behavior of universal inference one-sided Z-test e-process]\label{prop:div-ui-eproc-z-1s}
    Under any $P = \normal{\mu}{\sigma^2}$, suppose there is a $\gamma > 0$ such that $\{ \hmu_n \}$ converges to $\mu$ in $L^2$ with rate $\EE_P (|\hmu_n  - \mu|^2) \lesssim n^{-\gamma} $, and $\hsig_n^{2} = \sigma^2$ for all $n$. Then,
    \begin{equation}
        \lim_{n \to \infty} \frac{ \log \ell_n^{\sigma^2-}}{n} = \frac{(\mu \vee 0) ^2}{2\sigma^2}  \quad \text{almost surely}. 
    \end{equation}
     Consequently, $\{ \ell_n^{\sigma^2 -} \}$ diverges almost surely to $\ell^{ \sigma^2 -}_\infty = \infty$ exponentially fast under $\normals_{\mu > 0}$.
\end{proposition}

\section{Omitted Proofs}
\label{sec:pf}

\subsection{Proofs for Likelihood Ratio Martingales}\label{sec:pf-lr}

\begin{proof}[Proof of Lemma~\ref{lem:lrm-general}] It suffices to prove that
\begin{equation}
    \EE_P\left( \frac{\d Q_n}{\d P}(X_n) \middle \vert X_1, \dots, X_{n-1} \right) = 1.
\end{equation}
For any $B \in \cB(\mathbb R)$, by the Doob–Dynkin factorization lemma (e.g.\ \citet[Corollary 1.97]{klenke2013probability}), since $\omega \mapsto Q_n(\omega, B)$ is $\cF_{n-1}$ measurable, there is a measurable $q(\cdot, B): \mathbb R^{n-1} \to \mathbb R$ such that  $q(X_1,\dots, X_{n-1},  B) = Q_n(\omega, B)$.

Let $A \subseteq \mathbb R^{n-1}$ be measurable. Then,
\begin{align}
    & \EE_{P}\left( \id_{\{ (X_1,\dots, X_{n-1}) \in A  \}} \cdot  \frac{\d Q_n}{\d P}(X_n) \right) \\
    = & \int_{A\times \mathbb R}  \frac{\d q_{n} (x_1,\dots, x_{n-1}, \cdot)}{\d P}(x_n)  \, \d P^{\otimes n}(x_1,\dots, x_n)
    \\
     = &  \int_A \left\{ \int_{ \mathbb R}  \frac{\d q_{n} (x_1,\dots, x_{n-1}, \cdot)}{\d P}(x_n) \, \d P(x_n)  \right\} \, \d P^{\otimes (n-1)}(x_1,\dots, x_{n-1})
    \\ 
    = &  \int_A  q_{n} (x_1,\dots, x_{n-1}, \mathbb R)  \, \d P^{\otimes (n-1)}(x_1,\dots, x_{n-1})
    \\
    = &  \int_A   \d P^{\otimes (n-1)}(x_1,\dots, x_{n-1}) =  P^{\otimes (n-1)}(A).
\end{align}
Since $A$ is arbitrary, we conclude that
\begin{equation}
    \EE_P\left( \frac{\d Q_n}{\d P}(X_n) \middle \vert X_1, \dots, X_{n-1} \right) = 1,
\end{equation}
which yields the desired result.
\end{proof}

\begin{proof}[Proof of Lemma~\ref{lem:lrm-joint}]
Let $A \subseteq \mathbb R^{n-1}$ be measurable. Then,
\begin{align}
    & \EE_{\mathbbmsl P}\left( \id_{\{ (X_1,\dots, X_{n-1}) \in A  \}} \cdot  \frac{\d \mathbbmsl Q_{(n-1)}}{\d\mathbbmsl P_{(n-1)}} (X_1,\dots, X_{n-1}) \right) \\
    = & \int_A  \frac{\d \mathbbmsl Q_{(n-1)}}{\d\mathbbmsl P_{(n-1)}} (x_1,\dots, x_{n-1}) \, \d \mathbbmsl P_{(n-1)}(x_1,\dots, x_{n-1}) = \mathbbmsl Q_{(n-1)}(A).
\end{align}
And
\begin{align}
    & \EE_{\mathbbmsl P}\left( \id_{\{ (X_1,\dots, X_{n-1}) \in A  \}} \cdot  \frac{\d \mathbbmsl Q_{(n)}}{\d \mathbbmsl P_{(n)}} (X_1,\dots, X_{n}) \right) \\
    = & \int_{A\times \mathbb R} \frac{\d \mathbbmsl Q_{(n)}}{\d \mathbbmsl P_{(n)}} (x_1,\dots, x_{n}) \, \d \mathbbmsl P_{(n)}(x_1,\dots, x_{n}) =  \mathbbmsl Q_{(n)}(A\times \mathbb R) =  \mathbbmsl Q_{(n-1)}(A).
\end{align}
Since $A$ is arbitrary, we conclude that
\begin{equation}
    \EE_{\mathbbmsl P}\left(  \frac{\d \mathbbmsl Q_{(n)}}{\d \mathbbmsl P_{(n)}} (X_1,\dots, X_{n}) \middle \vert X_1, \dots, X_{n-1} \right) = \frac{\d \mathbbmsl Q_{(n-1)}}{\d \mathbbmsl P_{(n-1)}}  (X_1,\dots, X_{n-1}).
\end{equation}
Further, this holds for $n=1$ as well in the degenerate sense that $\mathbb R^0 = \{ \cdot \}$ and both $\mathbbmsl P_{(0)}$ 
 and $\mathbbmsl Q_{(0)}$ charge the singleton with 1. This concludes the full proof.  
\end{proof}

\subsection{Proofs for Universal Inference t-Tests}\label{sec:pf-ui}

\begin{proof}[Proof of Theorem~\ref{thm:ui-ttest}]
For each $\sigma > 0$ consider the likelihood ratio martingale under $\mathcal{N}(\mu, \sigma^2)$
\begin{equation}
     M_n^{\mu, \sigma^2} = \frac{\prod_{i=1}^n  p_{\hmu_{i-1}, \hsig^2_{i-1}}(X_i)}{\prod_{i=1}^n  p_{\mu, \sigma^2}(X_i)}.
\end{equation}
The following process is an e-process under any $ \mathcal{N}(\mu, \cdot ) $
\begin{align}
    R_n^\mu = \frac{\prod_{i=1}^n  p_{\hmu_{i-1}, \hsig^2_{i-1}}(X_i)}{\sup_{s>0} \prod_{i=1}^n  p_{\mu, s^2}(X_i)} = \inf_{s > 0}  M_n^{\mu,s^2}.
\end{align}
Note that
\begin{equation}\label{eqn:ville-type}
    \sup_{\sigma>0} \Prw_{\mathcal{N}(\mu, \sigma^2)}[ \exists n, R_n^\mu \ge 1/\alpha  ] \le \Prw_{\mathcal{N}(\mu, \sigma^2)}[ \exists n, M_n^{\mu,\sigma^2} \ge 1/\alpha  ] \le \alpha.
\end{equation}
Since $s= \sqrt{\frac{\sum (X_i - \mu)^2}{n}}$ maximizes the denominator of $R_n^\mu$, it has a closed-form expression:
\begin{align}
     & R_n^\mu 
     = \frac{ \prod \frac{1}{\hsig_{i-1}} \exp \left\{ -\frac{1}{2}\left( \frac{X_i - \hmu_{i-1}}{\hsig_{i-1}} \right)^2  \right\} }{ \sup_{s > 0} \prod \frac{1}{s} \exp \left\{ -\frac{1}{2}\left( \frac{X_i - \mu}{s} \right)^2  \right\}  }
     =  \frac{ \prod \frac{1}{\hsig_{i-1}} \exp \left\{ -\frac{1}{2}\left( \frac{X_i - \hmu_{i-1}}{\hsig_{i-1}} \right)^2  \right\} }{ \left(  \frac{\sum (X_i - \mu)^2}{n} \right)^{-n/2}  \exp\left\{-\frac{n}{2}\right\} } .
\end{align}
Hence due to \eqref{eqn:ville-type}, with probability at least $1-\alpha$, for all $n$,
\begin{align}
     \frac{\sum (X_i - \mu)^2}{n} \le & \left\{ \ \frac{1}{\alpha} \exp\left\{-\frac{n}{2}\right\}  \prod {\hsig_{i-1}} \exp \left\{ \frac{1}{2}\left( \frac{X_i - \hmu_{i-1}}{\hsig_{i-1}} \right)^2  \right\}  \ \right\}^{2/n}  
     \\
     = &  \frac{1}{\alpha^{2/n} \e} \prod_{i=1}^n  {\hsig_{i-1}^{2/n}} \exp \left\{ \frac{1}{n}\left( \frac{X_i - \hmu_{i-1}}{\hsig_{i-1}} \right)^2  \right\} 
     \\
     = & \underbrace{ \frac{1}{\alpha^{2/n} \e}  \exp \left\{ \frac{\sum_{i=1}^n \log \hsig_{i-1}^2 + \left( \frac{X_i - \hmu_{i-1}}{\hsig_{i-1}} \right)^2 }{n} \right\} }_{T_n}
\end{align}
Recall that the random variable on the RHS is $T_n$. The UI-CS then reads
\begin{equation}\label{eqn:ui-cs}
    \CI_n = \left[  \avgX{n}  \pm  \sqrt{ \avgX{n}^2 - \avgXsq{n} + T_n } \right].
\end{equation}
\end{proof}

\begin{proof}[Proof of Proposition~\ref{prop:div-ui-eproc}] 
\revise{Let $P$ be a distribution with mean $\mu$ and variance $\sigma^2$.}
Note that
\begin{equation}
    \frac{1}{n}\log R_n = \frac{1}{2}\log {\avgXsq{n}} + \frac{1}{2} - \frac{ \sum_{i=1}^n \log \hsig_{i-1}  }{n} - \frac{\sum_{i=1}^n \left( \frac{X_i - \hmu_{i-1}}{\hsig_{i-1}} \right)^2 }{2n}.
\end{equation}
The first three terms converge a.s.\ to $\frac{1}{2}\log (\mu^2 + \sigma^2)$, $1/2$, and $\log \sigma$. It remains to show that the last term above converges a.s.\ to $1/2$.


Define $A_n = \sum_{i=1}^n \frac{1}{i} \left|  \left( \frac{X_i - \mu}{\hsig_{i-1}} \right)^2 - \left( \frac{X_i - \mu}{\sigma} \right)^2 \right|$. Then,

\begin{align}
   & \EE_P(A_n) =   \sum_{i=1}^n \frac{1}{i} \EE_P \left( |X_i - \mu|^2 \cdot| \hsig_{i-1}^{-2} - \sigma^{-2}  | \right)
\\
\le &  \sum_{i=1}^n \frac{1}{i} \EE_P\left( \frac{ (X_i - \mu)^4 i^{-\gamma/2} + (\hsig_{i-1}^{-2} - \sigma^{-2}  )^2 i^{\gamma/2} }{2} \right)
\\
= &  \frac{1}{2} \sum_{i=1}^n \frac{1}{i} \left( \EE_P(  X_i -  \mu )^4 \cdot i^{-\gamma/2} + \EE_P( \hsig_{i-1}^{-2} - \sigma^{-2} )^2 \cdot i^{\gamma/2} \right)
\\
= &  \frac{1}{2} \sum_{i=1}^n \frac{1}{i} ( \mathcal{O}(1) \cdot i^{-\gamma/2} + \mathcal{O}(i^{-\gamma}) \cdot i^{\gamma/2} ) = \frac{1}{2}\sum_{i=1}^n \mathcal{O}(i^{-1-\gamma/2}) = \mathcal{O}(1).
\end{align}
So there exists a constant $K > 0$ such that $\EE_P(A_n) \le K$ for all $n$. By Markov's inequality and Fatou's lemma, 
\begin{equation}
    P( \lim_{n \to \infty} A_n \ge a ) \le \frac{\EE_P( \lim_{n \to \infty} A_n )}{a} \le \frac{\liminf_{n \to \infty}  \EE_P( A_n )}{a} \le \frac{K}{a},
\end{equation}
which implies $P( \lim_{n \to \infty} A_n < \infty ) = 1$. By Kronecker's lemma, this implies that
\begin{equation}\label{eqn:as1}
  \lim_{n \to \infty}  \frac{ \sum_{i=1}^n \left\{ \left( \frac{X_i - \mu}{\hsig_{i-1}} \right)^2 - \left( \frac{X_i - \mu}{\sigma} \right)^2  \right\} }{n} = 0 \quad \text{almost surely.}
\end{equation}

We now define $B_n =  \sum_{i=1}^n \frac{1}{i} \left|  \left( \frac{X_i - \hmu_{i-1}}{\hsig_{i-1}} \right)^2 - \left( \frac{X_i - \mu}{\hsig_{i-1}} \right)^2 \right|$. Without loss of generality assume $\sup_{n \ge 0} \EE_P (\hsig_n^{-6}) < \infty $ (otherwise, if only $\sup_{n \ge n_0} \EE_P (\hsig_n^{-6}) < \infty $ holds, use summation $\sum_{i=n_0+1}^n$ instead). Then,
\begin{align}
   & \EE_P(B_n) =  
 \sum_{i=1}^n \frac{1}{i} \EE_P \left( \hsig^{-2}_{i-1} \cdot  | 2X_i - \hmu_{i-1} - \mu  | \cdot | \hmu_{i-1} - \mu_1 |  \right)
 \\
 & \le \sum_{i=1}^n \frac{1}{i} \EE_P \left( \frac{ \hsig^{-6}_{i-1} i^{-\gamma/4} +  | 2X_i - \hmu_{i-1} - \mu  |^3 i^{-\gamma/4} + | \hmu_{i-1} - \mu |^3 i^{\gamma/2}  }{3}  \right)
 \\
 & = \frac{1}{3}  \sum_{i=1}^n \frac{1}{i} \left( \EE_P(\hsig^{-6}_{i-1} ) \cdot i^{-\gamma/4} + \EE_P (| 2X_i - \hmu_{i-1} - \mu  |^3) \cdot i^{-\gamma/4} + \EE_P (| \hmu_{i-1} - \mu |^3 ) \cdot i^{\gamma/2}  \right)
 \\
 & = \frac{1}{3}  \sum_{i=1}^n \frac{1}{i} \left( \mathcal{O}(1) \cdot i^{-\gamma/4} + \mathcal{O}(1) \cdot i^{-\gamma/4} + \mathcal{O}(i^{-\gamma}) \cdot i^{\gamma/2}  \right) = \frac{1}{3}\sum_{i=1}^n \mathcal{O}(i^{-1-\gamma/4}) = \mathcal{O}(1).
\end{align}
Similarly by Markov's inequality, Fatou's and Kronecker's lemma,
\begin{equation}\label{eqn:as2}
   \lim_{n \to \infty} \frac{\sum_{i=1}^n \left\{ \left( \frac{X_i - \hmu_{i-1}}{\hsig_{i-1}} \right)^2 - \left( \frac{X_i - \mu}{\hsig_{i-1}} \right)^2 \right\} }{n} = 0  \quad \text{almost surely.}
\end{equation}
Finally, strong law of large numbers implies that
\begin{equation}\label{eqn:as3}
    \lim_{n \to \infty} \frac{ \sum_{i=1}^n \left(\frac{X_i - \mu}{\sigma}\right)^2 }{n} = 1 \quad \text{almost surely.}
\end{equation}
Combining \eqref{eqn:as1}, \eqref{eqn:as2}, and \eqref{eqn:as3}, we see that
\begin{equation}
   \lim_{n \to \infty} \frac{\sum_{i=1}^n \left( \frac{X_i - \hmu_{i-1}}{\hsig_{i-1}} \right)^2 }{n} = 1 \quad \text{almost surely.}
\end{equation}
This concludes the proof.
\end{proof}

\begin{proof}[Proof of Theorem~\ref{thm:ui-ttest-onesided}]
The following process is an e-process under any $ \mathcal{N}(\mu, \sigma^2) $, ($\mu \le 0$ and $\sigma > 0$).
\begin{align}
    \frac{\prod_{i=1}^n  p_{\hmu_{i-1}, \hsig^2_{i-1}}(X_i)}{\sup_{\substack{ m \le 0 \\  s > 0 }} \prod_{i=1}^n  p_{m, s^2}(X_i)} = \inf_{\substack{ m \le 0 \\  s > 0 }}  M_n^{m,s^2}.
\end{align}
We have
\begin{align}
     & \frac{ \prod \frac{1}{\hsig_{i-1}} \exp \left\{ -\frac{1}{2}\left( \frac{X_i - \hmu_{i-1}}{\hsig_{i-1}} \right)^2  \right\} }{ \sup_{m \le 0}\sup_{s > 0} \prod \frac{1}{s} \exp \left\{ -\frac{1}{2}\left( \frac{X_i - m}{s} \right)^2  \right\}  }
     =  \frac{ \prod \frac{1}{\hsig_{i-1}} \exp \left\{ -\frac{1}{2}\left( \frac{X_i - \hmu_{i-1}}{\hsig_{i-1}} \right)^2  \right\} }{ \sup_{m \le 0}\left(  \frac{\sum (X_i - m)^2}{n} \right)^{-n/2}  \exp\left\{-\frac{n}{2}\right\} }
     \\
     & =  \frac{ \prod \frac{1}{\hsig_{i-1}} \exp \left\{ -\frac{1}{2}\left( \frac{X_i - \hmu_{i-1}}{\hsig_{i-1}} \right)^2  \right\} }{ \sup_{m \le 0}\left( m^2 - 2 \avgX{n} m + \avgXsq{n} \right)^{-n/2}  \exp\left\{-\frac{n}{2}\right\} }
\end{align}
The denominator is maximized by $m =  \avgX{n} \wedge 0$. Therefore
\begin{align}
     \dots  =& \frac{ \prod \frac{1}{\hsig_{i-1}} \exp \left\{ -\frac{1}{2}\left( \frac{X_i - \hmu_{i-1}}{\hsig_{i-1}} \right)^2  \right\} }{ \left( ( \avgX{n} \wedge 0 - \avgX{n})^2 + \avgXsq{n} - \avgX{n}^2 \right)^{-n/2}  \exp\left\{-\frac{n}{2}\right\} } = \frac{ \prod \frac{1}{\hsig_{i-1}} \exp \left\{ -\frac{1}{2}\left( \frac{X_i - \hmu_{i-1}}{\hsig_{i-1}} \right)^2  \right\} }{ \left( (0\vee \avgX{n})^2 + \avgXsq{n} - \avgX{n}^2 \right)^{-n/2}  \exp\left\{-\frac{n}{2}\right\} }
     \\
     =& \frac{ \prod \frac{1}{\hsig_{i-1}} \exp \left\{ -\frac{1}{2}\left( \frac{X_i - \hmu_{i-1}}{\hsig_{i-1}} \right)^2  \right\} }{ \left( \avgXsq{n} - (\avgX{n}\wedge 0)^2 \right)^{-n/2}  \exp\left\{-\frac{n}{2}\right\} } = R_n^-.
\end{align}
\end{proof}

\begin{proof}[Proof of Proposition~\ref{prop:div-ui-eproc-1s}] The four terms of
\begin{align}
    \frac{1}{n}\log R_n^- = \frac{1}{2}\log\left(\avgXsq{n} - (\avgX{n}\wedge 0)^2 \right) + \frac{1}{2} - \frac{ \sum_{i=1}^n \log \hsig_{i-1}  }{n} - \frac{\sum_{i=1}^n \left( \frac{X_i - \hmu_{i-1}}{\hsig_{i-1}} \right)^2 }{2n}
\end{align}    
converges a.s.\ to $\frac{1}{2}\log( \mu^2 + \sigma^2 -   (\mu\wedge 0)^2  )$, $1/2$, $\log \sigma$, and $1/2$ respectively, concluding the proof.
\end{proof}

\subsection{Proofs for Scale Invariant t-Tests}\label{sec:pf-si}

\begin{proof}[Proof of Lemma~\ref{lem:jeffreys}]
Let $Q^+_n$ be the distribution of $(X_1, X_2, \dots, X_n)$ conditioned on $\{ X_1 > 0 \}$.
Define the constant $p_{> 0} = \Pr[X_1 > 0] = P_{\mu, \sigma}\{(0,\infty) \}$.
Then,
\begin{equation}
    \frac{\d Q^+_n}{\d \eta^{\otimes n}}(x_1,x_2,\dots, x_n) = p_{> 0}^{-1} \cdot  \id(x_1 > 0)  \cdot \prod_{i=1}^n \left\{ \sigma^{-1} g(\sigma^{-1}(x_i-\mu))  \right\},
\end{equation}
meaning that for any measurable $A\subseteq \mathbb{R^+} \times \mathbb R^{n-1}$,
\begin{equation}
    Q_n^+ \{ A \} =  \int_A p_{> 0}^{-1} \cdot \prod_{i=1}^n \left\{ \sigma^{-1} g(\sigma^{-1}(x_i-\mu))  \right\} \eta(\d x_1) \eta(\d x_2) \dots \eta (\d x_n).
\end{equation}
Now recall $\varphi_n:(x_1, x_2, \dots, x_n) \mapsto ( x_1/|x_1|, x_2/|x_1|, \dots, x_n/|x_1| )$. Let $Q_n^{*+}$ be the push-forward measure of $Q_n^+$ by $\varphi_n$, a measure on $\{1\}\times \mathbb R^{n-1}$ (which is the conditional distribution of $(X_1/|X_1|,\dots, X_n/|X_1|)$ given $\{ X_1 > 0 \}$).  Suppose the set $V_n(A) = \{1\} \times B$. We have the following change of variables (letting $(x_1^*, \dots, x_n^*) = V_n(x_1, \dots, x_n)$, so $x_1^* = 1$ below),
\begin{align}
   & Q_n^{*+} \{ V_n(A) \} =  Q_n^+\{A\}= \int_A p_{> 0}^{-1} \cdot \prod_{i=1}^n \left\{ \sigma^{-1} g(\sigma^{-1}(x_i-\mu))  \right\} \eta(\d x_1) \eta(\d x_2) \dots \eta (\d x_n)
   \\
      & = p_{> 0}^{-1} \cdot \int_{x_1 > 0} \eta(\d x_1) \int_{B}  \prod_{i=1}^n \left\{ \sigma^{-1} g(\sigma^{-1}(x_i^* |x_1| -\mu))  \right\}  \eta(\d x_2^*) \dots \eta (\d x_n^*) \left|\frac{ \d(x_2, \dots, x_n)}{ \d(x_2^*, \dots, x_n^*)}\right| 
    \\
    &= p_{> 0}^{-1} \cdot\int_{x_1 > 0} \eta(\d x_1) \int_{B}  \prod_{i=1}^n \left\{ \sigma^{-1} g(\sigma^{-1}(x_i^* |x_1| -\mu))  \right\}  \eta(\d x_2^*) \dots \eta (\d x_n^*) \begin{vmatrix} 
\frac{1}{|x_1|} & 0 & \dots & 0 \\
0 & \frac{1}{|x_1|} &  & 0\\
\vdots &  & \ddots & 0 \\
0 & \dots & 0  & \frac{1}{|x_1|} 
\end{vmatrix}^{-1}
\\
&= p_{> 0}^{-1} \cdot\int_{x_1 > 0} \eta(\d x_1) \int_{B}  \prod_{i=1}^n \left\{ \sigma^{-1} g(\sigma^{-1}(x_i^* x_1 -\mu))  \right\}  x_1^{n-1} \eta(\d x_2^*)  \dots \eta (\d x_n^*).
\end{align}
We thus conclude that (letting $\delta_1$ be the Dirac point mass on 1),
\begin{equation}
    \frac{\d Q_n^{*+}}{\delta_1 \otimes \eta^{\otimes (n-1)}}(x_1^*, x_2^*, \dots, x_n^*) =  p_{> 0}^{-1} \cdot\int_{x_1 > 0} \prod_{i=1}^n \left\{ \sigma^{-1} g(\sigma^{-1}(x_i^* x_1 -\mu))  \right\}  x_1^{n-1} \eta(\d x_1).
\end{equation}
We can simplify the integral above by substituting $\tau = \sigma/x_1$, obtaining,
\begin{align}
     & \frac{\d Q_n^{*+}}{\delta_1 \otimes \eta^{\otimes (n-1)}}(x_1^*, x_2^*, \dots, x_n^*)
     \\
     = &  p_{> 0}^{-1}\cdot \int_{\tau > 0}  \left\{ \sigma^{-1} g(1/\tau-\sigma^{-1}\mu) \right\} \left\{\prod_{i=2}^n \sigma^{-1} g(x_i^*/\tau-\sigma^{-1}\mu)) \right\}   (\sigma/\tau)^{n-1} \sigma (1/\tau^2) \eta(\d \tau)
     \\
     = & p_{> 0}^{-1}\cdot \int_{\tau > 0}  \left\{\prod_{i=1}^n (1/\tau) g(x_i^* /\tau - \mu/\sigma)) \right\} (1/\tau)  \d \tau.
\end{align}

Similarly, let $Q_n^{*-}$ be the conditional distribution of $(X_1/|X_1|,\dots, X_n/|X_1|)$ given $\{ X_1 < 0 \}$). We have
\begin{equation}
      \frac{\d Q_n^{*-}}{\delta_{-1} \otimes \eta^{\otimes (n-1)}}(x_1^*, x_2^*, \dots, x_n^*)
     = (1-p_{> 0})^{-1}\cdot \int_{\tau > 0}  \left\{\prod_{i=1}^n (1/\tau) g(x_i^* /\tau - \mu/\sigma)) \right\} (1/\tau)  \d \tau.
\end{equation}
Finally, the measure $Q_n^* = \id_{ \{ x_1 > 0 \} } \cdot p_{> 0 } \cdot Q_n^{*+} +  \id_{ \{ x_1 < 0 \} } \cdot (1-p_{> 0 }) \cdot Q_n^{*-}$. So we see that
\begin{equation}
    \frac{\d Q_n}{\d \eta^{\langle n-1 \rangle}} (x_1^*, \dots, x_n^*) =  \int_{\tau > 0}  \left\{\prod_{i=1}^n (1/\tau) g(x_i^* /\tau - \mu/\sigma)) \right\} \frac{   \d \tau}{\tau}
\end{equation}
holds for both $x_1^* = \pm 1$. This concludes the proof.
\end{proof}

From now, let us denote
\begin{equation}\label{eqn:pstartheta}
     p^*_{\theta}(x_1^*, \dots, x_n^*) = \int_{\tau > 0}  \left\{\prod_{i=1}^n (1/\tau) g(x_i^* /\tau - \theta)) \right\} (1/\tau)  \d \tau,
\end{equation}
i.e.\ the density $ \frac{\d Q_n}{\d \eta^{\langle n-1 \rangle}}$ when $P_{\mu, \sigma^2}$ satisfies $\mu/\sigma =\theta$. Here is a property of $p_{\theta}^*$.

\begin{lemma}\label{lem:scaling} Let $\lambda > 0$ be any constant. Then, $ p^*_{\theta}(\lambda x_1^*, \dots, \lambda x_n^*) = \lambda^{-n} p^*_{\theta}(x_1^*, \dots, x_n^*)$. In particular, taking $\lambda = |x_1|$, we have $p^*_{\theta}(x_1, \dots, x_n) = |x_1|^{-n} p^*_{\theta}(x_1^*, \dots, x_n^*)$.
\end{lemma}
\begin{proof}[Proof of Lemma~\ref{lem:scaling}]
\begin{align}
    & p^*_{\theta}(\lambda x_1^*, \dots, \lambda x_n^*)
    \\
    = &  \int_{\tau > 0}  \left\{\prod_{i=1}^n (1/\tau) g(\lambda x_i^* /\tau - \theta)) \right\} (1/\tau)  \d \tau
    \\
    = & \int_{\gamma = \lambda ^{-1} \tau > 0}  \left\{\prod_{i=1}^n (1/\lambda \gamma) g( x_i^* /\gamma - \theta)) \right\} (1/\lambda \gamma) \lambda \d \gamma 
    \\
    = & \lambda^{-n} p^*_{\theta}(x_1^*, \dots, x_n^*).
\end{align}
\end{proof}

\begin{proof}[Proof of Lemma~\ref{lem:si-lr}]
    Denote $X_i/|X_1|$ by $X_i^*$.
    \begin{equation}
        h_n(\theta; \theta_0) = \underbrace{ \frac{p^*_{\theta} (X_1, \dots, X_n) }{p_{\theta_0}^*(X_1, \dots, X_n)} = \frac{p^*_{\theta} (X_1^*, \dots, X_n^*) }{p_{\theta_0}^*(X_1^*, \dots, X_n^*)} }_{\text{equality due to Lemma~\ref{lem:scaling}}} = \frac{\d Q^n_{\mu, \sigma^2} }{\d Q^n_{\mu_0, \sigma^2_0}}(X_1^*, \dots, X_n^*).
    \end{equation}
    for any $\mu/\sigma = \theta$ and $\mu_0/\sigma_0 = \theta_0$. Now it follows from Lemma~\ref{lem:lrm-joint} that $\{  h_n(\theta; \theta_0) \}$ is an NM for $P_{\mu_0, \sigma^2_0}$ on the canonical filtration generated by $\{X_n^*\}$, which according to Proposition~\ref{prop:si-filt-rule} is just $\{ \cF_n^* \}$. 
\end{proof}

\begin{proof}[Proof of Corollary~\ref{cor:t-lr}]
Using $g(x) = \frac{1}{\sqrt{2\pi}} \exp(-x^2/2)$ in \eqref{eqn:pstartheta},
\begin{align}
     & p^*_{\theta}(x_1, \dots, x_n) = \int_{\tau > 0}  \left\{\prod_{i=1}^n (1/\tau) g(x_i /\tau - \theta)) \right\} (1/\tau)  \d \tau
     \\
     =& \frac{1}{(2\pi)^{n/2}} \int_{\tau > 0}  \left\{\prod_{i=1}^n  \e^{-(x_i /\tau - \theta)^2/2} \right\} \frac{1}{\tau^{n+1}}  \d \tau
     \\
     = & \frac{\e^{- n \theta^2 /2 }}{(2\pi)^{n/2}}  \int_{\tau > 0}  \exp\left\{ -\frac{\sum_{i=1}^n {x_i }^2}{2\tau^2} + \frac{\theta \sum_{i=1}^n x_i}{\tau} \right\} \frac{1}{\tau^{n+1}}  \d \tau.
\end{align}
Therefore, applying Lemma~\ref{lem:si-lr} with $\theta_0 = 0$, we have
\begin{align}
    & h_n(\theta; 0) = \e^{- n \theta^2 /2 } \frac{\int_{\tau > 0}  \exp\left\{ -\frac{\sum_{i=1}^n {X_i }^2}{2\tau^2} + \frac{\theta \sum_{i=1}^n X_i}{\tau} \right\} \frac{1}{\tau^{n+1}}  \d \tau}{\int_{\tau > 0}  \exp\left\{ -\frac{\sum_{i=1}^n {X_i }^2}{2\tau^2}  \right\} \frac{1}{\tau^{n+1}}  \d \tau}
    \\
    & \text{(the denominator being a generalized Gaussian integral)}
    \\
    = & \frac{\e^{- n \theta^2 /2 }}{\Gamma(n/2)} \int_{y > 0} y^{n/2-1}  \exp\left\{ -y + \theta \left(\sum_{i=1}^n X_i \right) \sqrt{\frac{2y}{\sum_{i=1}^n {X_i}^2}} \right\}  \d y,
\end{align}
coinciding with $h_{\theta, n}$.
\end{proof}

\begin{proof}[Proof of Theorem~\ref{thm:lai-ensm} (as well as Theorem~\ref{thm:lai-cs})]
Putting a flat, {improper} prior $\d \pi(\theta) = \d \theta$ over $\mathbb R$, we have the following improperly mixed, extended nonnegative supermartingale,
\begin{align}
    & H_n = \int h_{\theta,n} \pi(\d\theta)
    \\
    = & \int_{\theta} \frac{\e^{- n \theta^2 /2 }}{\Gamma(n/2)} \left( \int_{y > 0} y^{n/2-1}  \exp\left\{ -y + \theta \left(\sum_{i=1}^n X_i \right) \sqrt{\frac{2y}{\sum_{i=1}^n {X_i}^2}} \right\}  \d y  \right)\cdot \d \theta
    \\
     = &  \frac{1}{  \Gamma(n/2)} \int_{y>0} y^{n/2-1} \e^{-y} \left(\int_{\theta} \exp\left\{ -\frac{n}{2}\theta^2 + \theta \left(\sum_{i=1}^n X_i \right) \sqrt{\frac{2y}{\sum_{i=1}^n {X_i}^2}}\right\} \d \theta \right) \d y
     \\
     = & \frac{1}{ \Gamma(n/2)} \int_{y>0} y^{n/2-1} \e^{-y} \left( \sqrt{\frac{2\pi}{n}} \exp\left\{ \frac{\left(\sum_{i=1}^n X_i \right)^2  \frac{2y}{\sum_{i=1}^n {X_i}^2} }{2n} \right\} \right) \d y
     \\
     = & \frac{1}{  \Gamma(n/2)}  \sqrt{\frac{2\pi}{n}} \int_{y>0} y^{n/2-1} \exp\left\{ - \left(1- \frac{S_n^2}{n V_n }  \right) y \right\} \d y
     \\
     = & \frac{1}{  \Gamma(n/2)}  \sqrt{\frac{2\pi}{n}}  \Gamma(n/2) \left(1- \frac{S_n^2}{n V_n }  \right)^{-n/2} =  \sqrt{\frac{2 \pi }{n}} \left(1- \frac{S_n^2}{n V_n }  \right)^{-n/2}
     \\
     = & \sqrt{\frac{2 \pi }{n}} \left( \frac{n V_n }{n V_n - S_n^2}  \right)^{n/2}
\end{align}
Note that when $n = 1$, $n V_n - S_n^2   = 0$, and $H_0$ is understood to be $+\infty$. Note that
\begin{equation}
  t_{n-1} \sim \frac{\frac{1}{n}S_n}{\frac{1}{\sqrt{n}}\sqrt{\frac{1}{n-1}(V_n - \frac{1}{n}S_n^2)}} = \sqrt{n-1} \sqrt{ \frac{S_n^2}{n V_n - S_n^2} } =: T_{n-1}
\end{equation}
Then $H_n = \sqrt{\frac{2\pi}{n}}  \left(\frac{T_{n-1}^2}{n-1} + 1\right)^{n/2} $.

In terms of CS, every $X_i$ should be seen as $X_i - 0 = X_i - \mu$, and hence
\begin{equation}
    H_n^{\mu} = \sqrt{\frac{2\pi}{n}}\left( \frac{n \sum (X_i - \mu)^2 }{n \sum (X_i - \mu)^2 - (\sum (X_i - \mu))^2}  \right)^{n/2}
\end{equation}
is also an extended nonnegative supermartingale.

The extended Ville's inequality, applied on the process $\sqrt{m/2\pi} H_n^\mu$, reads,
\begin{center}
   With probability at most
    \begin{equation}
        P =  \Pr\left[\frac{T_{m-1}^2}{m-1} + 1 \ge \eta\right] + \eta^{-m/2} \Exp\left[\id_{\left\{ \frac{T_{m-1}^2}{m-1} + 1 < \eta \right\}} \left(\frac{T_{m-1}^2}{m-1} + 1\right)^{m/2}\right],
    \end{equation}
    there exists $n \ge m$, 
    \begin{equation}
       \sqrt{\frac{m}{n}} \left( \frac{n \sum (X_i - \mu)^2 }{n \sum (X_i - \mu)^2 - (\sum (X_i - \mu))^2}  \right)^{n/2} \ge \eta^{m/2}.
    \end{equation}
\end{center}
Define $a =\sqrt{(m-1)(\eta - 1)}$. Then $\eta^{m/2} = (1 + a^2/(m-1))^{m/2} = h(a)$ and
\begin{align}
    P &= \Pr[ |T_{m-1}| \ge a ] + h(a)^{-1} \Exp[ \id_{\{|T_{m-1}| < a\}} h(T_{m-1}) ]
    \\
    & = 2 \Pr[ T_{m-1} > a ] + 2 \frac{ \Exp[ \id_{\{0 \le T_{m-1} < a\}} h(T_{m-1}) ] }{h(a)}
\end{align}
Let $f_{m-1}$, $F_{m-1}$ be the density and CDF of $T_{m-1}$. The expectation $\Exp[ \id_{\{0 \le T_{m-1} < a\}} h(T_{m-1}) ]$ must be smaller than $\frac{a}{a-a'}\Exp[ \id_{\{a' \le T_{m-1} < a\}} h(T_{m-1}) ]$, for all $a' \in (0, a)$. Note that
\begin{equation}
    \lim_{a' \to a} \frac{a}{a-a'}\Exp[ \id_{\{a' \le T_{m-1} < a\}} h(T_{m-1}) ] = a h(a) f_{m-1}(a).
\end{equation}
Hence
\begin{equation}
    P \le 2(1-F_{m-1}(a)) + 2a f_{m-1}(a).
\end{equation}
This concludes the missing proof of Lai's t-CS.
\end{proof}

\begin{proof}[Proof of Theorem~\ref{thm:lai-e}]
Let us put a Gaussian prior on $\theta$,
\begin{equation}
    \frac{\d \pi(\theta)}{\d \theta} = \frac{c}{ \sqrt{2\pi}} \e^{-c^2\theta^2/2},
\end{equation}
and define the martingale
\begin{align}
    & G_n^{(c)} = \int h_{\theta,n} \pi(\d\theta)
    \\
    = & \int_{\theta} \frac{\e^{- n \theta^2 /2 }}{\Gamma(n/2)} \left( \int_{y > 0} y^{n/2-1}  \exp\left\{ -y + \theta \left(\sum_{i=1}^n X_i \right) \sqrt{\frac{2y}{\sum_{i=1}^n {X_i}^2}} \right\}  \d y  \right)\cdot \pi(\d \theta)
    \\
     = &  \frac{c}{ \sqrt{2\pi} \Gamma(n/2)} \int_{y>0} y^{n/2-1} \e^{-y} \left(\int_{\theta} \exp\left\{ -\frac{n+c^2}{2}\theta^2 + \theta \left(\sum_{i=1}^n X_i \right) \sqrt{\frac{2y}{\sum_{i=1}^n {X_i}^2}}\right\} \d \theta \right) \d y
     \\
     = & \frac{c}{ \sqrt{2\pi} \Gamma(n/2)} \int_{y>0} y^{n/2-1} \e^{-y} \left( \sqrt{\frac{\pi}{\frac{n+c^2}{2}}} \exp\left\{ \frac{\left(\sum_{i=1}^n X_i \right)^2  \frac{2y}{\sum_{i=1}^n {X_i}^2} }{2(n+c^2)} \right\} \right) \d y
     \\
     = & \frac{c}{ \sqrt{2\pi} \Gamma(n/2)}  \sqrt{\frac{\pi}{\frac{n+c^2}{2}}} \int_{y>0} y^{n/2-1} \exp\left\{ - \left(1- \frac{S_n^2}{(n+c^2) V_n }  \right) y \right\} \d y
     \\
     = & \frac{c}{  \Gamma(n/2)}  \sqrt{\frac{1}{n+c^2}}  \Gamma(n/2) \left(1- \frac{S_n^2}{(n+c^2) V_n }  \right)^{-n/2} =  \sqrt{\frac{c^2}{n+c^2}} \left(1- \frac{S_n^2}{(n+c^2) V_n }  \right)^{-n/2}.
\end{align}
This is a valid test martingale.

While this is a test martingale for $\normals_{\mu = 0}$, a shifting argument ($X_i \to X_i - \mu_0$) yields that
\begin{equation}
    G_n^{(c, \mu_0)} =  \sqrt{\frac{c^2}{n+c^2}} \left(1- \frac{(S_n - n\mu_0)^2}{(n+c^2) (V_n - 2 S_n \mu_0 + n \mu_0^2) }  \right)^{-n/2}
\end{equation}
is a test martingale for $\normals_{\mu = \mu_0}$. Applying Ville's inequality, with probability at least $1-\alpha$, for all $n$,
\begin{align}
    \frac{nc^2 \mu_0^2 - 2c^2 S_n \mu_0 + (n+c^2)V_n - S_n^2}{ (n+c^2)(n \mu_0^2 - 2S_n\mu_0 +V_n) } &\ge \left( \frac{\alpha^2 c^2}{n+c^2} \right)^{1/n},
    \\
    \frac{ nc^2 (\mu_0 - \avgX{n})^2 +(n+c^2)(V_n - S_n^2/n)}{ (n+c^2)( n (\mu_0 - \avgX{n})^2 + V_n - S_n^2/n  ) } &\ge \left( \frac{\alpha^2 c^2}{n+c^2} \right)^{1/n},
    \\
    \left\{ \left( \frac{\alpha^2 c^2}{n+c^2} \right)^{1/n} (n+c^2)n - nc^2   \right\}  (\mu_0 - \avgX{n})^2 &\le \left\{ (n+c^2)\left(1- \left( \frac{\alpha^2 c^2}{n+c^2} \right)^{1/n}\right) \right\} ( V_n - S_n^2/n  ),
    \\
     (\mu_0 - \avgX{n})^2 &\le \frac{ (n+c^2)\left(1- \left( \frac{\alpha^2 c^2}{n+c^2} \right)^{1/n}\right) }{ \left( \frac{\alpha^2 c^2}{n+c^2} \right)^{1/n} (n+c^2) - c^2   } \left(\avgXsq{n} - \avgX{n}^2 \right)
\end{align}
This concludes the proof.
\end{proof}

\begin{proof}[Proof of Proposition~\ref{prop:conv-t-ensm}]
    First, under \revise{any distribution with mean $\mu$ and variance $\sigma^2$}, note that $T_{n-1}=\sqrt{n-1} \frac{S_n - n\mu}{ \sqrt {n V_n - S_n^2} }$ and $s_n^2 = V_n/n - S_n^2/n^2$. We replace $\frac{T_{n-1}}{\sqrt{n-1}}$ by $\frac{T_{n-1}}{\sqrt{n-1}} + \frac{\mu}{{s_{n}}}$ in \eqref{eqn:lai-ensm-T-expr}  to have
    \begin{gather}
         H_n = \sqrt{\frac{2 \pi }{n}} \left( 1 + \left(\frac{T_{n-1}}{\sqrt{n-1}} + \frac{\mu}{{s_{n}}}\right)^2 \right)^{n/2},
         \\
         \frac{\log H_n}{n} = \frac{\log(2\pi/n)}{2n} + \frac{1}{2} \log  \left( 1 + \left(\frac{T_{n-1}}{\sqrt{n-1}} + \frac{\mu}{{s_{n}}}\right)^2 \right).
    \end{gather}
    It is well known $s^2_{n}$ converges to $\sigma^2$ almost surely. Therefore $\frac{T_{n-1}}{\sqrt{n-1}} = \frac{\avgX{n} - \mu}{{s_{n}}}$ converges almost surely to 0. These imply that $\frac{\log H_n}{n}$ converges almost surely to $\frac{1}{2}\log(1 + \mu^2/\sigma^2 )$.

    Second, under any $P \in \normals_{\mu = 0}$, let $\delta > 0$.
    By \eqref{eqn:lai-ensm-T-expr} we have
    \begin{align}
        P( H_n > \delta ) = P\Bigg(  T_{n-1}^2 > \underbrace{ \left(\left( \frac{\delta^2 n}{2\pi} \right)^{1/n} - 1 \right)(n-1) }_{\gtrsim\log n} \Bigg) \lesssim P( T_1^2 > \log n  ).
    \end{align}
    The underbraced term grows logarithmically because $n^{1/n} - 1= \frac{\log n}{n} + \mathcal{O}(n^{-2})$. The ``$\lesssim$" holds because the tail of Student's t-distribution strictly decreases when the DOF increases. So $ P( H_n > \delta ) $ goes to 0 as $n\to\infty$, concluding that $\{ H_n \}$ converges to 0 in probability. By \citet[Proposition A.14]{ensm}, an ENSM converges almost surely, so  $\{ H_n \}$ converges to 0 almost surely. The proof is complete.
\end{proof}

\begin{proof}[Proof of Proposition~\ref{prop:conv-t-nm}]
  First, under \revise{any distribution with mean $\mu$ and variance $\sigma^2$}, we replace in \eqref{eqn:lai-nm-tstat} $\frac{T_{n-1}}{\sqrt{n-1}}$ by $\frac{T_{n-1}}{\sqrt{n-1}} + \frac{\mu}{s_{n}}$,
    \begin{gather}
         G_n^{(c)} = \sqrt{\frac{c^2}{n+c^2}} \left( 1 + \frac{ n}{ \frac{n+c^2}{ \left(\frac{T_{n-1}}{\sqrt{n-1}} + \frac{\mu}{s_{n}}\right)^2 } + {c^2}  }\right)^{n/2},
         \\
         \frac{\log  G_n^{(c)} }{n} = \frac{\log(\frac{c^2}{n+c^2})}{2n} + \frac{1}{2}\log \left( 1 + \frac{ n}{ \frac{n+c^2}{ \left(\frac{T_{n-1}}{\sqrt{n-1}} + \frac{\mu}{s_{n}}\right)^2 } + {c^2}  }\right).
    \end{gather}
    As in the proof of Proposition~\ref{prop:conv-t-ensm}, $\frac{T_{n-1}}{\sqrt{n-1}}$ converges to 0 and $s_n^2$ to $\sigma^2$ almost surely. Therefore $\frac{\log  G_n^{(c)} }{n}$ converges a.s.\ to $\frac{1}{2}\log(1 + \mu^2/\sigma^2)$.
  
  Second, under any $P \in \normals_{\mu = 0}$, let $\delta > 0$.
    By \eqref{eqn:lai-nm-tstat} we have
    \begin{align}
        P(  G_n^{(c)} > \delta ) = P\Bigg( \frac{ n}{ \frac{(n+c^2)(n-1) }{T_{n-1}^2 } + {c^2}  } > \underbrace{ \left( \frac{(n+c^2)\delta^2}{c^2} \right)^{1/n} -1 }_{\gtrsim \frac{\log n}{n}} \Bigg) \lesssim P( T_1^2 > \log n  ),
    \end{align}
    which converges to 0. So $\{ G_n^{(c)} \}$ converges to 0 in probability. It thus also converges to 0 almost surely due to Doob's martingale convergence theorem, concluding the proof.
\end{proof}

\begin{proof}[Proof of Theorem~\ref{thm:si-onesided}]
Let us put a half Gaussian prior on $\theta$,
\begin{equation}
    \frac{\d \pi^+(\theta)}{\d \theta} = \id_{\theta \ge 0}\frac{c}{ \sqrt{2\pi}} \e^{-c^2\theta^2/2},
\end{equation}
and define the martingale
\begin{align}
    & \int h_{\theta,n} \pi^+(\d\theta)
    \\
    = & \int_{\theta > 0} \frac{\e^{- n \theta^2 /2 }}{\Gamma(n/2)} \left( \int_{y > 0} y^{n/2-1}  \exp\left\{ -y + \theta \left(\sum_{i=1}^n X_i \right) \sqrt{\frac{2y}{\sum_{i=1}^n {X_i}^2}} \right\}  \d y  \right)\cdot \pi(\d \theta)
    \\
     = &  \frac{c}{ \sqrt{2\pi} \Gamma(n/2)} \int_{y>0} y^{n/2-1} \e^{-y} \left(\int_{\theta > 0} \exp\left\{ -\frac{n+c^2}{2}\theta^2 + \theta \left(\sum_{i=1}^n X_i \right) \sqrt{\frac{2y}{\sum_{i=1}^n {X_i}^2}}\right\} \d \theta \right) \d y
     \\ 
     = & \frac{c}{ \sqrt{2\pi} \Gamma(n/2)} \int_{y>0} y^{n/2-1} \e^{-y} \left( \sqrt{\frac{\pi}{\frac{n+c^2}{2}}} \exp\left\{ \frac{\left(\sum_{i=1}^n X_i \right)^2  \frac{2y}{\sum_{i=1}^n {X_i}^2} }{2(n+c^2)} \right\} \Phi\left( \frac{ \left(\sum_{i=1}^n X_i \right) \sqrt{\frac{2y}{\sum_{i=1}^n {X_i}^2}}}{\sqrt{n+c^2}}  \right) \right) \d y 
     \\
     = & \frac{c}{ \sqrt{2\pi} \Gamma(n/2)}  \sqrt{\frac{\pi}{\frac{n+c^2}{2}}} \int_{y>0} y^{n/2-1} \exp\left\{ - \left(1- \frac{S_n^2}{(n+c^2) V_n }  \right) y \right\} \Phi\left( \frac{ S_n \sqrt{2y}}{\sqrt{(n+c^2)V_n}}  \right) \d y
     \\
     & \text{(Chernoff approximation } 1-\Phi(x)\le \e^{-(x\vee 0)^2/2} 
     \\
     \ge &  \frac{c}{ \sqrt{2\pi} \Gamma(n/2)}  \sqrt{\frac{\pi}{\frac{n+c^2}{2}}} \int_{y>0} y^{n/2-1} \exp\left\{ - \left(1- \frac{S_n^2}{(n+c^2) V_n }  \right) y \right\} 
\left( 1 - \exp\left( - \frac{(S_n \vee 0)^2 y}{(n+c^2) V_n }  \right) \right)  \d y
\\
= &    \frac{c}{ \sqrt{2\pi} \Gamma(n/2)}  \sqrt{\frac{\pi}{\frac{n+c^2}{2}}}  \bigg \{ \int_{y>0} y^{n/2-1} \exp\left\{ - \left(1- \frac{S_n^2}{(n+c^2) V_n }  \right) y \right\} 
  \d y  \\
  &  -  \int_{y>0} y^{n/2-1}  \exp\left(\frac{(S_n^2 - (S_n \vee 0)^2) y}{(n+c^2) V_n } - 1 \right)
  \d y \bigg\} 
  \\
  = &   \frac{c}{  \Gamma(n/2)}  \sqrt{\frac{1}{n+c^2}}  \left\{ 
 \Gamma(n/2) \left(1- \frac{S_n^2}{(n+c^2) V_n }  \right)^{-n/2} - \Gamma(n/2) \left(1- \frac{(S_n \wedge0) ^2}{(n+c^2) V_n }  \right)^{-n/2}  \right\} 
 \\ 
 = & \sqrt{\frac{c^2}{n+c^2}} \left( \left(1- \frac{S_n^2}{(n+c^2) V_n }  \right)^{-n/2} - \left(1- \frac{(S_n \wedge0) ^2}{(n+c^2) V_n }  \right)^{-n/2} \right) =: G^{(c-)}_n/2.
\end{align}
Since $\pi^+$ is of total measure $1/2$, $G^{c+}_n$ is an e-process for $\normals_{\mu = 0}$.
\end{proof}

\revise{
\begin{proof}[Proof of Theorem~\ref{thm:asympcs}]
    
Recall that Theorem~\ref{thm:lai-e} states that, letting
\begin{equation}
    \operatorname{radius}_n = \sqrt{\underbrace{\frac{ (n+c^2)\left(1- \left( \frac{\alpha^2 c^2}{n+c^2} \right)^{1/n}\right) }{ \left\{ \left( \frac{\alpha^2 c^2}{n+c^2} \right)^{1/n} (n+c^2) - c^2 \right\} \vee 0   }}_{\rho^2(n, c, \alpha)} \ \cdot \ \underbrace{\left(\avgXsq{n} - \avgX{n}^2 \right)}_{s_n^2(X)}},
\end{equation}
then the intervals
\begin{equation}\label{eqn:lai-e-cs-reprise}
   \left[  \avgX{n}(X) \pm \operatorname{radius}_n   \right]
\end{equation}
form a $(1-\alpha)$-CS for $\mu$ over $\normals$.


Now, suppose $\{X_n\} \iid P $ with mean 
$\mu$ and variance $\sigma^2$. By \citet[Lemma A.1]{waudby2021time}, there exist $\{G_n\} \iid \normal{0}{1}$ such that
\begin{equation}
    \avgX{n}(X) - \mu = {s_n(X)}\avgX{n}(G) + \varepsilon_n, \quad \varepsilon_n = o(\sqrt{\log \log n / n}).
\end{equation}

Applying Theorem~\ref{thm:lai-e} to $\{G_n\}$, we have, with probability at least $1-\alpha$,
\begin{equation}
    | \avgX{n}(G)  | \le \rho(n,c,\alpha) s_n(G).
\end{equation}
Therefore, with probability at least $1-\alpha$,
\begin{align}
    | \avgX{n}(X) -  \mu| &\le   | \avgX{n}(X) -  \mu - {s_n(X)} \avgX{n}(G) |  + {s_n(X)} |\avgX{n}(G)| 
    \\
    &\le  |\varepsilon_n| + {s_n(X) \rho(n,c,\alpha) s_n(G) }.
\end{align}
Finally, we recall from our previous analysis that $s_n(X) \rho(n,c,\alpha) = \mathcal{O}(\sqrt{\log n / n}) = \Omega(\varepsilon_n)$, therefore
\begin{equation}
    \frac{ |\varepsilon_n| + {s_n(X) \rho(n,c,\alpha) s_n(G) }  }{s_n(X)   \rho(n,c,\alpha) } =  \frac{ |\varepsilon_n|  }{s_n(X)   \rho(n,c,\alpha) } +{s_n(G)} \stackrel{a.s.}{\longrightarrow}0+1=1,
\end{equation}
concluding that \eqref{eqn:lai-e-cs-reprise} is an $(1-\alpha)$-asymptotic CS over $\mathcal{L}^2$.
\end{proof}
}

\subsection{Proofs for Optimality}\label{sec:pf-opt}

\begin{proof}[Proof of Theorem~\ref{thm:it-lb}]
    Fix an $L \in \cL$, $\sigma > 0$, and $\alpha > 0$. Let 
\begin{equation}
    \mu =\sigma\sqrt{ (1-(1- 3\alpha)^2)^{-2/n}  - 1 }.
\end{equation}
First, consider the distribution $N_{0,\mu^2+\sigma^2}$. We have
\begin{equation}
    N_{0,\mu^2+\sigma^2}(  0 < L(X_1,\dots,X_n) ) \le \alpha.
\end{equation}
Next, consider the distribution $N_{\mu,\sigma^2}$. Since
$\kl( N_{\mu,\sigma^2}^{\otimes n} \| N_{0,\mu^2+\sigma^2}^{\otimes n} ) = \frac{n}{2}\log(1 + \mu^2/\sigma^2)$, 
it follows from Bretagnolle-Huber inequality that
\begin{equation}
     \operatorname{TV} ( N_{\mu,\sigma^2}^{\otimes n} , N_{0,\mu^2+\sigma^2}^{\otimes n}   ) \le \sqrt{1 - \exp(-\kl( N_{\mu,\sigma^2}^{\otimes n} \| N_{0,\mu^2+\sigma^2}^{\otimes n} ))} = \sqrt{1 - (1+\mu^2/\sigma^2)^{-n/2}}.
\end{equation}
Therefore, 
\begin{equation}
    N_{\mu,\sigma^2}(  0 < L(X_1,\dots,X_n) ) \le \alpha + \sqrt{1 - (1+\mu^2/\sigma^2)^{-n/2}} = \alpha + 1-3\alpha = 1-2\alpha.
\end{equation}
By the definition of $W_{2\alpha}(L, P)$,
\begin{equation}
     N_{\mu,\sigma^2}(   \mu - L(X_1,\dots,X_n) < W_{2\alpha}(L, P) )=  1-2\alpha,
\end{equation}
indicating that $W_{2\alpha}(L, N_{\mu,\sigma^2}) \ge \mu = \sigma\sqrt{ (1-(1-3\alpha)^2)^{-2/n}  - 1 }$. Since this holds for any lower confidence bound $L \in \mathcal{L}_{\alpha, n}$, we see that
\begin{align}
  M^{-}_{\alpha, n} =  \inf_{ L\in  \mathcal{L}_{\alpha, n}} \sup_{P \in \normals} \sigma^{-1}(P) \cdot W_{2\alpha}(L, P) \ge&  \sqrt{ (1-(1- 3\alpha)^2)^{-2/n}  - 1 } \\
  =& \sqrt{ ( 6\alpha - 9\alpha^2 )^{-2/n}  - 1 } .
\end{align}
\end{proof}

\begin{proof}[Proof of Corollary~\ref{cor:it-lb2}] Fix a $\CI = [L(X_1,\dots, X_n), U (X_1,\dots, X_n)] \in \mathcal{I}_{\alpha, n}$ and consider any $P \in \normals$. First, since $L \in \mathcal{L}_{\alpha, n}$, we have
\begin{equation}
    P(  \mu(P) -  L(X_1,\dots, X_n) \ge W_{2\alpha}(L, P)  ) = 2\alpha.
\end{equation}
Next, since $[L(X_1,\dots, X_n), U (X_1,\dots, X_n)]$ is a $(1-\alpha)$-CI,
\begin{equation}
    P(  \mu(P) \le U (X_1,\dots, X_n)  ) \ge 1-\alpha.
\end{equation}
By a union bound, 
\begin{equation}
    P(  U (X_1,\dots, X_n) -  L(X_1,\dots, X_n) \ge W_{2\alpha}(L, P)  ) \ge \alpha,
\end{equation}
implying that $W_{\alpha}(\CI, P ) \ge W_{2\alpha}(L, P)$. Therefore $ M_{\alpha, n}    \ge M_{\alpha, n}^{-} $ by taking a minimax.
\end{proof}

 {\paragraph*{Discussion on Proposition~\ref{prop:ripr-simple}.} Before we present and prove an extended version of Proposition~\ref{prop:ripr-simple}, let us briefly review some key definitions and results of \cite{larsson2024numeraire}. Let $\cP$ and $Q$ be mutually equivalent probability measures.
Let $\cP^{\circ \circ}$ be the ``bipolar" of the measure set $\cP$, which contains all sub-probability measures for which any e-value for $\cP$ is an e-value as well. $\cP^{\circ \circ}$ is thus a superset of $\cP$, understood as the ``effective null".
If the infimum $ \inf_{P \in \cP^{\circ \circ}} \kl(Q\| P)$ is attained as minimum by some $P^* \in \cP^{\circ \circ}$, we say that  $P^*$ is the \emph{reverse information projection} of $Q$ onto the set $\cP$. 
Proposition 2.4, Theorem 3.4, and Theorem 4.1 of \cite{larsson2024numeraire} jointly state the following.
\begin{lemma}\label{lem:num}
 Let $P^*$ be a sub-probability measure such that $M^* = \frac{\d Q}{\d P^*}(X_1)$ is an e-value for $\cP$. Then, the following are equivalent:
 \begin{enumerate}
     \item $P^* \in \cP^{\circ \circ}$;
     \item $P^*$ is the reverse information projection of $Q$ onto $\cP$;
     \item $M^*$ is the optimal e-value for $\cP$: any e-value $M$ for $\cP$ satisfies $\EE_Q(\log M) \le \EE_Q(\log M^*)$.
 \end{enumerate}
   When this happens, the optimal e-value $M^*$ is called the \emph{numeraire e-value}.
\end{lemma}




We have the following statement regarding the reverse information projection and the numeraire in t-tests where $\cP = \normals_{\mu = 0}$ and $Q = \normal{\mu}{\sigma^2}$.


\begin{proposition}[Reverse information projection and numeraire in the t-test] \label{prop:ripr} The probability measure $\normal{0}{\sigma^2+
\mu^2}$ attains the following minimum
\begin{equation}
   \kl( \normal{\mu}{\sigma^2} \| \normal{0}{\sigma^2+
\mu^2}) =  \min_{P \in  \normals_{\mu = 0}} \kl( \normal{\mu}{\sigma^2} \| P) = \frac{1}{2}\log(1+\mu^2/\sigma^2).
\end{equation}
Further, $M^* =  \frac{\d \normal{\mu}{\sigma^2}}{\d \normal{0}{\sigma^2+
\mu^2}}(X_1)$ is an e-value for $\normals_{\mu = 0}$, thus it is the numeraire e-value and $\normal{0}{\sigma^2+
\mu^2}$ the reverse information projection of $\normal{\mu}{\sigma^2}$ onto $\normals_{\mu = 0}$. Consequently, any e-value $ M $ for $\normals_{\mu = 0}$ has e-power at most
\begin{equation}
    \EE_{\normal{\mu}{\sigma^2}}(\log M_n) \le \frac{1}{2}\log(1+\mu^2/\sigma^2),
\end{equation}
equality if and only if $M$ is the numeraire $M^*$.
\end{proposition}

\begin{proof}[Proof of Proposition~\ref{prop:ripr}]
    Let $P = \normal{0}{\sigma^2_0}$. Then
    \begin{align}
        \kl( \normal{\mu}{\sigma^2} \| P) = \frac{1}{2} (\log \sigma_0^2 - \log \sigma^2) + \frac{\sigma^2 + \mu^2}{2\sigma_0^2} - \frac{1}{2}.
    \end{align}
    As a function of $\sigma_0^2$, the above takes minimum when $\sigma_0^2 = \sigma^2 + \mu^2$, in which case the minimum is $\kl( \normal{\mu}{\sigma^2} \| \normal{0}{\sigma^2+\mu^2}) = \frac{1}{2}\log(1+\mu^2/\sigma^2)$.

Under any $P_0 = \normal{0}{\sigma_0^2} \in \cP$:
    \begin{align}
       &  \EE_{P_0} \left( \frac{\d Q}{\d P^*} (X_1) \right) = \int \frac{\sqrt{\sigma^2 + \mu^2}}{\sigma \sigma_0 \sqrt{2\pi}} \exp \left( -  \frac{(x-\mu)^2}{2\sigma^2} +  \frac{x^2}{2(\sigma^2 + \mu^2)} -  \frac{x^2}{2\sigma_0^2} \right) \d x
    \\
    = & \frac{\sqrt{\sigma^2 + \mu^2}}{\sigma \sigma_0 \sqrt{2\pi}}  \int \exp \left( - \frac{\mu^2 \sigma_0^2 + \sigma^2(\sigma^2 + \mu^2)}{2 \sigma^2 \sigma_0^2 (\sigma^2 + \mu^2)} x^2 + \frac{\mu}{\sigma^2} x - \frac{\mu^2}{2\sigma^2}  \right) \d x
    \\
    = & \frac{\sqrt{\sigma^2 + \mu^2}}{\sigma \sigma_0 \sqrt{2\pi}}   \cdot \sqrt{\frac{\pi}{ \frac{\mu^2 \sigma_0^2 + \sigma^2(\sigma^2 + \mu^2)}{2 \sigma^2 \sigma_0^2 (\sigma^2 + \mu^2)} 
 } } \exp\left(  \frac{ \mu^2/\sigma^4 }{4 \frac{\mu^2 \sigma_0^2 + \sigma^2(\sigma^2 + \mu^2)}{2 \sigma^2 \sigma_0^2 (\sigma^2 + \mu^2)}  } - \frac{\mu^2}{2\sigma^2}  \right)
\\
= & \frac{\sigma^2 + \mu^2}{\sqrt{\mu^2 \sigma_0^2 + \sigma^2(\sigma^2 + \mu^2)}}  \exp\left(  \frac{ \mu^2  (\sigma_0^2 - (\sigma^2 + \mu^2))  }{2 (\mu^2 \sigma_0^2 + \sigma^2(\sigma^2 + \mu^2))} \right)
\\
& \quad \text{letting $z = \sqrt{\mu^2 \sigma_0^2 + \sigma^2(\sigma^2 + \mu^2)}$ and $y = \sigma^2 + \mu^2$}
\\
= & \frac{y}{z} \exp\left( \frac{z^2 - y^2}{2z^2} \right) =  \sqrt{\frac{y^2}{z^2} \exp\left(1 -  \frac{y^2}{z^2}  \right)} \le 1.
    \end{align}
    This concludes that the likelihood ratio is an e-value for $\normals_{\mu = 0}$. The rest of the statement follows from Lemma~\ref{lem:num}.
\end{proof}}

\subsection{Proofs for Additional t-Tests}\label{sec:pf-additional}
\begin{proof}[Proof of Proposition~\ref{prop:betabinom}]
    Following \citet[Proposition 7]{howard2021time}, let $p = \frac{\exp(\lambda/2)}{\exp(\lambda/2) + \exp(-\lambda/2)}$. We can now write the process $B_n^\lambda$ as
    \begin{equation}
        B_n^{\lambda} = \frac{p^{n/2 + T_n} (1-p)^{n/2-T_n}}{(1/2)^{n}}.
    \end{equation}
    Now, integrating the supermartingale above with $p \sim \operatorname{Beta}(a, b)$,
    \begin{align}
        B_n^{(a, b)} = \int_{0}^1 \frac{p^{n/2 + T_n} (1-p)^{n/2-T_n}}{(1/2)^{n/2 + T_n} (1/2)^{n/2-T_n}} \cdot \frac{p^{a-1}(1-p)^{b-1}}{\operatorname{B}(a,b)} \d p = \frac{\operatorname{B}(a + n/2 + T_n,b + n/2 - T_n)}{ (1/2)^{n} \operatorname{B}(a,b)}
    \end{align}
    forms a supermartingale as well.

    Next, due to the facts
    \begin{enumerate}
        \item $\lim_{n\to\infty} \frac{\log \operatorname{B}(\alpha n, \beta n)}{n} = \log \frac{\alpha^\alpha \beta^\beta}{(\alpha + \beta)^{\alpha + \beta}}$ for any $\alpha, \beta > 0$; and
        \item $T_n/n \to \Phi(\mu/\sigma) - 1/2$ almost surely,
    \end{enumerate}
we can see that the test supermartingale $\{  B_n^{(a,b)}  \}$ under some $\normal{\mu}{\sigma^2}$ has the following e-power:
\begin{equation}
    \lim_{n\to \infty} \frac{\log(B_n^{(a,b)})}{n} = \log \left( \Phi(\mu/\sigma)^{\Phi(\mu/\sigma)} \Phi(-\mu/\sigma)^{\Phi(-\mu/\sigma)} \right)  - \log(1/2).
\end{equation}
\end{proof}